\documentclass{amsart}
\usepackage{amsfonts,amscd,amsthm,amsgen,amsmath,amssymb}
\usepackage[all]{xy}
\usepackage[vcentermath]{youngtab}
\newtheorem{theorem}{Theorem}[section]
\newtheorem{lemma}[theorem]{Lemma}
\newtheorem{corollary}[theorem]{Corollary}
\newtheorem{proposition}[theorem]{Proposition}

\theoremstyle{definition}
\newtheorem{definition}[theorem]{Definition}

\theoremstyle{remark}
\newtheorem{remark}[theorem]{Remark}

\numberwithin{equation}{section}

\usepackage[pdftex]{hyperref}

\begin{document}

\title{Topological T-duality for torus bundles with monodromy}
\author{David Baraglia}

\address{Mathematical sciences institute, The Australian National University, Canberra ACT 0200, Australia}


\email{david.baraglia@anu.edu.au}

\thanks{This work is supported by the Australian Research Council Discovery Project DP110103745.}


\date{\today}


\begin{abstract}
We give a simplified definition of topological T-duality that applies to arbitrary torus bundles. The new definition does not involve Chern classes or spectral sequences, only gerbes and morphisms between them. All the familiar topological conditions for T-duals are shown to follow. We determine necessary and sufficient conditions for existence of a T-dual in the case of affine torus bundles. This is general enough to include all principal torus bundles as well as torus bundles with arbitrary monodromy representations. We show that isomorphisms in twisted cohomology, twisted K-theory and of Courant algebroids persist in this general setting. We also give an example where twisted K-theory groups can be computed by iterating T-duality.
\end{abstract}

\maketitle


\section{Introduction}

T-duality is a duality arising from string theory between spaces $X,\hat{X}$ which are torus fibrations equipped with gerbes $\mathcal{G},\hat{\mathcal{G}}$. From the string theory point of view the dual spaces $X,\hat{X}$ are target spaces for certain string theories and the duality is an equivalence between these two theories. The gerbes $\mathcal{G},\hat{\mathcal{G}}$ naturally arise in this setting in connection with holonomy over string worldsheets. Aside from the presence of gerbes the other characteristic feature of $T$-duality is that the spaces involved $X,\hat{X}$ are torus fibrations over a common base $M$. Roughly speaking the non-singular fibers of $X,\hat{X}$ are understood as being dual tori.\\

In the presence of gerbes one can define various twisted topological and geometric structures. On the topological side we consider twisted de Rham cohomology and twisted $K$-theory. On the differential geometry side gerbes appear in generalized geometry in the sense of Hitchin, where they are related to twists of the generalized tangent bundle \cite{hit}. From a purely mathematical point of view T-duality is a relation between pairs $(X,\mathcal{G})$ that entails isomorphisms of these twisted geometric and topological structures. The situation is in many ways similar to that of mirror symmetry. Indeed, if the Strominger-Yau-Zaslow conjecture \cite{syz} is taken seriously, then mirror symmetry is a (highly elaborate) kind of T-duality.\\

A local expression for T-duality given by the so called Buscher rules \cite{bus} gives a more precise expression of the duality between the fibers of $X$ and $\hat{X}$. However the Buscher rules apply only in local charts where both the torus bundle and the gerbe admit trivializations. This leaves the question of how to understand the global nature of T-duality. In \cite{bem} a global topological description of T-duality was formulated which was shown to be consistent with the Buscher rules \cite{bhm1}. This {\em topological T-duality} expresses a purely topological statement of T-duality that does not involve metrics or other geometric structures. The usual starting point for topological T-duality, for instance in \cite{bem},\cite{bhm1},\cite{bunksch},\cite{brs}, is a principal torus bundle $X \to M$ and a degree $3$ integral cohomology class $h$ (or perhaps more correctly a gerbe $\mathcal{G}$) on $X$. In this paper we describe an extension of topological T-duality to arbitrary torus bundles. Note that there are other directions in which topological T-duality can be generalized including non-commutative, even non-associative spaces \cite{mr1},\cite{mr2},\cite{bhm2} and stacks (orbispaces) \cite{bunksch2},\cite{bss}.\\

In \cite{bar} we showed how topological T-duality could be extended from principal circle bundles to general circle bundles. Many of the result we established could easily be extended to non-principal torus bundles of arbitrary rank. However we restricted attention to the circle bundle case, partly because of some techincal details that we could not resolve at the time. In the process of resolving these difficulties some new insights have come to light. We now have a different, greatly simplified definition of T-duality applicable to arbitrary torus bundles and a new approach to the proof of the existence of T-duals. In the existing mathematical literature on T-duality, T-duality is often defined in terms of an exchange of Chern classes and Dixmier-Douady classes. To even state this properly is complicated as it involves looking at the filtrations associated to the Leray-Serre spectral sequences. Our simpler definition does not involve Chern classes, only graded bundle gerbes and morphisms between them. In particular spectral sequences are no longer involved in the definition, though they are used extensively in the proofs that follow. We show that from our new definition the familiar exchange of Chern classes and Dixmier-Douady classes is actually a consequence of T-duality.\\

There are additional benefits to our approach to T-duality. In particular we have reason to suspect that our new definition can be readily adapted to the holomorphic setting, namely holomorphic torus fibrations and homolomorphic gerbes. This brings T-duality in line with existing work on Fourier-Mukai tranforms of twisted sheaves \cite{donpa},\cite{cal},\cite{bas}. Indeed from our new definition of T-duality it is clear that the Fourier-Mukai transform and its generalizations are deeply connected to T-duality. Such connections have been observed in \cite{horoz},\cite{hor},\cite{shar}. The holomorphic setting is likely to be useful in suggesting how to proceed in the case of singular torus fibrations.\\

We sketch the definition of topological T-duality. The various terms used here are defined in the paper. Let $\pi : X \to M$, $\hat{\pi} : \hat{X} \to M$ be rank $n$ torus bundles on $M$, $\mathcal{G},\hat{\mathcal{G}}$ graded gerbes on $X$,$\hat{X}$. Let $C = X \times_M \hat{X}$ be the fiber product and $p,\hat{p}$ the projections from $C$ to $X,\hat{X}$. We say that $(X,\mathcal{G}),(\hat{X},\hat{\mathcal{G}})$ are {\em topological T-duals} if
\begin{itemize}
\item{The gerbes $\mathcal{G},\hat{\mathcal{G}}$ are trivial along the fibers of $X,\hat{X}$.}
\item{There exists a stable isomorphism $\gamma : p^*(\mathcal{G}) \to \hat{p}^*(\hat{\mathcal{G}} \otimes \hat{\pi}^*(L(V)))$, where $V$ is the flat vector bundle $V = (R^1 \pi_* \mathbb{R})^*$ and $L(V)$ the lifting gerbe of $V$.}
\item{The stable isomorphism $\gamma$ satisfies an additional property, we call the {\em Poincar\'e property}.}
\end{itemize}
What we call the Poincar\'e property corresponds to Equation (2.7) in \cite{brs}. As explained in Section \ref{defpro}, this property roughly states that the isomorphism $\gamma$ locally looks like the Poincar\'e line bundle on the fibers of $C$. In particular it establishes a duality between the local systems $R^1 \pi_* \mathbb{Z}$, $R^1 \hat{\pi}_* \mathbb{Z}$, so that $X$, $\hat{X}$ have dual monodromy representations.\\

Aside from the new definition of T-duality our work advances topological T-duality by expanding the scope of applicability. Previous definitions of topogical T-duality have relied on principal torus bundles. As already stated, our definition applies perfectly well to arbitrary torus bundles. We prove the existence of T-duals in the class of {\em affine torus bundles}, that is torus bundles with structure group ${\rm Aff}(T^n) = {\rm GL}(n,\mathbb{Z}) \ltimes T^n$, the group of affine transformations of the torus $T^n$. This class not only includes all principal torus bundles, but allows for arbitrary monodromy representations. It is even known that every torus bundle of rank $\le 3$ admits an affine structure.

Our interest in T-duality for torus bundles with monodromy stems from a desire to develop T-duality for singular torus fibrations. As this is our prime motivation let us describe the argument in some detail. Given a torus fibration $\pi : X \to M$ with singularities, we want to construct a dual fibration $\hat{X} \to M$. Following in the footsteps of topological mirror symmetry \cite{gro}, the plan of attack is to first restrict attention to the non-singular fibers. This gives a locally trivial torus bundle $X' \to (M - \Delta)$, where $\Delta$ is the singular locus. We construct a topological T-dual $\hat{X}' \to (M - \Delta)$ for $X'$. Next we attempt to complete $\hat{X}'$ to a singular torus fibration $\hat{X} \to M$, by adding in dual singular fibers over $\Delta$. This is the approach taken in topological mirror symmetry and we expect that it can be extended to a more general setting of torus fibrations with gerbes. The point we would like to emphasize here is that one does not expect the torus bundle $X' \to (M - \Delta)$ to be principal. It is well known that there can be non-trivial monodromy around non-singular fibers. Indeed for sufficiently well-behaved singularities the monodromy contains a lot of information about the nature of the singularities. Incorporating monodromy into topological T-duality is therefore an essential step towards a establishing a singular version of T-duality.\\

The main conclusion of this paper is that topological T-duality adapts to torus bundles with monodromy without any serious difficulties. In addition there are some features of T-duality which only become apparent when one considers non-principal torus bundles, such as the need to use graded gerbes and to incorporate these grading structures into twisted cohomology and twisted $K$-theory.\\

We now describe the contents of this paper and the main results. Section \ref{prelim} covers the necessary background material on affine torus bundles (Section \ref{afftorbund}) and graded gerbes (Section \ref{ggrb}). Section \ref{ttdoatb} is the main section of the paper. We give our definition of T-duality (Definition \ref{tdual}) and establish the necessary and sufficient conditions for a T-dual to exist (Theorem \ref{exist}). To state the theorem we need to recall, as explained in Section \ref{ggrb} that stable isomorphisms classes of graded gerbes on $X$ correspond to elements of $H^1(X,\mathbb{Z}_2) \times H^2(X , \mathcal{C}_U)$, where $\mathcal{C}_U$ denotes the sheaf of ${\rm U}(1)$-valued continuous functions. For a graded gerbe $\mathcal{G}$, the corresponding class in $H^1(X,\mathbb{Z}_2) \times H^2(X , \mathcal{C}_U)$ is called the {\em (graded) Dixmier-Douady class}. Theorem \ref{exist} is as follows:
\begin{theorem}
Given a pair $(X,\mathcal{G})$ let $(\xi,h) \in H^1(X,\mathbb{Z}_2) \oplus H^2(X,\mathcal{C}_U )$ denote the Dixmier-Douady class of $\mathcal{G}$. Then $(X,\mathcal{G})$ is T-dualizable if and only if the following conditions hold:
\begin{itemize}
\item{$h$ lies in the image of the pullback $\pi^* : H^2(M , \pi_* (\mathcal{C}_U)  ) \to H^2(X , \mathcal{C}_U)$},
\item{$\xi$ lies in the image of the pullback $\pi^* : H^1(M,\mathbb{Z}_2) \to H^1(X,\mathbb{Z}_2)$.}
\end{itemize}
\end{theorem}

In Section \ref{dtd} we restrict attention to the the setting of smooth manifolds and smooth gerbes. Here with the aid of gerbe connections, reviewed in Section \ref{grbconn}, we are able to translate T-duality into a statement involving differential forms on $X$ and $\hat{X}$. Let $(X,\mathcal{G}),(\hat{X},\hat{\mathcal{G}})$ be T-dual pairs and suppose $A,\hat{A}$ are twisted connections on $X,\hat{X}$ with curvature $2$-forms $F,\hat{F}$ (these notions, which generalize connections in the principal case are introduced in Section \ref{sstdt}). Theorem \ref{tdtthm} states that there exists a $3$-form $H_3$ on the base and connections and curvings for $\mathcal{G},\hat{\mathcal{G}}$ so that the curvature $3$-forms $H,\hat{H}$ have the form
\begin{eqnarray*}
H &=& \pi^*(H_3) + ( A \buildrel \wedge \over , \pi^*(\hat{F})), \\
\hat{H} &=& \hat{\pi}^*(H_3) + ( \hat{A} \buildrel \wedge \over , \hat{\pi}^*(F)),
\end{eqnarray*}
where $\pi,\hat{\pi}$ are the projections from $X,\hat{X}$ to the base. The forms $A,F$ and $\hat{A},\hat{F}$ take values in dual local systems and $( \, , \, )$ denotes the dual pairing. These relations are the natural generalization of similar relations in \cite{bem},\cite{bhm1}, which in turn are derived from the Buscher rules. The above relations come close to an alternative characterization of T-duality, but are weaker due to a loss of torsion subgroups in passing to de Rham cohomology. Nevertheless the existence of triples $(A,\hat{A},H_3)$ satisfying these relations turns out to be quite useful, so we call such a triple a {\em T-duality triple}.

In Section \ref{stdtrans} we show that T-duality preserves twisted de Rham cohomology (Theorem \ref{ttdttwcoiso}), twisted $K$-theory (Theorem \ref{ttdttktiso}) and yields an isomorphism of Courant algebroids (Theorem \ref{cai}). In the case of twisted $K$-theory the isomorphism is essentially a $K$-theoretic twisted Fourier-Mukai transform and our definition of T-duality can understood as precisely the conditions needed for such a twisted Fourier-Mukai transform. This is clearly seen in the proof of \ref{ttdttktiso}. In the cases of twisted cohomology and Courant algebroids, the isomorphisms are essentially corollaries of the existence of T-duality triples.

In Section \ref{examps} we give some examples of T-duality with non-trivial monodromy and compute the twisted $K$-theory groups. The example of \ref{tdea} is particularly interesting because it shows how one can use repeated applications of T-duality to calculate twisted $K$-theory groups.


\section{Preliminaries}\label{prelim}


\subsection{Affine torus bundles}\label{afftorbund}
Let $T^n = \mathbb{R}^n / \mathbb{Z}^n$ be the standard $n$-torus. We define ${\rm Aff}(T^n)$ to be the semi-direct product ${\rm Aff}(T^n) = {\rm GL}(n,\mathbb{Z}) \ltimes T^n$ which acts on $T^n$ by affine transformations.

\begin{definition} An {\em affine torus bundle} is a torus bundle $X \to M$ with structure group ${\rm Aff}(T^n)$. Thus there is a principal ${\rm Aff}(T^n)$-bundle $P \to M$ such that $X$ is the associated bundle $X = P \times_{{\rm Aff}(T^n)} T^n$.
\end{definition}

Affine torus bundles are general enough to include all principal bundles as well as torus bundles with non-trivial monodromy. This makes them a good choice for applications to T-duality. In fact classifying affine torus bundles goes a long way towards a full classification of torus bundles for it is known that the inclusion ${\rm Aff}(T^n) \to {\rm Diff}(T^n)$ is a homotopy equivalence for $n \le 3$ (\cite{eaee} for $n=2$, \cite{hat} for $n=3$), so every torus bundle of rank $\le 3$ has an affine structure.\\

Let $P \to M$ be a principal ${\rm Aff}(T^n)$-bundle over $M$. Using the homomorphism ${\rm Aff}(T^n) \to {\rm GL}(n,\mathbb{Z})$ we obtain from $P$ a principal ${\rm GL}(n,\mathbb{Z})$-bundle. When $M$ is connected, principal ${\rm GL}(n,\mathbb{Z})$-bundles up to isomorphism correspond to conjugacy classes of homomorphisms $\rho : \pi_1(M) \to {\rm GL}(n,\mathbb{Z})$, that is to representations of $\pi_1(M)$ on $\mathbb{Z}^n$. We call the representation of $\pi_1(M)$ on $\mathbb{Z}^n$ associated to any principal ${\rm Aff}(T^n)$-bundle the {\em monodromy} of $P$. 

Instead of representations of the fundamental group, we can think of monodromy in term of the corresponding local system. If $\tilde{M} \to M$ is the universal cover of $M$ thought of as a principal $\pi_1(M)$-bundle and $\rho : \pi_1(M) \to {\rm GL}(n,\mathbb{Z})$ a representation then we have an associated bundle of groups $\tilde{M} \times_{\rho} \mathbb{Z}^n$, with fibers isomorphic to $\mathbb{Z}^n$. The sheaf $\Lambda$ of sections of $\tilde{M} \times_{\rho} \mathbb{Z}^n$ is then a local system with coefficient group $\mathbb{Z}^n$. If $\pi : X \to M$ is an affine torus bundle associated to a principal ${\rm Aff}(T^n)$-bundle $P$ then the local system $\Lambda$ has a more direct interpretation in terms of $X$. Consider the higher direct image sheaf $R^1 \pi_* \mathbb{Z}$, that is the sheaf associated to the presheaf which sends an open set $U \subseteq M$ to $H^1( \pi^{-1}(U), \mathbb{Z} )$. If $M$ is locally contractible then $R^1 \pi_* \mathbb{Z}$ is a local system with coefficient group $H^1(T^n , \mathbb{Z})$. One finds that $\Lambda$ is precisely the dual local system, that is $\Lambda = (R^1 \pi_* \mathbb{Z})^* = {\rm Hom}( R^1 \pi_* , \mathbb{Z} , \mathbb{Z})$. We will call $\Lambda$ the {\em monodromy local system} of $X$.\\

So far we have associated to any affine torus bundle $\pi : X \to M$ a local system $\Lambda$. The monodromy local system $\Lambda$ represents the linear part of the ${\rm Aff}(T^n)$-valued transition functions of $X$. In addition we can associate to $X$ a class $c \in H^2(M , \Lambda)$ which represents the translational part of the ${\rm Aff}(T^n)$-valued transitions functions. We call $c$ the {\em twisted Chern class} of $X$. We have seen that associated to $\pi : X \to M$ is a principal ${\rm GL}(n,\mathbb{Z})$-bundle which we denote by $G \to M$. Note that $G$ depends only on $\Lambda$, indeed $G$ can be thought of as the ${\rm GL}(n,\mathbb{Z})$-frame bundle of $\Lambda$. Since ${\rm GL}(n,\mathbb{Z})$ acts on the $n$-torus $T^n$ by group automorphisms, we get a bundle of groups $T^n_\Lambda = G \times_{{\rm GL}(n,\mathbb{Z})} T^n$ with fiber $T^n$. Let $\mathcal{C}(T^n_{\Lambda})$ denote the sheaf of continuous sections of $T^n_\Lambda$. Writing out the cocycle data for ${\rm Aff}(T^n)$-valued transition functions as done in \cite{bar}, one finds that the translational part determines a class $c' \in H^1(M , \mathcal{C}(T^n_\Lambda))$. 

Using the action of ${\rm GL}(n,\mathbb{Z})$ on $\mathbb{R}$ we similarly get a bundle of groups $\mathbb{R}^n_\Lambda$ and a short exact sequence
\begin{equation*}
0 \to \Lambda \to \mathcal{C}( \mathbb{R}^n_\Lambda ) \to \mathcal{C}( T^n_\Lambda ) \to 0.
\end{equation*}
The twisted Chern class of $X$ is defined to be the image of $c'$ under the coboundary $\delta : H^1(M , \mathcal{C}(T^n_\Lambda)) \to H^2(M , \Lambda)$. When $M$ is paracompact we can use a partition of unity to show that the coboundary $\delta$ is an isomorphism.

\begin{theorem}[\cite{bar}]\label{pab}
Let $M$ be locally contractible and paracompact. To every pair $(\Lambda , c)$, where $\Lambda$ is a $\mathbb{Z}^n$-valued local system on $M$ and $c \in H^2(M , \Lambda)$, there is an affine $T^n$-bundle $\pi : X \to M$ with monodromy local system $\Lambda$ and twisted Chern class $c$. Two pairs $(\Lambda_1,c_1)$,$(\Lambda_2,c_2)$ determine the same $T^n$-bundle (up to bundle isomorphisms covering the identity on $M$) if and only if there is an isomorphism $\phi : \Lambda_1 \to \Lambda_2$ of local systems which sends $c_1$ to $c_2$ under the induced homomorphism $\phi : H^2(M ,\Lambda_1) \to H^2(M,\Lambda_2)$.
\end{theorem}

\begin{remark}
Note that it is fairly straightforward to show that principal ${\rm Aff}(T^n)$-bundles up to principal bundle isomorphism correspond to pairs $(\Lambda,c)$ modulo the equivalence $(\Lambda_1,c_1) \simeq (\Lambda_2,c_2)$ as described. The non-trivial part of Theorem \ref{pab} is the fact that two principal ${\rm Aff}(T^n)$-bundles $P_1,P_2$ determine the same $T^n$-bundle if and only if $P_1,P_2$ are isomorphic as principal bundles. To put it another way, any two ${\rm Aff}(T^n)$-structures on a torus bundle are related by a bundle isomorphism.
\end{remark}

Let $\pi : X \to M$ be a rank $n$ affine torus bundle with monodromy local system $\Lambda$ and twisted Chern class $c \in H^2(M,\Lambda)$. Let $\{(E_r^{p,q},d_r)\}$ denote the Leray-Serre spectral sequence associated to $\pi : X \to M$, using cohomology with integral coefficients. Since the fibers of $X \to M$ are tori we find that the $E_2$ stage is given by $E_2^{p,q} = H^p(M , \wedge^q \Lambda^*)$. We need a description of the differential $d_2$. For $x \in \Lambda$ and $\alpha \in \wedge^q \Lambda^*$ let $i_x \alpha \in \wedge^{q-1}\Lambda^*$ denote the contraction of $x$ by $\alpha$. Combining with the cup product we get a natural contraction map $\smallsmile : H^r(M,\Lambda) \otimes H^s(M,\wedge^q \Lambda) \to H^{s+r}(M,\wedge^{q-1}\Lambda)$. To be more specific, take $M$ to be connected, let $\tilde{M}$ be the universal cover of $M$ and $\rho : \pi_1(M) \to {\rm GL}(n,\mathbb{Z})$ the monodromy representation. Let $\gamma \in \pi_1(M)$ act on the group $S^r(\tilde{M},\mathbb{Z}^n)$ of singular $r$-cochains on $\tilde{M}$ by $\varphi \mapsto \rho(\gamma)( \varphi \circ R_\gamma)$, with $R_\gamma : \tilde{M} \to \tilde{M}$ the right action on $\tilde{M}$. Then $H^*(M,\Lambda)$ is the cohomology of $S^*(\tilde{M},\mathbb{Z}^n)^{\pi_1(M)}$, the $\pi_1(M)$-invariant subcomplex of $S^*(\tilde{M},\mathbb{Z}^n)$. Similarly $H^*(M,\wedge^q \Lambda^*)$ is the cohomology of the $\pi_1(M)$-invariant subcomplex of $S^*(\tilde{M} , \wedge^q (\mathbb{Z})^*)$. For $\varphi \in S^r(\tilde{M},\mathbb{Z}), \psi \in S^s(\tilde{M},\mathbb{Z}), x \in \Lambda, \alpha \in \wedge^q \Lambda^*$ we use the convention
\begin{equation*}
( \varphi \otimes x ) \smallsmile ( \psi \otimes \alpha) = (-1)^s(\varphi \smallsmile \psi) \otimes ( i_x \alpha).
\end{equation*}
This determines the desired map $\smallsmile : H^r(M,\Lambda) \otimes H^s(M,\wedge^q \Lambda) \to H^{s+r}(M,\wedge^{q-1}\Lambda)$. With this convention in place, we have:
\begin{theorem}[\cite{bar}]\label{atblsd}
The differential $d_2 : H^p(M,\wedge^q \Lambda^*) \to H^{p+2}(M,\wedge^{q-1} \Lambda^*)$ is given by contraction with the twisted Chern class: $d_2(a) = c \smallsmile a$.
\end{theorem}


\subsection{Graded gerbes}\label{ggrb}

To understand T-duality it is important to realize that the duality is not between spaces alone, but rather spaces equipped with an additional structure we call the {\em $H$-flux}. In the physical interpretation $H$-flux is identified with the Neveu-Schwarz $B$-field of string theory. We shall represent $H$-flux in terms of {\em graded bundle gerbes}. In this section we review the details of the theory of bundle gerbes necessary for T-duality. The gerbes we define below possess a grading structure, which we have found to be essential for T-duality of non-oriented torus bundles. Gerbes with such a grading structure were found to be relevant to type II superstring theories \cite{dfm}. We note also that here we are only interested in capturing the topological aspect of the $B$-field. It should be possible to express a fully geometric notion of T-duality using gerbes with connections and curvings, along the same lines as done in \cite{kava} using differential cohomology. Actually we come close to proving such a statement in Proposition \ref{linkgerbeh}.\\

Most of the definitions and results in this this section can be obtained (after some minor modifications) from standard references for bundle gerbes such as \cite{mm},\cite{stev},\cite{bcmms}. See also \cite{fht} for the closely related notion of graded central extensions.\\

Let $X$ be a (paracompact) topological space. By a {\em quasi-cover} $X$, we mean a space $Y$ and a continuous map $f : Y \to X$ that admits local sections. That is, for each $x \in X$ there is an open neighborhood $x \in U \subseteq X$ and map $s : U \to X$ such that $f(s(u)) = u$ for all $u \in U$. Given a cover of $X$ by open subsets $\{ i_\alpha : U_\alpha \to X \}_{\alpha \in I}$ we find that the disjoint union $Y = \coprod_{\alpha \in I} U_\alpha$ is a quasi-cover, where $f : Y \to X$ is the map $f = \coprod_{\alpha \in I} i_\alpha$. We sometimes use $(Y,f)$ to denote a quasi-cover $f : Y \to X$ or simply say that $Y$ is a quasi-cover of $X$ with the understanding that part of the structure of $Y$ is a map $f : Y \to X$. 

Observe that if $f : Y \to X$ and $r : Z \to Y$ are quasi-covers then so is $f \circ r : Z \to X$. The quasi-covers of $X$ form a category as follows. Let $(Y,f),(Z,g)$ be quasi-covers of $X$. A morphism $r : (Z,g) \to (Y,f)$ is a map $r : Z \to Y$ such that $g = f \circ r$ and such that $r : Z \to Y$ admits local sections (i.e. $r$ is itself a quasi-cover). A {\em refinement} of a quasi-cover $(Y,f)$ is a quasi-cover $(Z,g)$ and morphism $r : (Z,g) \to (Y,f)$. Similarly if we are given quasi-covers $(Y,f),(Y',f')$, a {\em common refinement} is a quasi-cover $(Z,g)$ and morphisms $r : (Z,g) \to (Y,f)$, $r' : (Z,g) \to (Y',f')$. Observe that if $(Y,f),(Y',f)$ are quasi-covers then so is the fiber product $(Y \times_X Y' , f \times_X f')$, which is then a common refinement. In fact common refinements of $(Y,f),(Y',f')$ are in bijection with refinements of $(Y \times_X Y' , f \times_X f')$, or even more simply with maps $Z \to Y \times_X Y'$ admitting local sections.\\

If $(Y,f)$ is a quasi-cover of $X$ we write $Y^{[k]}$ for the $k$-fold fiber product. We have an associated topological groupoid with objects $Y$ and morphisms $Y^{[2]}$. Two pairs $(x,y),(y',z) \in Y^{[2]}$ are taken to be composable if and only if $y = y'$ and their composition is taken to be $(x,y)(y,z) = (x,z)$. A refinement $r : Z \to Y$ induces maps $r^{[k]} : Z^{[k]} \to Y^{[k]}$ as well as a functor between associated groupoids (in fact an equivalence of groupoids).

Let us introduce some notation. For $k = 0,1, \dots , (n-1)$, let $\partial_k : Y^{[n]} \to Y^{[n-1]}$ denote the map that omits the $(k+1)$-th factor. If $L \to Y^{[k]}$ is a line bundle or more generally a fiber bundle over $Y^{[k]}$ we write $L_{x_1,x_2,\dots , x_k}$ for the fiber of $L$ over $(x_1,x_2,\dots , x_k ) \in Y^{[k]}$.

\begin{definition}
A {\em bundle gerbe} \cite{mm} on $X$ consists of a quasi-cover $(Y,f)$ on $X$, a Hermitian line bundle $L \to Y^{[2]}$ on $Y^{[2]}$ and for each pair of composable pairs $(x,y),(y,z) \in Y^{[2]}$ an isomorphism $\theta_{x,y,z} : L_{x,y} \otimes L_{y,z} \to L_{x,z}$ of Hermitian vector spaces such that $\theta_{x,y,z}$ depends continuously on the triple $(x,y,z) \in Y^{[3]}$. In addition it is required that $\theta$ satisfies an associativity condition: for each quadruple $(x,y,z,w)$ there are two possible ways to use $\theta$ to go from $L_{x,y} \otimes L_{y,z} \otimes L_{z,w}$ to $L_{x,w}$ depending on which multiplication is performed first. These two ways are required to give the same isomorphism.

The bundle gerbe defined by the data $f : Y \to X$, $L \to Y^{[2]}$, $\theta$ will be denoted by $(Y,f,L,\theta)$ or simply $(L,\theta)$ if $f : Y \to X$ is understood. If $\mathcal{G} = (Y,f,L,\theta)$ then we say that $\mathcal{G}$ is defined with respect to $Y$.
\end{definition}

\begin{remark}
In the above definition continuity of $\theta_{x,y,z}$ on the triple $(x,y,z)$ means that $(x,y,z) \mapsto \theta_{x,y,z}$ defines a bundle isomorphism $\theta : \partial_2^* L \otimes \partial_0^* L \to \partial_1^* L$. In what follows we will see further occasions where a bundle map is most easily described in terms of the maps on individual fibers. Continuity refers to the requirement that the fiberwise maps collectively define a continuous bundle mapping.
\end{remark}

Adapting the notion of graded central extensions in \cite{fht} to the language of bundle gerbes, we arrive at the following:

\begin{definition}
A {\em graded bundle gerbe} is a bundle gerbe $(Y,f,L,\theta)$ together with an assignment of a (continuously varying) $\mathbb{Z}_2$-grading on $L$ such that $\theta$ has degree $0$. That is, for every point $(x,y) \in Y^{[2]}$ the line bundle $L_{x,y}$ is assigned a degree $\epsilon_{x,y} \in \mathbb{Z}_2$ that varies continuously with $(x,y)$ and such that for every triple $(x,y,z) \in Y^{[3]}$ we have $\epsilon_{x,y} + \epsilon_{y,z} = \epsilon_{x,z}$. Such a graded gerbe is denoted $(Y,f,L,\theta,\epsilon)$ or simply $(L,\theta,\epsilon)$.
\end{definition}

Let $f : Y \to X$ be a quasi-cover of $X$. Suppose $\mathcal{G} = (L,\theta,\epsilon)$, $\mathcal{G}' = (L',\theta',\epsilon')$ are graded gerbes defined with respect to $Y$. By a {\em strict isomorphism} we mean an isomorphism $\phi : L \to L'$ of graded Hermitian line bundles that respects the gerbe products, that is for each $(x,y,z) \in Y^{[3]}$ we have a commutative diagram
\begin{eqnarray*}\xymatrix{
L_{x,y} \otimes L_{y,z} \ar[rr]^-{\phi_{x,y} \otimes \phi_{y,z}} \ar[d]^{\theta_{x,y,z}} & & L'_{x,y} \otimes L'_{y,z} \ar[d]^{\theta'_{x,y,z}} \\
L_{x,z} \ar[rr]^{\phi_{x,z}} & & L'_{x,z}
}
\end{eqnarray*}
The adjective strict is used here to distinguish from the notion of {\em stable isomorphism} defined below.\\

Let $M \to Y$ be a graded Hermitian line bundle on $Y$. We let $\delta(M) \to Y^{[2]}$ be the line bundle $\delta(M) = \partial_1^*(M) \otimes \partial_0^*(M^*)$ defined on $Y^{[2]}$. That is, for each $(x,y) \in Y^{[2]}$ we have $\delta(M)_{x,y} = M_x \otimes M_y^*$. We claim that $\delta(M)$ admits a natural gerbe product $\theta^M$, indeed we define $\theta^M_{x,y,z} : \delta(M)_{x,y} \otimes \delta(M)_{y,z} \to \delta(M)_{x,z}$ to be the map
\begin{equation*}
M_x \otimes M_y^* \otimes M_y \otimes M_z^* \to M_x \otimes M_z^*,
\end{equation*}
obtained by pairing $M_y^*$ and $M_y$. One checks that $\theta^M$ satisfies the associativity condition so that $(\delta(M),\theta^M)$ is a gerbe defined with respect to $Y$. A gerbe of this form is called a {\em trivializable gerbe}.\\

Let $\mathcal{G} = (L,\theta,\epsilon)$, $\mathcal{G}' = (L',\theta',\epsilon')$ be graded gerbes defined with respect to a quasi-cover $Y$ of $X$. We will define a gerbe $\mathcal{G} \otimes \mathcal{G}'$, the {\em product} of $\mathcal{G}$ and $\mathcal{G}'$. The underlying line bundle for $\mathcal{G} \otimes \mathcal{G}'$ is simply the tensor product $L \otimes L'$. The gerbe product $(L \otimes L')_{x,y} \otimes (L \otimes L')_{y,z} \to (L \otimes L')_{x,z}$ is given by the composition
\begin{equation*}\xymatrix{
L_{x,y} \otimes L'_{x,y} \otimes L_{y,z} \otimes L'_{y,z} \ar[r]^-{ e S} & L_{x,y} \otimes L_{y,z} \otimes L'_{x,y} \otimes L'_{y,z} \ar[rr]^-{\theta_{x,y,z} \otimes \theta'_{x,y,z}} & & L_{x,z} \otimes L'_{x,z}
}
\end{equation*}
where $S : L'_{x,y} \otimes L_{y,z} \to L_{y,z} \otimes L'_{x,y}$ is the isomorphism that swaps factors and $e = \pm 1$ is a sign factor which accounts for the fact we are using {\em graded line bundles}. Explicitly $e = (-1)^{\epsilon'_{x,y} \epsilon_{y,z}}$ so that $e = -1$ only when the two factors being swapped are odd. One checks easily that this satisfies the associativity condition and so defines a gerbe.\\

Let $\mathcal{G} = (Y,f,L,\theta,\epsilon)$ be a graded gerbe defined with respect to a quasi-cover $(Y,f)$ of $X$. If $r : (Z,g) \to (Y,f)$ is a refinement of $Y$, then we define a pullback gerbe $r^*\mathcal{G}$ defined with respect to $Z$. The underlying line bundle is the pullback $r^*(L)$ and the gerbe product is similarly a pullback, that is if $(x,y,z) \in Z^{[3]}$ then $(r(x),r(y),r(z)) \in Y^{[3]}$ and we get $\theta_{r(x),r(y),r(z)} : r^*(L)_{x,y} \otimes r^*(L)_{y,z} \to r^*(L)_{x,z}$. Lastly the grading is given by the pullback $r^*(\epsilon)$.

We can also introduce a notion of pullback that changes the base. Let $\phi : X' \to X$ be a map between spaces and $f : Y \to X$ a quasi-cover. Then $\phi^*(Y) = Y \times_{X} X'$ is a quasi-cover of $X'$ and there is a map $r : f^*(Y) \to Y$ covering $\phi$. Given a gerbe $\mathcal{G}$ defined with respect to $Y$ we can now easily define a pull back $r^*(\mathcal{G})$ which is a gerbe defined with respect to $\phi^*(Y)$.

\begin{definition}[\cite{stev},\cite{bcmms}]\label{defstabiso}
Let $\mathcal{G} = (Y,f,L,\theta,\epsilon)$, $\mathcal{G}' = (Y',f',L',\theta',\epsilon')$ be graded gerbes on $X$. A {\em stable isomorphism} $\mathcal{G} \to \mathcal{G}'$ consists of a common refinement $r : (Z,g) \to (Y,f)$, $r' : (Z,g) \to (Y',f')$, a graded line bundle $M \to Z$ and a strict isomorphism of graded gerbes $\phi : r^*(L) \otimes \delta(M) \to r'^*(L')$. Two graded gerbes $\mathcal{G},\mathcal{G}'$ are said to be {\em stably isomorphic} if there exists a stable isomorphism $\mathcal{G} \to \mathcal{G}'$.
\end{definition}

Our definition of stable isomorphism is actually slightly different from \cite{stev},\cite{bcmms}, but in a mild way that becomes irrelevant when we pass to the category $GrGrb(X)$ later in this Section.\\

Observe that given a gerbe $\mathcal{G}$ defined with respect to some quasi-cover $Y$ of $X$ and a refinement $r : Z \to Y$, the pullback $r^*(\mathcal{G})$ is stably isomorphic to $\mathcal{G}$. Thus if we are interested in the classification of gerbes up to stable isomorphism we may choose refinements as desired. In particular, since $f : Y \to X$ admits local sections we may choose an open cover $\{ U_\alpha \}_{\alpha \in I}$ of $X$ and local sections $s_\alpha : U_\alpha \to Y$. Setting $Z = \coprod_{\alpha \in I} U_\alpha$ we have a refinement $r : Z \to Y$ given by $r = \coprod_{\alpha \in I} s_\alpha$. Therefore in the classification of graded gerbes up to stable isomorphism it suffices to consider only those quasi-covers arising from open cover of $X$. If $Y = \coprod_{\alpha} U_\alpha$ then one sees that $Y^{[2]} = \coprod_{\alpha , \beta} U_{\alpha \beta}$, where as usual $U_{\alpha \beta} = U_\alpha \cap U_\beta$. A graded gerbe $\mathcal{G}$ defined with respect to $Y$ is then a collection of graded line bundles $L_{\alpha \beta} \to U_{\alpha \beta}$ together with an associative multiplication $\theta_{\alpha \beta \gamma} : L_{\alpha \beta} \otimes L_{\beta \gamma} \to L_{\alpha \gamma}$. Note that the grading is a $\mathbb{Z}_2$-valued cocycle $\epsilon_{\alpha \beta} : U_{\alpha \beta} \to \mathbb{Z}_2$.

By refining the cover if necessary we may assume there exists sections $\sigma_{\alpha \beta} : U_{\alpha \beta} \to L_{\alpha \beta}$ of unit norm. Define $g_{\alpha \beta \gamma} : U_{\alpha \beta \gamma} \to {\rm U}(1)$ by the relation $\theta_{\alpha \beta \gamma}((\sigma_{\alpha \beta} \otimes \sigma_{\beta \gamma}) = g_{\alpha \beta \gamma} \sigma_{\alpha \gamma}$. Then $\{ g_{\alpha \beta \gamma} \}$ is a \v{C}ech $2$-cocycle with values in the sheaf $\mathcal{C}({\rm U}(1))$. One checks easily that the \v{C}ech cohomology classes associated $\{ \epsilon_{\alpha \beta} \}$, $\{ g_{\alpha \beta \gamma} \}$ are independent of the choice of refinements and local sections, so that we get a well defined pair of classes $[\mathcal{G}] = (\xi , h) \in H^1(X,\mathbb{Z}_2) \times H^2(X,\mathcal{C}({\rm U}(1))) \simeq H^1(X,\mathbb{Z}_2) \times H^3(X , \mathbb{Z})$. The pair $(\xi,h)$ will be called the {\em graded Dixmier-Douady class} of $\mathcal{G}$. The class $\xi \in H^1(X,\mathbb{Z}_2)$ will be called the {\em grading class} of $\mathcal{G}$, while the class $h \in H^3(X,\mathbb{Z})$ the {\em (ungraded) Dixmier-Douady class} of $\mathcal{G}$.

Reversing the construction of the graded Dixmier-Douady class it is clear that any pair $(\xi,h) \in H^1(X,\mathbb{Z}_2) \times H^3(X,\mathbb{Z})$ arises from a graded gerbe. A little more work shows that two gerbes have the same Dixmier-Douady class if and only if they are stably isomorphic. We conclude that stable isomorphism classes of graded gerbes are in bijection with $H^1(X,\mathbb{Z}_2) \times H^3(X,\mathbb{Z})$.\\

The multiplication operation on graded gerbes descends to isomorphism classes and gives $H^1(X,\mathbb{Z}_2) \times H^3(X,\mathbb{Z})$ an abelian group structure. Note that it is not the product structure but rather a (generally) non-trivial extension of $H^1(X,\mathbb{Z}_2)$ by $H^3(X,\mathbb{Z})$. Explicitly the product is as follows \cite{fht}
\begin{equation*}
( a , h ) ( b , k ) = (a + b , h + k + \beta(a \smallsmile b))
\end{equation*}
where $\beta : H^2(X,\mathbb{Z}_2) \to H^3(X,\mathbb{Z})$ is the Bockstein homomorphism.\\

For the purposes of T-duality and twisted $K$-theory, the data defining a stable isomorphism carries too much information. The main point is that we do not want to know about the choice of common refinement $Z$ used. Our next step therefore will be to partition stable isomorphisms into suitable equivalence classes.\\

Let $\mathcal{G} = (Y,f,L,\theta,\epsilon)$, $\mathcal{G}' = (Y',f',L',\theta',\epsilon')$ be stably isomorphic graded gerbes on $X$. Consider two stable isomorphisms. Thus we have two common refinements $r : Z \to Y$, $r' : Z \to Y'$, $s : W \to Y$, $s' : W \to Y'$, two graded line bundles $M \to Z$, $N \to W$ and two strict isomorphisms $\phi : r^*(L) \otimes \delta(M) \to r'^*(L')$ and $\psi : s^*(L) \otimes \delta(N) \to s'^*(L')$. To compare the two stable isomorphisms, choose a common refinement of common refinements, that is let $V$ be a quasi-cover which is a refinement of $Z$ and $W$, $a : V \to Z$, $b : V \to W$ and such that $r \circ a = s \circ b$, $r' \circ a = s' \circ b$. Such refinements are possible to find because the maps $Z \to Y \times_X Y'$ and $W \to Y \times_X Y'$ admit local sections. Indeed we could take $V$ to be $\{ (z,w) \in Z \times W \, | \, r(z) = s(w), \; r'(z) = s'(w)\}$. Let $c : V \to Y$ be the composition $c = r \circ a = s \circ b$ and let $c' : V \to Y'$ similarly be given by $c' = r' \circ a = s' \circ b$. We then have stable isomorphisms $a^*(\phi) : c^*(L) \otimes a^*(\delta(M)) \to c'^*(L')$ and $b^*(\psi) : c^*(L) \otimes b^*(\delta(N)) \to c'^*(L')$. Combining the two isomorphisms we get a strict isomorphism $a^*(\delta(M)) \to b^*(\delta(N))$. Alternatively, letting $D \to V$ be the graded line bundle $D = a^*(M) \otimes b^*(N^*)$ we have a strict isomorphism $\mu : \delta(D) \simeq 1$, where $1$ denotes the trivial gerbe with respect to $V$, that is the gerbe consisting of the trivial line bundle $\mathbb{C} \to V$, and gerbe product simply given by multiplication of complex numbers. From the definition of strict isomorphism this means that for every pair $(x,y) \in V^{[2]}$ there is an isomorphism $\mu_{x,y} : D_y \to D_x$ (continuous in $x,y$) and satisfying the cocycle condition $\mu_{x,y} \circ \mu_{y,z} = \mu_{x,z}$. Thus $\mu$ is a descent isomorphism \cite{bry} and there is a uniquely determined line bundle $E \to X$ such that if $t : V \to X$ is the map from the quasi-cover $V$ to $X$ then $D \simeq t^*(E)$. It is not hard to see that $E$ is independent of the choice common refinement $V$.

From the above discussion we conclude that for any two stable isomorphisms $\alpha,\beta : \mathcal{G} \to \mathcal{G}'$, there is a uniquely determined graded line bundle $E \to X$ over $X$. We introduce an equivalence relation $\sim$ on stable isomorphisms, namely $\alpha \sim \beta$ if and only if the corresponding graded line bundle is trivial (i.e. the grading takes the value $0$ and the line bundle is trivial). For any two graded gerbes $\mathcal{G},\mathcal{G}'$, we let ${\rm Hom}(\mathcal{G},\mathcal{G}')$ denote the set of equivalence classes of stable isomorphisms $\mathcal{G} \to \mathcal{G}'$ (in particular ${\rm Hom}(\mathcal{G},\mathcal{G}')$ is empty if $\mathcal{G},\mathcal{G}'$ are not stably isomorphic). It is straightforward to see that equivalence classes of stable isomorphisms can be composed and that the composition defines a category $GrGrb(X)$ of graded gerbes on $X$. That is the objects of $GrGrb(X)$ are graded gerbes on $X$ and the morphisms are equivalence classes of stable isomorphisms. Let us observe some useful properties of $GrGrb(X)$:
\begin{itemize}
\item{The only morphisms in $GrGrb(X)$ are isomorphisms, so $GrGrb(X)$ is a groupoid.}
\item{$GrGrb(X)$ is contravariant with respect to $X$ in the sense that for any map $\phi : Y \to X$ there is a pullback functor $\phi^* : GrGrb(X) \to GrGrb(Y)$ and if $\psi : Z \to Y$ is another map then $(\phi \circ \psi)^* = \psi^* \circ \phi^*$.}
\item{If $\mathcal{G},\mathcal{H} \in GrGrb(X)$ are isomorphic then ${\rm Hom}(\mathcal{G},\mathcal{H})$ is naturally a torsor for the group $H^0(X,\mathbb{Z}_2) \oplus H^2(X,\mathbb{Z})$ of graded Hermitian line bundles on $X$. In particular ${\rm Hom}(\mathcal{G},\mathcal{G})$ is canonically the group of graded Hermitian line bundles on $X$.}
\item{Think of ${\rm Hom}(\mathcal{G},\mathcal{G}')$ and ${\rm Hom}(\mathcal{G}',\mathcal{G}'')$ as torsors for the group of graded line bundles. Then if $\alpha \in {\rm Hom}(\mathcal{G},\mathcal{G}')$, $\beta \in {\rm Hom}(\mathcal{G}',\mathcal{G}'')$ and $L,M$ graded line bundles, we have $(\beta + M) \circ (\alpha + L) = (\beta \circ \alpha) + (L \otimes M)$.}
\end{itemize}

For any orthogonal vector bundle $V \to X$ on $X$ we can define a graded gerbe $L(V)$ called the {\em lifting gerbe of $V$} \cite{mm},\cite{stev}, which captures the failure of $V$ to admit a ${\rm Spin}^c$-structure. Let $P \to X$ be the orthogonal frame bundle of $V$. There is a natural map $\phi : P^{[2]} \to {\rm O}(n)$, where $n$ is the rank of $V$. The map $\phi$ is defined by letting $\phi( pg , p) = g$ for any $p \in P$ and $g \in {\rm O}(n)$. We can think of ${\rm Pin}^c(n) \to {\rm O}(n)$ as a principal circle bundle over ${\rm O}(n)$. Pulling back the associated Hermitian line bundle by $\phi$ yields a line bundle $L \to P^{[2]}$ which will define the lifting gerbe. The grading on $L$ is given by the composition $P^{[2]} \to {\rm O}(n) \to \mathbb{Z}_2$ where the second map is the determinant. The multiplication on $L$ is induced by the group composition on ${\rm Pin}^c(n)$. One checks that this defines a graded gerbe $L(V)$, the lifting gerbe of $V$. The stable isomorphism class of $L(V)$ is given by $|L(V)| = (w_1(V) , W_3(V)) \in H^1(X,\mathbb{Z}_2) \oplus H^3(X,\mathbb{Z})$, where $w_1(V)$ is the first Stiefel-Whitney class of $V$ and $W_3(V)$ is the third integral Stiefel-Whitney class of $V$. Observe that $w_1(V),W_3(V)$ are indeed the obstructions to a ${\rm Spin}^c(n)$-structure on $V$, so $L(V)$ is trivializable if and only if $V$ admits a ${\rm Spin}^c(n)$-structure. In fact the different trivializations of $L(V)$ can be identified with the different possible ${\rm Spin}^c(n)$-structures on $V$.

If $V$ is a rank $n$ vector bundle without metric we can still define a lifting gerbe $L(V,h)$ by first choosing a metric $h$ on $V$. Although $L(V,h)$ depends on $h$ we have as a consequence of the contractibility of the space of metrics on $V$, that for any two metrics $h_1,h_2$ there exists a canonical isomorphism $L(V,h_1) \to L(V,h_2)$ as elements of $GrGrb(X)$. In this sense the lifting gerbe of $V$ is essentially unique and we continue to denote it by $L(V)$.


\section{Topological T-duality of affine torus bundles}\label{ttdoatb}

In this section we wish to consider T-duality in a purely topological setting, thus the spaces involved need not be smooth manifolds. For technical reasons we assume our spaces to be locally contractible and have the property that every subspace is paracompact. This is true for instance of locally finite CW complexes, since every CW complex is locally compact \cite{hat1} and every locally finite CW complex is metrizable \cite{frpi}.

\subsection{Definition and properties}\label{defpro}

Let $\pi : X \to M$, $\hat{\pi} : \hat{X} \to M$ be two rank $n$ affine torus bundles over a common base $M$. We let $C = X \times_M \hat{X}$ be the fiber product and call $C$ the {\em correspondence space}. We then have naturally defined maps $p,\hat{p},q$ forming a commutative diagram
\begin{equation*}\xymatrix{
& C \ar[dl]_p \ar[dr]^{\hat{p}} \ar[dd]^q & \\
X \ar[dr]^\pi & & \hat{X} \ar[dl]_{\hat{\pi}} \\
& M &
}
\end{equation*}

Given a point $m \in M$ we let $T_m = \pi^{-1}(m)$, $\hat{T}_m = \hat{\pi}^{-1}(m)$ denote the fibers of $X,\hat{X}$ over $m$. In addition it follows that $q^{-1}(m) = T_m \times \hat{T}_m$. 

Let $\Lambda = (R^1\pi_* \mathbb{Z})^* = Hom( R^1 \pi_* \mathbb{Z} , \mathbb{Z})$, $\hat{\Lambda} = (R^1 \hat{\pi}_* \mathbb{Z})^* = Hom(R^1 \hat{\pi}_* \mathbb{Z} , \mathbb{Z})$ be the local systems associated to $\pi,\hat{\pi}$ by taking degree $1$ fiber cohomology and dualizing. So $\Lambda,\hat{\Lambda}$ represent fiber homology. Then $\Lambda \otimes \mathbb{R}$ is a local system with coefficient group $\mathbb{R}^n$. Therefore there exists a rank $n$ vector bundle $V \to M$ with flat connection such that $\Lambda \otimes \mathbb{R}$ is the sheaf of constant sections of $V$. We call $V$ the {\em vertical bundle} of $\pi : X \to M$. In the setting of smooth manifolds we have that $\pi^*(V)$ is the vertical tangent bundle of $X \to M$, that is $\pi^*(V) = \ker (\pi_* : TX \to TM)$ with flat connection induced by the affine structure on the fibers. Similarly let $\hat{V}$ denote the vertical bundle of $\hat{\pi} : \hat{X} \to M$.

\begin{definition}\label{tdual}
Let $\pi : X \to M$, $\hat{\pi} : \hat{X} \to M$ be rank $n$ affine torus bundles on $M$, $\mathcal{G},\hat{\mathcal{G}}$ graded gerbes on $X,\hat{X}$ respectively. We say that the pairs $(X,\mathcal{G})$ and $(\hat{X},\hat{\mathcal{G}})$ are {\em (topologically) T-dual} if the following conditions hold:
\begin{itemize}
\item[(T1)]{For all $m \in M$ the restrictions $\mathcal{G}|_{T_m}$, $\hat{\mathcal{G}}|_{\hat{T}_m}$ are trivial.}
\item[(T2)]{There exists an isomorphism $\gamma : p^*(\mathcal{G}) \to \hat{p}^*(\hat{\mathcal{G}}) \otimes q^*(L(V))$, where $L(V)$ is the lifting gerbe of $V$.}
\item[(T3)]{The isomorphism $\gamma$ can be chosen to satisfy the Poincar\'e property, defined below.}
\end{itemize}
\end{definition}

Note that the Poincar\'e property is used by Bunke, Rumpf and Schick in their definition of topological T-duality \cite[Equation (2.7)]{brs}.

\begin{remark}
We will see shortly (Proposition \ref{dualmono}) that if $(X,\mathcal{G}),(\hat{X},\mathcal{G})$ are T-dual then there is an induced isomorphism $\hat{V} \simeq V^*$ of flat bundles. It follows that there is a canonical isomorphism $L(\hat{V}) \simeq L(V)^*$ and from this one deduces that our definition of T-duality is symmetric in $(X,\mathcal{G}),(\hat{X},\mathcal{G})$.
\end{remark}

\begin{remark}
Nowhere in the definition of topological T-duality do we use the affine structures on $X,\hat{X}$ (this includes the Poincar\'e property below). Thus our definition of topological T-duality actually applies to {\em arbitrary torus bundles}. However, we have restricted attention to the affine case so that we can prove the existence of T-duals.
\end{remark}

To explain the Poincar\'e property we assume axioms (T1),(T2) of Definition \ref{tdual} hold. Choose a point $m \in M$. Then by (T1) there exist trivializations $\tau_m : 1 \to \mathcal{G}|_{T_m}$ and $\hat{\tau}_m : 1 \to \hat{\mathcal{G}}|_{\hat{T}_m}$, where $1$ denotes the trivial gerbe. Certainly there is also a trivialization $\psi_m : 1 \to L(V)|_{\{m\}}$. On $q^{-1}(m) = T_m \times \hat{T}_m$ we have two isomorphisms $p^*( \mathcal{G}|_{T_m} ) \to \hat{p}^*( \hat{\mathcal{G}} |_{\hat{T}_m})$. Indeed the first is given by the composition
\begin{equation*}\xymatrix{
p^*(\mathcal{G}|_{T_m}) \ar[rr]^-{\gamma|_{T_m \times \hat{T}_m}} & & \hat{p}^*(\hat{\mathcal{G}}|_{\hat{T}_m} ) \otimes q^*(L(V)|_{\{m\}}) \ar[rr]^-{id \otimes q^*(\psi_m^{-1})} & & \hat{p}^*(\hat{\mathcal{G}}|_{\hat{T}_m} )
}
\end{equation*}
and the second by the composition
\begin{equation*}\xymatrix{
p^*(\mathcal{G}|_{T_m}) \ar[rr]^-{p^*(\tau_m^{-1})} & & 1 \ar[rr]^-{\hat{p}^*(\hat{\tau}_m)} & & \hat{p}^*(\hat{\mathcal{G}}|_{\hat{T}_m} ).
}
\end{equation*}
The difference between the first and second isomorphism determines a class $d \in H^0(T_m \times \hat{T}_m , \mathbb{Z}_2) \times H^2(T_m \times \hat{T}_m , \mathbb{Z})$ which is the isomorphism class of a graded line bundle on $T_m \times \hat{T}_m$. However the class $d$ is not unique because we chose trivializations $\tau_m,\hat{\tau}_m,\psi_m$. Changing these trivializations we note that the changes in $d$ correspond to pulling back graded line bundles from $T_m$,$\hat{T}_m$ or $\{m\}$ to $T_m \times \hat{T}_m$. Factoring out this ambiguity and using the K\"unneth theorem we have for each $m \in M$ a well defined class $\delta_m \in H^1(T_m , \mathbb{Z}) \otimes H^1(\hat{T}_m , \mathbb{Z})$. Alternatively we can rewrite this as $\delta_m \in Hom( H_1(T_m , \mathbb{Z}) , H^1(\hat{T}_m , \mathbb{Z}))$. We say that $\gamma$ satisfies the {\em Poincar\'e property at $m \in M$} if $\delta_m \in Hom( H_1(T_m , \mathbb{Z}) , H^1(\hat{T}_m , \mathbb{Z}))$ is an isomorphism. If $\gamma$ satisfies satisfies the Poincar\'e property at all points $m \in M$ we simply say that $\gamma$ satisfies the Poincar\'e property.

\begin{remark}
The groups $H^1(T_m , \mathbb{Z})$, $H^1(\hat{T}_m , \mathbb{Z})$ can be naturally identified with the stalks at $m$ of the local systems $\Lambda^* = R^1\pi_*\mathbb{Z}$, $\hat{\Lambda}^* = R^1\hat{\pi}_*\mathbb{Z}$. Likewise the corresponding homology groups $H_1(T_m,\mathbb{Z}),H_1(\hat{T}_m,\mathbb{Z})$ are the stalks of the dual local systems $\Lambda,\hat{\Lambda}$. It is straightforward to see that the assignment $m \mapsto \delta_m$ is actually a global section of $Hom(\Lambda , \hat{\Lambda}^*)$. In particular if $M$ is connected then if $\gamma$ satisfies the Poincar\'e property at a single point $m \in M$, then $\gamma$ satisfies the Poincar\'e property.
\end{remark}

As a consequence of the above remark we immediately have:
\begin{proposition}\label{dualmono}
Let $(X,\mathcal{G}),(\hat{X},\hat{\mathcal{G}})$ be T-dual. There is an isomorphism of local systems $\Lambda^* \simeq \hat{\Lambda}$. Thus the monodromy representations associated to $X,\hat{X}$ are dual and the vertical bundles $V,\hat{V}$ are dual as flat vector bundles.
\end{proposition}


\subsection{Proof of existence}\label{poe}

Given a pair $(X,\mathcal{G})$ consisting of a rank $n$ affine torus bundle $\pi : X \to M$ and a graded gerbe $\mathcal{G}$ on $X$ it is natural to ask whether $(X,\mathcal{G})$ admits a T-dual. We say that $(X,\mathcal{G})$ is {\em T-dualizable} if $(X,\mathcal{G})$ admits a T-dual $(\hat{X},\hat{\mathcal{G}})$, where $\hat{X}$ is an affine torus bundle on $M$ and $\hat{\mathcal{G}}$ a graded gerbe on $\hat{X}$. It is clear that not every pair $(X,\mathcal{G})$ can have a T-dual in the sense of Definition \ref{tdual}. Indeed an obvious requirement is that $\mathcal{G}$ is trivial along the fibers of $X \to M$. However this is generally not a sufficient condition. In this section we give the precise conditions under which a T-dual exists.\\

We will need to make use of multiple instances of the Leray spectral sequence. Let us fix some notation and at the same time recall some general features of the spectral sequence. Given spaces $Z,W$, a map $f : Z \to W$ and a sheaf $\mathcal{F}$ on $Z$ we let $(E_r^{p,q}(f,\mathcal{F}) , d_r)$ denote the associated Leray spectral sequence. Recall that to the data $Z,W,f,\mathcal{F}$ there is a canonically defined filtration
\begin{equation*}
0 = F^{n+1,n}(f,\mathcal{F}) \subseteq F^{n,n}(f,\mathcal{F}) \subseteq \cdots \subseteq F^{0,n}(f,\mathcal{F}) = H^n(Z,\mathcal{F})
\end{equation*}
such that $E_\infty^{p,q}(f,\mathcal{F}) \simeq F^{p,p+q}(f,\mathcal{F})/F^{p+1,p+q}(f,\mathcal{F})$. The spectral sequence is canonically defined starting from the $E_2$ term, with $E_2^{p,q}(f,\mathcal{F}) = H^p(W , R^q f_* \mathcal{F})$. In addition the natural map $E_2^{p,0}(f,\mathcal{F}) \to E_\infty^{p,0}(f,\mathcal{F}) = F^{p,p}(f,\mathcal{F}) \to H^p(Z,\mathcal{F})$ is given by the natural pullback map $H^p(W,f_*\mathcal{F}) \to H^p(Z,\mathcal{F})$, while the natural map $H^q(Z,\mathcal{F}) \to E_\infty^{0,q}(f,\mathcal{F}) \to E_2^{0,q}(f,\mathcal{F}) = H^0(W,R^q f_* \mathcal{F})$ is the natural map sending a class in $H^q(Z,\mathcal{F})$ to the corresponding global section of $R^q f_* \mathcal{F}$.\\

Let $[\mathcal{G}] = (\xi , h ) \in H^1(X,\mathbb{Z}_2) \oplus H^3(X,\mathbb{Z})$ be the graded Dixmier-Douady class of $\mathcal{G}$. In particular $h \in H^3(X,\mathbb{Z})$ is the Dixmier-Douady class of the underlying ungraded gerbe. Actually we will find it is useful to instead think of $h$ as an element $h \in H^2(X , \mathcal{C}_U )$, where we have introduced the notation $\mathcal{C}_U$ for the sheaf $\mathcal{C}({\rm U}(1))$ of continuous sections valued in ${\rm U}(1)$.
\begin{theorem}\label{exist}
Given a pair $(X,\mathcal{G})$ let $(\xi,h) \in H^1(X,\mathbb{Z}_2) \oplus H^2(X,\mathcal{C}_U )$ denote the Dixmier-Douady class of $\mathcal{G}$. Then $(X,\mathcal{G})$ is T-dualizable if and only if the following conditions hold:
\begin{itemize}
\item{$h$ lies in the image of the pullback $\pi^* : H^2(M , \pi_* (\mathcal{C}_U)  ) \to H^2(X , \mathcal{C}_U)$},
\item{$\xi$ lies in the image of the pullback $\pi^* : H^1(M,\mathbb{Z}_2) \to H^1(X,\mathbb{Z}_2)$.}
\end{itemize}
\end{theorem}

\begin{remark}
Alternatively the Leray-Serre spectral sequence in integral cohomology for $\pi : X \to M$ determines a filtration $F^{3,3}(\pi,\mathbb{Z}) \subseteq F^{2,3}(\pi,\mathbb{Z}) \subseteq F^{1,3}(\pi,\mathbb{Z}) \subseteq F^{0,3}(\pi,\mathbb{Z}) = H^3(X,\mathbb{Z})$. We have that $(X,\mathcal{G})$ is T-dualizable if and only if $h$ thought of as an element of $H^3(X,\mathbb{Z})$, lies in the subspace $F^{2,3}(\pi,\mathbb{Z})$ and $\xi$ is a pullback of a class in $H^1(M,\mathbb{Z}_2)$. That this is an equivalent characterization will emerge from the proof of Theorem \ref{exist} below.
\end{remark}

Before we can get to the proof of Theorem \ref{exist} we need to go over some technical results. Let $q : C \to M$ be an affine torus bundle over $M$. Later $C$ will be a correspondence space, but we need not assume this for the moment. Consider the differential $d_2 : E_2^{0,1}(q,\mathcal{C}_U) \to E_2^{2,0}(q,\mathcal{C}_U)$ in the Leray spectral sequence. This is a homomorphism $d_2 : H^0(M , R^1 q_* \mathcal{C}_U) \to H^2(M , q_* \mathcal{C}_U)$. An element $l \in H^0(M , R^1 q_* \mathcal{C}_U)$ is represented by a collection $\{ L_i\}_{i \in I}$ of (unitary) line bundles $L_i \to q^{-1}(U_i)$ for some open cover $\{ U_i \}_{i \in I}$ of $M$ with the property that $L_i |_{q^{-1}(U_i \cap U_j)} \simeq L_j |_{q^{-1}(U_i \cap U_j)}$ for all non-empty double intersections. From the properties of the spectral sequence it is clear that $d_2(l)$ must represent the obstruction to finding a line bundle $L \to M$ such that after passing to a suitable refinement $\{ V_j \}_{j \in J}$, $r : J \to I$ we have $L|_{ q^{-1}(V_j)} \simeq L_{r(j)}|_{q^{-1}(V_j)}$. That is, $d_2(l)$ represents the obstruction to patching together the $\{ L_i \}$ to form a line bundle on $X$.

Let us describe in detail how the obstruction can be obtained. Since on double intersections (omitting restriction notation) we have $L_i \simeq L_j$, we can choose isomorphisms $\phi_{ij} : L_j \to L_i$ on $q^{-1}(U_{ij}) = q^{-1}(U_i) \cap q^{-1}(U_j)$. The isomorphisms $\phi_{ij}$ might fail to obey the cocycle condition, preventing us from patching the $\{L_i\}$ together. Instead we find that on the triple intersections $q^{-1}(U_{ijk})$ we get automorphisms $\phi_{ij} \phi_{jk} \phi_{ki} : L_i \to L_i$ which is then a map $g_{ijk} : q^{-1}(U_{ijk}) \to {\rm U}(1)$. We have $g_{ijk} = g_{jki}$, $g_{ikj} = g_{ijk}^{-1}$ and that $g_{ijk}$ satisfies the cocycle condition. So $\{ g_{ijk} \}$ represents a class in $H^2(M , q_* \mathcal{C}_U)$. One sees without much effort that the class of $\{ g_{ijk} \}$ in $H^2(M , q_* \mathcal{C}_U)$ is independent of the choice of representative $\{ L_i \}$ of $l$ and the choice of isomorphisms $\phi_{ij}$.
\begin{proposition}\label{obstruction}
For any $l \in H^0(M , R^1 q_* \mathcal{C}_U)$ represented by a collection $\{ L_i \}$ of line bundles, we have that $d_2(l) \in H^2(M , q_* \mathcal{C}_U)$ is represented by the cocycle $\{ g_{ijk} \}$ constructed above.
\end{proposition}
We defer the proof of Proposition \ref{obstruction} to appendix \ref{cechapp}, since our proof is surprisingly long and technical. Despite how plausible Proposition \ref{obstruction} may sound, we are not aware of a simpler proof.\\

Let $\pi : X \to M$, $\hat{\pi} : \hat{X} \to M$ be rank $n$ affine torus bundles on $M$, define the correspondence space $C = X \times_M \hat{X}$ and maps $p,\hat{p},q$ as usual. Let $k,\hat{k}$ be classes $k \in H^2(M , \pi_* \mathcal{C}_U)$, $\hat{k} \in H^2(M , \hat{\pi}_* \mathcal{C}_U)$ and let $l \in H^0(M , R^1q_* \mathcal{C}_U)$. Suppose that $k,\hat{k}$ and $l$ are related by $p^*(k) + d_2(l) = \hat{p}^*(\hat{k})$. Let $h,\hat{h}$ be the images of $k,\hat{k}$ in $H^2(X , \mathcal{C}_U), H^2(\hat{X} , \mathcal{C}_U)$ respectively. Then it follows that $p^*(h) = \hat{p}^*(\hat{h})$. If $\mathcal{G},\hat{\mathcal{G}}$ are (ungraded) gerbes on $X,\hat{X}$ with Dixmier-Douady classes $k,\hat{k}$ then it follows that $p^*(\mathcal{G}) \simeq \hat{p}^*(\hat{\mathcal{G}})$, however more can be said. If we choose an isomorphism $\gamma : p^*(\mathcal{G}) \to \hat{p}^*(\hat{\mathcal{G}})$, then as in the discussion of the Poincar\'e property in Section \ref{defpro}, for each $m \in M$ we get a well defined class $\delta_m \in H^1(T_m , \mathbb{Z}) \otimes H^1(\hat{T}_m , \mathbb{Z})$. Let us view $H^1(T_m , \mathbb{Z}) \otimes H^1(\hat{T}_m , \mathbb{Z})$ as the quoitent group of $H^2(T_m \times \hat{T}_m , \mathbb{Z})$ by $H^2(T_m , \mathbb{Z}) \times H^2(\hat{T}_m , \mathbb{Z})$, so there is a projection $Pr : H^2(T_m \times \hat{T}_m ,\mathbb{Z}) \to H^1(T_m,\mathbb{Z}) \otimes H^1(\hat{T}_m , \mathbb{Z})$. Note also that by restriction to the stalk of $R^1q_* \mathcal{C}_U$ over $m$ we get an element $l_m \in H^2(T_m \times \hat{T}_m , \mathbb{Z})$.
\begin{proposition}\label{cpp}
It is possible to choose the isomorphism $\gamma : p^*(\mathcal{G}) \to \hat{p}^*(\hat{\mathcal{G}})$ such that if $\delta_m,l_m$ and $Pr$ are defined as above then $\delta_m = Pr(l_m)$.
\end{proposition}
\begin{proof}
First choose an open cover $\{ U_i \}$ on $M$ so that $k,\hat{k},l$ admit corresponding \v{C}ech representatives $\{ k_{ijk} \}$, $\{ \hat{k}_{ijk} \}$, $\{ L_i \}$. Thus $k_{ijk}$ are functions $k_{ijk} : \pi^{-1}(U_{ijk}) \to {\rm U}(1)$, $\hat{k}_{ijk}$ are functions $\hat{k}_{ijk} : \hat{\pi}^{-1}(U_{ijk}) \to {\rm U}(1)$ and $L_i$ are line bundles $L_i \to q^{-1}(U_i)$. As in the construction of $d_2 l$ we choose isomorphisms $\phi_{ij} : L_j \to L_i$ and get a corresponding cocycle $\{ g_{ijk} \}$. The equality $p^*(k) + d_2(l) = \hat{p}^*(\hat{k})$ implies that (after refining the cover if necessary) we can find maps $h_{ij} : q^{-1}(U_{ij}) \to {\rm U}(1)$ so that
\begin{equation*}
p^*(k_{ijk})  g_{ijk} = \hat{p}^*(\hat{k}_{ijk})  h_{ij}  h_{jk}  h_{ki}.
\end{equation*}
We can redefine the isomorphisms $\phi_{ij}$ to absorb the $h_{ij}$ terms and arrive at
\begin{equation}\label{morphcond}
p^*(k_{ijk})  g_{ijk} = \hat{p}^*(\hat{k}_{ijk}).
\end{equation}
To proceed let us interpret $k_{ijk}$ as a gerbe isomorphic to $\mathcal{G}$. Indeed define line bundles $K_{ij} \to \pi^{-1}(U_{ij})$ by simply setting $K_{ij} = 1$, the trivial bundle and with gerbe multiplication $\theta_{ijk} : K_{ij} \otimes K_{jk} \to K_{ik}$ given by $\theta_{ijk}(1 \otimes 1) = k_{ijk}$. Construct a similar gerbe $\{ \hat{K}_{ij}, \hat{\theta}_{ijk} \}$ out of the cocycle $\{ \hat{k}_{ijk} \}$. Note that the isomorphisms $\phi_{ij} : L_j \to L_i$ determine local sections $\sigma_{ij} : q^{-1}(U_{ij}) \to L_{ij}$ where $L_{ij} = L_i \otimes L_j^*$ is the trivializable gerbe built out of $\{ L_i \}$. Let $\lambda_{ijk} : L_{ij} \otimes L_{jk} \to L_{ik}$ be the gerbe multiplication for this trivializable gerbe. Thus $\lambda_{ijk}(\sigma_{ij},\sigma_{jk}) = g_{ijk} \sigma_{ik}$. We define line bundle isomorphisms $\psi_{ij} : p^*(K_{ij}) \otimes L_{ij} \to \hat{p}^*( \hat{K}_{ij})$ by setting $\psi_{ij}( 1 \otimes \sigma_{ij}) = 1$. The composition 
\begin{equation*}\xymatrix{
(p^*(K_{ij}) \otimes L_{ij}) \otimes (p^*(K_{jk}) \otimes L_{jk}) \ar[rr]^-{p^*\theta_{ijk} \otimes \lambda_{ijk}} & & (p^*(K_{ik}) \otimes L_{ik}) \ar[r]^-{\psi_{ik}} & \hat{p}^*( \hat{K}_{ik})
}
\end{equation*}
sends $(1 \otimes \sigma_{ij}) \otimes (1 \otimes \sigma_{jk})$ to $p^*(k_{ijk}) g_{ijk}$. On the other hand the composition
\begin{equation*}\xymatrix{
(p^*(K_{ij}) \otimes L_{ij}) \otimes (p^*(K_{jk}) \otimes L_{jk}) \ar[rr]^-{\psi_{ij} \otimes \psi_{jk}} & & \hat{p}^*(\hat{K}_{ij}) \otimes \hat{p}^*(\hat{K}_{jk}) \ar[r]^-{\hat{p}^*\hat{\theta}_{ijk}} & \hat{p}^*(\hat{K}_{ik})
}
\end{equation*}
sends $(1 \otimes \sigma_{ij}) \otimes (1 \otimes \sigma_{jk})$ to $\hat{p}^*(\hat{k}_{ijk})$. By Equation (\ref{morphcond}) these two compositions agree so that the $\{ \psi_{ij} \}$ define an isomorphism of gerbes. We take $\gamma : p^*(\mathcal{G}) \to \hat{p}^*(\hat{\mathcal{G}})$ to be this isomorphism. For any $U_i$ with $m \in U_i$ we have that the line bundle $L_i$ restricted to the fiber over $m$ has Chern class $l_m \in H^2(T_m \times \hat{T}_m , \mathbb{Z})$. Restricted to the fibers $\{ K_{ij},\theta_{ijk} \}$ and $\{ \hat{K}_{ij} , \hat{\theta}_{ijk} \}$ are of course trivializable. Comparing with the restriction of the $\{ \psi_{ij} \}$ and using the definition of the class $\delta$ as in Section \ref{defpro} we immediately see the relation $\delta_m = Pr(l_m)$. 
\end{proof}

Let $X,\hat{X} \to M$ be rank $n$ affine torus bundles, $\mathcal{G},\hat{\mathcal{G}}$ gerbes on $X,\hat{X}$ and $L(V)$ the lifting gerbe of the vertical bundle of $X \to M$. Set $\tilde{\mathcal{G}} = \hat{\mathcal{G}} \otimes \hat{\pi}^*(L(V))$ and suppose $\mathcal{G},\tilde{\mathcal{G}}$ have graded Dixmier-Douady classes of the form $[\mathcal{G}] = (\pi^*(\xi),\pi^*(k)) \in H^1(X,\mathbb{Z}_2) \times H^2(X,\mathcal{C}_U)$, $[\tilde{\mathcal{G}}] = (\hat{\pi}^*(\tilde{\xi}),\hat{\pi}^*(\tilde{k})) \in H^1(\hat{X},\mathbb{Z}_2) \times H^2(\hat{X},\mathcal{C}_U)$, where $\xi,\tilde{\xi} \in H^1(M,\mathbb{Z}_2)$, $k \in H^2(M , \pi_* \mathcal{C}_U)$, $\tilde{k} \in H^2( M , \hat{\pi}_* \mathcal{C}_U)$. Suppose additionally that $\xi = \tilde{\xi}$ and that the local systems $\Lambda^* = R^1 \pi_* \mathbb{Z}$, $\hat{\Lambda}^* = R^1 \hat{\pi}_* \mathbb{Z}$ are dual. Fix an isomorphism $\delta : \Lambda \to \hat{\Lambda}^*$. By Proposition \ref{cpp} a sufficient condition for $(X,\mathcal{G}),(\hat{X},\hat{\mathcal{G}})$ to be T-dual is the equality
\begin{equation}\label{plbr}
p^*(k) + d_2 P = \hat{p}^*(\tilde{k}),
\end{equation}
where $P \in H^0(M , R^1 q_* \mathcal{C}_U)$ is defined as follows: choose an open cover $\{ U_i \}$ of $M$ for which there are local trivializations $\pi^{-1}(U_i) \simeq U_i \times T$, $\hat{\pi}^{-1}(U_i) \simeq U_i \times \hat{T}$ (where $T,\hat{T}$ denote rank $n$ tori). The local system $\Lambda \otimes \hat{\Lambda}$ can be thought of as a subsheaf of the local system $R^2 q_* \mathbb{Z} \simeq \wedge^2( \Lambda^* \oplus \hat{\Lambda}^*) \simeq \wedge^2 ( \hat{\Lambda} \oplus \Lambda)$. Thinking of $\delta$ as a section of $R^2 q_* \mathbb{Z}$, we have on $q^{-1}(U_i) \simeq U_i \times T \times \hat{T}$ a corresponding line bundle $P_i \to U_i \times T \times \hat{T}$ which we can think of as a locally defined Poincar\'e line bundle. Then $P$ is defined as the class $\{ P_i \} \in H^0(M , R^1 q_* \mathcal{C}_U)$ made up of all the local Poincar\'e line bundles.

\begin{remark}
We can think of the element $P \in H^0( M , R^1q_* \mathcal{C}_U)$ as being the collection $\{ P_i \}$ of Poincar\'e line bundles along the fibers of the correspondence space. Then $d_2(P) \in H^2(M , q_* \mathcal{C}_U)$ can be thought of on the one hand as the obstruction to patching the Poincar\'e line bundles together and on the other hand as a trivializable gerbe $\{ P_i \otimes P_j^* \}$ constructed from the Poincar\'e line bundles. Equation (\ref{plbr}) is central to T-duality. Roughly it states that the gerbes $\mathcal{G},\tilde{\mathcal{G}}$ pulled back to $C$ differ by the trivializable gerbe $\{ P_i \otimes P_j^* \}$.
\end{remark}

We will need to compare the Leray spectral sequences for the sheaves $\mathbb{Z}$ and $\mathcal{C}_U$. Since the sheaf cohomology of these two sheaves are closely related by the coboundary in the long exact sequence associated to the exponential sequence
\begin{equation}\label{expseq}
0 \to \mathbb{Z} \to \mathcal{C}(\mathbb{R}) \to \mathcal{C}_U \to 0,
\end{equation}
one would expect a relation between the Leray spectral sequences. This is the case, but the only proof we are aware of is long and complicated. We provide the full details of the proof in a separate paper \cite{bar2}. Let us state here the results from \cite{bar2} that we need.
\begin{theorem}\label{cobfilt}
Let $\pi : X \to M$ be a map between spaces $X,M$, with every subspace of $X$ paracompact. Let $E_r^{p,q}(\pi,\mathbb{Z})$,$E_r^{p,q}(\pi,\mathcal{C}_U)$ be the Leray spectral sequences associated to the sheaves $\mathbb{Z},\mathcal{C}_U$ and let $F^{p,n}(\pi,\mathbb{Z}),F^{p,n}(\pi,\mathcal{C}_U)$ be the associated filtrations on $H^n(X,\mathbb{Z})$,$H^n(X,\mathcal{C}_U)$. Then
\begin{itemize}
\item{There is a morphism $\delta_r : E_r^{p,q}(\pi,\mathcal{C}_U) \to E_r^{p,q+1}(\pi,\mathbb{Z})$ of spectral sequences, which at the $E_2$ stage is the map $H^p(M,R^q \pi_* \mathcal{C}_U) \to H^p(M,R^{q+1} \pi_*\mathbb{Z})$ induced by the coboundary map $R^q \pi_* \mathcal{C}_U \to R^{q+1} \pi_* \mathbb{Z}$.}
\item{The coboundary $\delta : H^n(X,\mathcal{C}_U) \to H^{n+1}(X,\mathbb{Z})$ restricts to morphisms $\delta : F^{p,p+q}(\pi,\mathcal{C}_U) \to F^{p,p+q+1}(\pi,\mathbb{Z})$ which are isomorphisms whenever $p+q \ge 1$ and surjective for $p=q=0$.}
\item{The coboundary maps between filtrations induce quotient maps $E_\infty^{p,q}(\pi,\mathcal{C}_U) \to E_\infty^{p,q+1}(\pi,\mathbb{Z})$ which are isomorphisms for $q \ge 1$ and surjections for $q = 0$. Moreover these maps coincide with the limits $\delta_\infty = \lim_r \delta_r$.}
\end{itemize}
\end{theorem}

There is one more ingredient we need before we can get to the proof of Theorem \ref{exist}.
\begin{proposition}\label{longex1thm}
Let $\pi : X \to M$ be a rank $n$ affine torus bundle on $M$ with monodromy local system $\Lambda = (R^1 \pi_* \mathbb{Z})^*$ and twisted Chern class $c \in H^2(M , \Lambda)$. There is an exact sequence
\begin{equation}\label{longex1}
\xymatrix{
H^2(M , \mathcal{C}_U) \ar[r]^-{\pi^*} & H^2(M , \pi_*\mathcal{C}_U) \ar[r]^-{\delta} & H^2(M,\Lambda^*) \ar[r]^-{ \smallsmile c } & H^4(M , \mathbb{Z})
}
\end{equation}
where the map $\delta : H^2(M , \pi_*\mathcal{C}_U) \to H^2(M,\Lambda^*)$ corresponds to the sheaf morphism $\delta : \pi_*\mathcal{C}_U \to R^1 \pi_* \mathbb{Z} = \Lambda^*$ in the long exact sequence of higher direct image functors of $\pi_*$ applied to the exponential sequence (\ref{expseq}).
\end{proposition}
\begin{proof}
Starting with the exponential sequence (\ref{expseq}) we get an induced long exact sequence of higher direct image sheaves. The relevant part is the following exact sequence:
\begin{equation*}
0 \to \pi_*\mathbb{Z} \to \pi_* \mathcal{C}(\mathbb{R}) \to \pi_* \mathcal{C}_U \buildrel \delta \over \to R^1 \pi_* \mathbb{Z} \to R^1 \pi_* \mathcal{C}(\mathbb{R}).
\end{equation*}
We note that $\pi_* \mathbb{Z} = \mathbb{Z}$ as the fibers of $X \to M$ are connected and that $R^1 \pi_* \mathcal{C}(\mathbb{R}) = 0$, since by assumption every subset of $X$ is paracompact and we then use partitions of unity. Letting $Q = \pi_* \mathcal{C}(\mathbb{R}) / \mathbb{Z}$ be the quotient sheaf we now get short exact sequences:
\begin{eqnarray}
&& 0 \to Q \to \pi_* \mathcal{C}_U \to R^1 \pi_* \mathbb{Z} \to 0, \label{exact1} \\
&& 0 \to \mathbb{Z} \to \pi_* \mathcal{C}(\mathbb{R}) \to Q \to 0.
\end{eqnarray}
Observe that $\pi_* \mathcal{C}(\mathbb{R})$ is a fine sheaf so that for $k \ge 1$ we have $H^k(M,Q) \simeq H^{k+1}(M,\mathbb{Z})$. Combining this with the long exact sequence for (\ref{exact1}) we get a long exact sequence part of which looks like
\begin{equation*}\xymatrix{
H^2(M , \mathcal{C}_U) \ar[r]^-{\pi^*} & H^2(M , \pi_*\mathcal{C}_U) \ar[r]^-{\delta} & H^2(M,\Lambda^*) \ar[r]^-{ \beta} & H^4(M , \mathbb{Z}).
}
\end{equation*}
We can prove that $\beta$ is given by the cup product with the twisted Chern class, but the proof is not so short, so it will be easier to simply replace $\beta$ by $\smallsmile c$ and show that exactness still holds at $H^2(M,\Lambda^*)$. Note first of all that $\delta(x) \smallsmile c = 0$ for all $x \in H^2(M , \pi_*\mathcal{C}_U)$ follows from the commutativity of the following diagram:
\begin{equation*}\xymatrix{
E_2^{2,1}(\pi,\mathbb{Z}) \ar[r]^{d_2} & E_2^{4,0}(\pi,\mathbb{Z}) \\
E_2^{2,0}(\pi,\mathcal{C}_U) \ar[u]^\delta \ar[r]^{d_2} & 0 \ar[u]
}
\end{equation*}
together with that fact that $d_2 : E_2^{2,1}(\pi,\mathbb{Z}) \to E_2^{4,0}(\pi,\mathbb{Z})$ is given cupping with $c$ to get a class in $H^4(M,\Lambda \otimes \Lambda^*)$ and pairing off the dual local systems.\\

Let $x \in H^2(M,\Lambda^*)$ be such that $x \smallsmile c = 0$. As we have seen $x \smallsmile c = d_2 x \in E_2^{4,0}(\pi,\mathbb{Z})$, so the condition $d_2x = 0$ implies that there exists an element $\tilde{x} \in F^{2,3}(\pi,\mathbb{Z}) \subseteq H^3(X,\mathbb{Z})$ such that $x \in H^2(M,\Lambda^*) = E_2^{2,1}(\pi,\mathbb{Z})$ is a representative for the image of $\tilde{x}$ in $E_\infty^{2,1}(\pi,\mathbb{Z}) = F^{2,3}(\pi,\mathbb{Z})/F^{3,3}(\pi,\mathbb{Z})$. By Theorem \ref{cobfilt} there exists a unique $\tilde{y} \in F^{2,2}(\pi,\mathcal{C}_U)$ such that $\tilde{x} = \delta( \tilde{y})$. Thinking of $\tilde{y}$ as an element $\tilde{y} \in E_\infty^{2,0}(\pi,\mathcal{C}_U)$ we have that $\delta_{\infty}(\tilde{y})$ is the projection of $\tilde{x}$ to $E_\infty^{2,1}(\pi,\mathbb{Z})$. Let $y' \in E_2^{2,0}(\pi,\mathcal{C}_U)$ be a lift of $\tilde{y}$ to the $E_2$ stage. Then it follows that $\delta_2(y'),x \in E_2^{2,1}(\pi,\mathbb{Z})$ represent the same element in $E_\infty^{2,1}(\pi,\mathbb{Z})$. Therefore there exists $w \in E_2^{0,3}(\pi,\mathbb{Z})$ such that $\delta_2(y') + d_2(w) = x$. However it is clear that the coboundary $\delta_2 : E_2^{0,2}(\pi,\mathcal{C}_U) \to E_2^{0,3}(\pi,\mathbb{Z})$ is an isomorphism so there is a $z \in E_2^{0,2}(\pi,\mathcal{C}_U)$ such that $w = \delta_2 (z)$. Let $y = y' + d_2(z)$. We immediately see that $\delta_2(y) = x$ as required.
\end{proof}

\begin{proof}[Proof of Theorem \ref{exist}]
We first show sufficiency: suppose $h$ is in the image of $H^2(M , \pi_*\mathcal{C}_U ) \to H^2(X , \mathcal{C}_U)$. We will construct a T-dual. By assumption there exists $k \in H^2(M , \pi_*\mathcal{C}_U )$ that maps to $h \in H^2(X , \mathcal{C}_U)$. Let $\hat{c} = \delta(k)$, where $\delta$ is the map $H^2(M , \pi_*\mathcal{C}_U) \to H^2(M,\Lambda^*)$ in the exact sequence (\ref{longex1}). Therefore $c \smallsmile \hat{c} = 0$. Let $\hat{\Lambda} = \Lambda^*$ be the local system dual to $\Lambda$. The pair $(\hat{\Lambda},\hat{c})$ determines an affine torus bundle $\hat{\pi} : \hat{X} \to M$ such that $\hat{\Lambda} = (R^1 \hat{\pi}_* \mathbb{Z})^*$ is the monodromy local system and $\hat{c} \in H^2( M , \hat{\Lambda} ) = H^2( M , \Lambda^* )$ is the twisted Chern class.\\

Just as we constructed the exact sequence (\ref{longex1}) out of the torus bundle $X \to M$, we have a similar exact sequence corresponding to $\hat{X} \to M$ which has the form
\begin{equation}\label{longex2}
\xymatrix{
H^2(M , \mathcal{C}_U) \ar[r]^-{\hat{\pi}^*} & H^2(M , \hat{\pi}_*\mathcal{C}_U) \ar[r]^-{\delta} & H^2(M,\Lambda) \ar[r]^-{ \smallsmile \hat{c} } & H^4(M , \mathbb{Z}).
}
\end{equation}
Making use of $c \smallsmile \hat{c} = 0$ we may find a $k' \in H^2(M , \hat{\pi}_*\mathcal{C}_U)$ so that $\delta(k') = c$. Note that such a $k'$ is only unique up to addition of terms $\hat{\pi}^*(e)$, where $e \in H^2(M , \mathcal{C}_U)$. We will soon make use of this freedom.\\

The local systems of $X,\hat{X}$ are dual by construction of $\hat{X}$, so there is a section $\delta$ of $\Lambda^* \otimes \hat{\Lambda}^* \subset \wedge^2 ( \Lambda^* \oplus \hat{\Lambda}^* )$. Note that $\Lambda \oplus \hat{\Lambda}$ is the monodromy local system for the correspondence space $C = X \times_M \hat{X}$. Therefore $R^1q_* \mathcal{C}_U \simeq \wedge^2 ( \Lambda^* \oplus \hat{\Lambda}^*)$ and $\delta$ determines a corresponding element $P \in H^0(C , R^1 q_* \mathcal{C}_U)$ which represents the Poincar\'e line bundles of the fibers of $C$.

We set $d = p^*(k) + d_2(P) - \hat{p}^*(k')$. The idea now is that if $d \neq 0$ then we will suitably modify our choice of $k'$ to absorb away $d$. Let us write out the exact sequence for $q : C \to M$ determined in Proposition \ref{longex1thm}. It is as follows:
\begin{equation}\label{longex3}
\xymatrix{
H^2(M , \mathcal{C}_U) \ar[r]^-{q^*} & H^2(M , q_*\mathcal{C}_U) \ar[r]^-{\delta} & H^2(M,\Lambda^*) \oplus H^2(M,\hat{\Lambda}^*) \ar[r]^-{ \alpha } & H^4(M , \mathbb{Z}),
}
\end{equation}
where $\alpha : H^2(M,\Lambda^*) \oplus H^2(M,\hat{\Lambda}^*) \to H^4(M,\mathbb{Z})$ is the map $\alpha(x,y) = x \smallsmile c + y \smallsmile \hat{c}$. We note that $d \in H^2(M , q_*\mathcal{C}_U)$, so we may apply to $d$ the map $\delta$ in (\ref{longex3}). The result is $\delta(d) = (\hat{c},-c) + \delta d_2(P)$. But $\delta d_2(P) = d_2 \delta_2(P)$, where $\delta_2 : E_2^{0,1}(q,\mathcal{C}_U) \to E_2^{0,2}(q,\mathbb{Z})$ is the coboundary map as in Theorem \ref{cobfilt}. Thus $\delta_2(P) \in H^0(M , \wedge^2( \Lambda^* \oplus \hat{\Lambda}^*))$ is the Chern class of the Poincar\'e line bundles of the fibers. From this and Theorem \ref{atblsd} we deduce that\footnote{Actually the value of $d_2 \delta_2(P)$ can equal $(-\hat{c},c)$ or $(\hat{c},-c)$ depending on the convention used to identify $\Lambda^* \otimes \hat{\Lambda}^*$ as a subsheaf of $\wedge^2( \Lambda^* \oplus \hat{\Lambda}^*)$. Of course we are free to choose the convention where $d_2 \delta_2(P) = (-\hat{c},c)$.} $d_2 \delta_2(P) = (-\hat{c},c)$, so we conclude that $\delta(d) = 0$. Looking again at the exact sequence (\ref{longex3}) we find that $d = q^*(e)$ for some $e \in H^2(M,\mathcal{C}_U)$. If we let $\tilde{k} = k' + \hat{\pi}^*(e)$ then we still have $\delta(\tilde{k}) = c$, but we also have
\begin{equation}\label{ppeq}
p^*(k) + d_2(P) = \hat{p}^*(\tilde{k}).
\end{equation}
Let $\pi^*(\xi) \in H^1(X,\mathbb{Z}_2)$ be the grading class of $\mathcal{G}$, that is $(\pi^*(\xi) , \pi^*(k)) \in H^1(X , \mathbb{Z}_2) \times H^2(X, \mathcal{C}_U)$ is the graded Dixmier-Douady class of $\mathcal{G}$, where $\xi \in H^1(M,\mathbb{Z}_2)$. The pair $(\hat{\pi}^*(\xi) , \hat{\pi}^*(\tilde{k})) \in H^1(\hat{X} , \mathbb{Z}_2) \times H^2(\hat{X}, \mathcal{C}_U)$ determine up to isomorphism a graded gerbe $\tilde{\mathcal{G}}$ on $\hat{X}$. Let $V \to M$ be the vertical bundle of $X \to M$ and $L(V)$ the lifting gerbe. Define $\hat{\mathcal{G}} = \tilde{\mathcal{G}} \otimes L(V)^*$. From the discussion following Proposition \ref{cpp} we have established that Equation (\ref{ppeq}) is sufficient for the existence of an isomorphism $\gamma : p^*(\mathcal{G}) \to \hat{p}^*( \tilde{\mathcal{G}}) = \hat{p}^*(\hat{\mathcal{G}}) \otimes q^*(L(V))$ satisfying the Poincar\'e property. We have therefore constructed a T-dual pair $(\hat{X},\hat{\mathcal{G}})$.\\

We now prove necessity of the conditions on the graded Dixmier-Douady class of $\mathcal{G}$. Thus let $(X,\mathcal{G})$ be a pair that admits a T-dual $(\hat{X},\hat{\mathcal{G}})$, let $(\chi , h) \in H^1(X,\mathbb{Z}_2) \times H^2(X,\mathcal{C}_U)$ and $(\hat{\chi},\hat{h}) \in H^1(\hat{X},\mathbb{Z}_2) \times H^2(\hat{X},\mathcal{C}_U)$ be the graded Dixmier-Douady classes of $\mathcal{G},\hat{\mathcal{G}}$. In addition let $\tilde{\mathcal{G}} = \hat{\mathcal{G}} \otimes \hat{\pi}^*(L(V))$ and let $(\tilde{\chi},\tilde{h})$ be the corresponding graded Dixmier-Douady class. Letting $w_1$ be the first Stiefel-Whitney class of $V$ and $W_3$ the third integral Stiefel-Whitney class, we in fact have $\tilde{\chi} = \hat{\chi} + \hat{\pi}^*(w_1)$ and $\tilde{h} = \hat{h} + \hat{\pi}^*(W_3) + \beta( \hat{\chi} \smallsmile \hat{\pi}^*(w_1))$, where $\beta : H^2(\hat{X},\mathbb{Z}_2) \to H^3(\hat{X},\mathbb{Z})$ is the Bockstein homomorphism and we have identified $H^3(\hat{X},\mathbb{Z})$ with $H^2(\hat{X},\mathcal{C}_U)$.\\

By the definition of T-duality we know that $p^*(\mathcal{G})$ and $\hat{p}^*(\tilde{\mathcal{G}})$ are isomorphic. In particular $p^*(\chi) = \hat{p}^*(\tilde{\chi})$. By considering the Leray spectral sequences for the constant sheaf $\mathbb{Z}_2$ we deduce that there exists a class $\xi \in H^1(M,\mathbb{Z}_2)$ such that $\chi = \pi^*(\xi)$ and $\tilde{\chi} = \hat{\pi}^*(\xi)$. Note that this also implies $\hat{\chi} = \hat{\pi}^*(\hat{\xi})$, where $\hat{\xi} = \xi + w_1$. Next we consider the implications of the relation $p^*(h) = \hat{p}^*(\tilde{h})$. Since $(X,\mathcal{G}),(\hat{X},\hat{\mathcal{G}})$ are T-dual we are implicitly assuming $\mathcal{G},\hat{\mathcal{G}}$ are trivial on the fibers. In terms of filtrations in cohomology this implies $p^*(h),\hat{p}^*(\tilde{h}) \in F^{1,2}(q,\mathcal{C}_U)$. Let $x,\tilde{x}$ denote the projections of $p^*(h),\hat{p}^*(\tilde{h})$ to $E_\infty^{1,1}(q,\mathcal{C}_U)$. Note that $E_\infty^{1,1}(q,\mathcal{C}_U)$ is a subgroup of $E_2^{1,1}(q,\mathcal{C}_U) \simeq H^1(M , \wedge^2 ( \Lambda^* \oplus \hat{\Lambda}^* ) )$. But it is clear that $x$ lies in the subgroup $H^1(M, \wedge^2 ( \Lambda^*) )\simeq E_2^{1,1}(\pi,\mathcal{C}_U)$ and $\tilde{x}$ in the subgroup $H^1(M, \wedge^2 ( \hat{\Lambda}^*)) \simeq E_2^{1,1}(\hat{\pi},\mathcal{C}_U)$, so the equality $x = \tilde{x}$ is only possible if $x = \tilde{x} = 0$. This also implies that $h \in F^{2,2}(\pi,\mathcal{C}_U) = E_\infty^{2,0}(\pi,\mathcal{C}_U)$, which is a quotient of $E_2^{2,0}(\pi,\mathcal{C}_U) = H^2(M , \pi_* \mathcal{C}_U )$. Thus $h$ lies in the image of $\pi^* : H^2(M , \pi_* \mathcal{C}_U ) \to H^2(X, \mathcal{C}_U)$.
\end{proof}

With a little more work we can deduce from the proof of Theorem \ref{exist} some topological consequences of the T-duality relation.
\begin{proposition}\label{toprelations}
Let $(X,\mathcal{G}),(\hat{X},\mathcal{G})$ be T-dual pairs. The graded Dixmier-Douady classes of $\mathcal{G},\hat{\mathcal{G}}$ have the form $[\mathcal{G}] = ( \pi^*(\xi) , h ) \in H^1(X,\mathbb{Z}_2) \times H^3(X,\mathbb{Z})$, $[\hat{\mathcal{G}}] = (\hat{\pi}^*(\hat{\xi}) , \hat{h}) \in H^1(\hat{X},\mathbb{Z}_2) \times H^3(\hat{X},\mathbb{Z})$, where $\xi,\hat{\xi} \in H^1(M,\mathbb{Z}_2)$, $h \in F^{2,3}(\pi,\mathbb{Z}) \subseteq H^3(X,\mathbb{Z})$ and $\hat{h} \in F^{2,3}(\hat{\pi},\mathbb{Z}) \subseteq H^3(\hat{X},\mathbb{Z})$. Let $c \in H^2(M,\Lambda)$, $\hat{c} \in H^2(M,\Lambda^*)$ be the twisted Chern classes of $X,\hat{X}$. We then have
\begin{eqnarray}
\hat{\xi} &=& \xi + w_1, \label{tc1} \\
p^*(h) &=& \hat{p}^*(\hat{h}) + q^*( W_3 + \beta( \xi \smallsmile w_1)), \label{tc2} \\
c \smallsmile \hat{c} &=& 0, \label{tc3}
\end{eqnarray}
where $w_1,W_3$ are the first Stiefel-Whitney class and third integral Stiefel-Whitney class of the vertical bundle $V$ of $X$ and $\beta : H^2(\hat{X},\mathbb{Z}_2) \to H^3(\hat{X},\mathbb{Z})$ is the Bockstein homomorphism.

Since $c \smallsmile \hat{c} = 0$ we find that $c \in E_2^{2,1}(\hat{\pi},\mathbb{Z})$, $\hat{c} \in E_2^{2,1}(\pi,\mathbb{Z})$ project to classes $[\hat{c}] \in E_\infty^{2,1}(\pi,\mathbb{Z})$, $[c] \in E_\infty^{2,1}(\hat{\pi},\mathbb{Z})$. On the other hand let $[h] \in E_\infty^{2,1}(\pi,\mathbb{Z})$, $[\hat{h}] \in E_\infty^{2,1}(\hat{\pi},\mathbb{Z})$ be the images of $h,\hat{h}$ under the projections $F^{2,3}(\pi,\mathbb{Z}) \to E_2^{2,1}(\pi,\mathbb{Z})$, $F^{2,3}(\hat{\pi},\mathbb{Z}) \to E_2^{2,1}(\hat{\pi},\mathbb{Z})$. We have:
\begin{eqnarray*}
\left[ h \right] &=& \left[ \hat{c} \right] , \\
\left[ \right.  \! \hat{h}  \left. \! \right] &=& \left[ c\right].
\end{eqnarray*}
\end{proposition}
\begin{proof}
In the proof of Theorem \ref{exist} we already established the relations (\ref{tc1}),(\ref{tc2}). Let $\tilde{h} = \hat{h} + \hat{\pi}^*( W_3 + \beta( \xi \smallsmile w_1))$. Choose lifts $k,\hat{k}$ for $h,\hat{h}$ to classes $k \in H^2(M,\pi_* \mathcal{C}_U)$, $\hat{k} \in H^2(M,\hat{\pi}_* \mathcal{C}_U)$. Let $\tilde{k} = k + \hat{\pi}^*( W_3 + \beta( \xi \smallsmile w_1))$, where we are using the map $\hat{\pi}^* : H^3(M,\mathbb{Z}) \simeq H^2(M,\mathcal{C}_U) \to H^2(M , \hat{\pi}_* \mathcal{C}_U)$. The proof of Proposition \ref{cpp} and the discussion directly after can be reversed, that is to say that since $(X,\mathcal{G}),(\hat{X},\hat{\mathcal{G}})$ are T-dual it is possible to choose the lifts $k,\hat{k}$ so that
\begin{equation}\label{kpkt}
p^*(k) + d_2(P) = \hat{p}^*(\tilde{k}),
\end{equation}
where $P \in E_2^{0,1}(q,\mathcal{C}_U) = H^0(M,R^1 q_* \mathcal{C}_U)$ represents the Poincar\'e line bundle on the fibers of $C \to M$. Now if we apply the coboundary map $\delta_2 : E_2^{2,0}(q,\mathcal{C}_U) \to E_2^{2,1}(q,\mathbb{Z}) = H^2(M,\Lambda^*) \oplus H^2(M,\Lambda)$ as defined in Theorem \ref{cobfilt} to Equation (\ref{kpkt}), we get $(\delta_2(k),0) + (-\hat{c},c) = (0,\delta_2(\tilde{k})) = (0,\delta_2(\hat{k}))$. We have also used the identity $\delta_2 d_2(P) = d_2 \delta_2 P = (-\hat{c},c)$ which was established in the proof of Theorem \ref{exist}. It follows immediately that
\begin{eqnarray*}
\delta_2(k) &=& \hat{c}, \\
\delta_2(\hat{k}) &=& c.
\end{eqnarray*}
Now using the exact sequence (\ref{longex1}) applied to $k$, we immediately find $\hat{c} \smallsmile c = 0$, that is Equation (\ref{tc3}). Finally we note that $k$ is a representative in $E_2^{2,0}(\pi,\mathcal{C}_U)$ for $h \in F^{2,2}(\pi,\mathcal{C}_U) = E_\infty^{2,0}(\pi,\mathcal{C}_U)$. By Theorem \ref{cobfilt} it follows that $\delta_2(k) \in E_2^{2,1}(\pi,\mathbb{Z})$ is a representative for the image of $h$ in $E_\infty^{2,1}(\pi,\mathbb{Z})$, that is $\delta_2(k) = \hat{c}$ and $h$ represent the same class in $E_\infty^{2,1}(\pi,\mathbb{Z})$. So $[\hat{c}] = [h]$. Similarly $[c] = [\hat{h}]$.
\end{proof}


\section{Differential T-duality}\label{dtd}


\subsection{Twisted cohomology}\label{twco}

There are two types of twists of de Rham cohomology we wish to consider. If $M$ is a smooth manifold and $H \in \Omega^3(M)$ is a closed $3$-form we define a twisted differential $d_H : \Omega^*(M) \to \Omega^*(M)$ by the formula $d_H \omega = d\omega + H \wedge \omega$. We consider $\Omega^*(M)$ to be $\mathbb{Z}_2$-graded using the mod-$2$ degree of forms and then $d_H$ has odd degree. Thus we obtain a $\mathbb{Z}_2$-graded cohomology $H^*(M,H)$ which is usually called the {\em $H$-twisted cohomology} of $M$. 

If $H$ and $H'$ determine the same cohomology class in $H^3(M,\mathbb{R})$ then the twisted cohomologies are isomorphic. Indeed if we choose a $2$-form $B$ such that $H' = H + dB$ then one sees that $d_{H'} \circ e^{-B} = e^{-B} \circ d_H$ where $e^B : \Omega^*(M) \to \Omega^*(M)$ is the map $e^B \omega = \omega + B \wedge \omega + \frac{1}{2} B \wedge B \wedge \omega + \dots $. This gives an induced map $e^{-B} : H^*(M,H) \to H^*(M,H')$ which is an isomorphism. Note however that the isomorphism $e^{-B} : H^*(M,H) \to H^*(M,H')$ depends on the choice of $2$-form $B$. In general we do not have a {\em canonical} isomorphism between twisted cohomology for different representatives $H,H'$ of the same cohomology class.\\

The second kind of twist we wish to define involves twisting by de Rham complex by a bundle $V$ with flat connection $\nabla$. Thus we take the complex $\Omega^*(M,V) = \Gamma( V \otimes \wedge^* T^*M)$ of $V$-valued forms equipped with the differential $d_\nabla$ which is characterized by the property
\begin{equation*}
d_\nabla ( v \otimes \omega ) = \nabla(v) \wedge \omega + v \otimes d\omega
\end{equation*}
for any section $v$ of $V$ and differential form $\omega$. The resulting cohomology $H^*(M,V)$ is easily shown to coincide with the sheaf cohomology of the local system given by the sheaf of $\nabla$-constant sections of $V$.\\

We observe that the two twists described above can be combined without difficulty. Thus given a closed $3$-form $H \in \Omega^3(M)$ and flat vector bundle $(V,\nabla)$, we define a differential $d_{\nabla,H}$ on the $\mathbb{Z}_2$-graded complex $\Omega^*(M,V)$ of $V$-valued differential forms which is given by
\begin{equation*}
d_{\nabla,H} \omega = d_\nabla \omega + H \wedge \omega.
\end{equation*}
The resulting $\mathbb{Z}_2$-graded cohomology will be denoted $H^*(M,(V,H))$. Note that as before if $H,H'$ are closed $3$-forms representing the same cohomology class and $V,V'$ are isomorphic flat vector bundles, then the twisted cohomologies $H^*(M,(V,H))$, $H^*(M,(V,H'))$ are isomorphic (non-canonically).

Observe that the differential $d_{\nabla,H}$ preserves a filtration on $\Omega^*(M,V)$ by form degree, that is if we let $F^p = \Omega^{\ge p}(M,V)$ be the space of $V$-valued forms of degree $\ge p$ then $d_{\nabla , H} F^p \subseteq F^{p+1}$. This determines a singly graded spectral sequence $\{ (E_r^p,d_r) \}_{r=1}^{\infty}$ which abuts to $H^*(M,(V,H))$. One easily sees that $E_3^p = H^p(M,V)$ is degree $p$ cohomology twisted by $V$ and the differential $d_3$ is the product $d_3 \omega = H \wedge \omega$. There are higher order differentials that can expressed using Massey products.


\subsection{Invariant forms and invariant cohomology}\label{invtc}

Let $\pi : X \to M$ be a rank $n$ affine torus bundle over $M$. A choice a affine structure\footnote{Recall that by Theorem \ref{pab} the affine structure on $X$ is unique up to isomorphism.} on $X$ is given by specifying a choice of principal ${\rm Aff}(T^n)$-bundle $f : P \to M$ and an identification of $X$ with the quotient $X \simeq P / {\rm GL}(n,\mathbb{Z})$. Let $p : P \to X$ be the quotient, so $f = \pi \circ p$. Also note that this means $P$ is a principal ${\rm GL}(n,\mathbb{Z})$-bundle over $X$. Now while $X$ itself is generally not a principal bundle the affine structure allows us to speak of invariant forms on $X$.

\begin{definition}
A differential form $\omega \in \Omega^*(X)$ is said to be {\em invariant} if $p^*(\omega) \in \Omega^*(P)$ is invariant under the ${\rm Aff}(T^n)$-action.
\end{definition}

Note that a form $\beta \in \Omega^*(P)$ is invariant under the ${\rm GL}(n,\mathbb{Z})$ subgroup if and only if it is a pullback $\beta = p^*(\omega)$ of a form $\omega \in \Omega^*(X)$. The form $\omega$ is necessarily unique. In particular if $\beta \in \Omega^*(P)$ is invariant under the full affine group ${\rm Aff}(T^n)$, then $\beta = p^*(\omega)$ is the pullback of a unique form $\omega \in \Omega^*(X)$ and in this case $\omega$ is an invariant form on $X$. So invariant forms on $X$ are in bijection with ${\rm Aff}(T^n)$-invariant forms on $P$.\\

Suppose that $V$ is a vector bundle on $M$ with flat connection. We can extend the definition of invariance to $\pi^*(V)$-valued forms. Indeed the bundle $f^*(V)$ on $P$ admits a natural action of the structure group ${\rm Aff}(T^n)$. Thus a form $\omega \in \Omega^*(X, \pi^*(V))$ will be called invariant if $p^*(\omega) \in \Omega^*(P,f^*(V))$ is ${\rm Aff}(T^n)$-invariant. Moreover the action of ${\rm Aff}(T^n)$ preserves the pullback connection on $f^*(V)$ so the twisted differential $d_{\pi^*(\nabla)} : \Omega^*(X,\pi^*(V)) \to \Omega^*(X,\pi^*(V))$ preserves the subspace $\Omega^*_{{\rm inv}}(X,\pi^*(V))$ of invariant forms. More generally given an invariant $3$-form $H \in \Omega^3(X)$ we see that the twisted differential $d_{\pi^*(\nabla),H}$ preserves the space of invariant forms. Let $H^*_{{\rm inv}}(X,(\pi^*(V),H))$ denote the associated cohomology. The inclusion $i : \Omega^*_{{\rm inv}}(X,\pi^*(V)) \to \Omega^*(X,\pi^*(V))$ induces a map $i_* : H^*_{{\rm inv}}(X,(\pi^*(V),H)) \to H^*(X,(\pi^*(V),H))$.\\

The main result of this section is the following:

\begin{proposition}\label{iso}
Let $\pi : X \to M$ be an affine torus bundle, $A$ a flat vector bundle on $M$ and $H$ an invariant $3$-form on $X$. The map $i_* : H^*_{{\rm inv}}(X,(\pi^*(A),H)) \to H^*(X,(\pi^*(A),H))$ is an isomorphism.
\end{proposition}

As an immediate corollary we find
\begin{corollary}\label{inv3}
For any class $h \in H^3(X,\mathbb{R})$ there is a representative $H \in \Omega^3(X)$ which is invariant.
\end{corollary}

\begin{proof}[Proof of Proposition \ref{iso}]
First we prove the result when $H=0$. Recall that for a fiber bundle $\pi : X \to M$ there is a natural filtration on differential forms, namely $F^p \Omega^n(X) = \{ \omega \in \Omega^n(X) \; | \; i_{V_1} i_{V_2} \dots i_{V_{n+1-p}} \omega = 0, \; \pi_*(V_1) = \dots = \pi_*(V_{n+1-p}) = 0 \}$. That is, $F^p \Omega^n(X)$ consists of $n$-forms that vanish when contracted with $(n+1-p)$ vertical vector fields. Then since $d( F^p \Omega^n(X)) \subseteq F^p \Omega^{n+1}$, there is an associated spectral sequence $(E_r^{p,q},d_r)$ which abuts to the cohomology $H^*(X,\mathbb{R})$. In fact, this spectral sequence is just the Leray-Serre spectral sequence with real coefficients \cite{grha}. 

If $(A,\nabla)$ is a flat vector bundle on $M$ then there is a similar filtration $\linebreak F^p \Omega^n(X,\pi^*(A))$ for the $\pi^*(A)$-valued forms on $X$ and again we get a spectral sequence $(E_r^{p,q},d_r)$ which abuts to the cohomology $H^*(X,\pi^*(A))$. We can think of this as the Leray-Serre spectral sequence with local coefficients.

Next we observe that there is a similar filtration $F^p \Omega^n_{{\rm inv}}(X,\pi^*(A))$ on the invariant forms, also preserved by the differential $d_\nabla$, so there is a similar spectral sequence $E_{r,{\rm inv}}^{p,q}$ which abuts to $H^*_{{\rm inv}}(X,\pi^*(A))$. The inclusion $i : \Omega^*_{{\rm inv}}(X,\pi^*(A)) \to \Omega^*(X,\pi^*(A))$ clearly preserves the filtration so induces a morphism $E_{r,{\rm inv}}^{p,q} \to E_r^{p,q}$ of spectral sequences. We claim that already at the $E_1$-stage the inclusion $E_{1,{\rm inv}}^{p,q} \to E_1^{p,q}$ is an isomorphism, so this is true for all $r \ge 1$. This means the inclusion $i_* : H^*_{{\rm inv}}(X,\pi^*(A)) \to H^*(X,\pi^*(A))$ is an isomorphism. Moreover we also see from this that the inclusion isomorphism $i_* : H^*_{{\rm inv}}(X,\pi^*(A)) \to H^*(X,\pi^*(A))$ respects the filtrations on $H^*_{{\rm inv}}(X,\pi^*(A)),H^*(X,\pi^*(A))$ induced by the filtrations on differential forms.\\

Let $V$ denote the vertical bundle of $\pi : X \to M$. In particular the tangent bundle of $X$ fits into a short exact sequence
\begin{equation*}
0 \to \pi^*(V) \to TX \to \pi^*(TM) \to 0.
\end{equation*}
From this one finds that there are natural identifications
\begin{eqnarray*}
F^p \Omega^{p+q}(X,\pi^*(A))/ F^{p+1} \Omega^{p+q}(X,\pi^*(A)) &\simeq& \Gamma( X , \pi^*( A \otimes \wedge^p T^*M \otimes \wedge^q V^*)), \\
F^p \Omega^{p+q}_{{\rm inv}}(X,\pi^*(A))/F^{p+1} \Omega^{p+q}_{{\rm inv}}(X,\pi^*(A)) &\simeq& \Gamma( M , A \otimes \wedge^p T^*M \otimes \wedge^q V^*),
\end{eqnarray*}
where we use $\Gamma(Y,W)$ to denote the global sections of a vector bundle $W \to Y$.

From this we can proceed to compute the $E_1$ stages of the spectral sequences. Essentially one takes cohomology in the vertical direction. Since the cohomology of the fibers identifies with the local systems $\wedge^q V^*$ on $M$ one finds easily that
\begin{equation*}
E_1^{p,q} = E_{1,{\rm inv}}^{p,q} = \Omega^p(M , A \otimes \wedge^q V^*),
\end{equation*}
which proves the claim.\\

Now let $H \in \Omega^3_{{\rm inv}}(X)$ be an invariant $3$-form on $X$. We have a filtration $F'^p = \Omega^{\ge p}(X,\pi^*(A))$ on $\Omega^*(X,\pi^*(A))$ which determines a spectral sequence $\{(E_r^p,d_r)\}_{r=1}^\infty$ described at the end of Section \ref{twco} which abuts to the twisted cohomology groups $H^*(X,(\pi^*(A),H))$. On the other hand we may define a similar filtration $F'^p_{{\rm inv}} = \Omega^{\ge p}_{{\rm inv}}(X,\pi^*(A))$ of the invariant forms on $X$. In the same manner we have a similar spectral sequence $\{ (E_{r,{\rm inv}}^p , d_r ) \}$ which abuts to the invariant twisted cohomology $H^*_{{\rm inv}}(X,(\pi^*(A),H))$. Moreover we see that the $E_2$ stage is given by the invariant cohomology $H^*_{{\rm inv}}(X,\pi^*(A))$ not twisted by $H$. The inclusion $i : \Omega^*_{{\rm inv}}(X,\pi^*(A)) \to \Omega^*(X,\pi^*(A))$ respects the filtrations and thus induces a morphism between spectral sequences. On the $E_2$ stage this is the inclusion $i_* : H^*_{{\rm inv}}(X,\pi^*(A)) \to H^*(X,\pi^*(A))$, which we have just shown to be an isomorphism. It follows immediately that $i_* : H^*_{{\rm inv}}(X,(\pi^*(A),H)) \to H^*(X,(\pi^*(A),H))$ is an isomorphism.
\end{proof}


\subsection{Gerbe connections}\label{grbconn}

In order to link the T-duality relation to a statement involving differential forms on $X,\hat{X}$, we need a way to relate gerbes to differential forms. This is achieved using the notions of gerbe connections, curvings and curvature. Let us review the relevant details, which we have adapted from standard references such as \cite{mm},\cite{stev},\cite{bcmms}.\\

Gerbe connections and curving are useful tools for doing differential geometry with bundle gerbes. Therefore we restrict attention to spaces $X$, which are smooth manifolds. Recall the notion of a quasi-cover $f : Y \to X$ from Section \ref{ggrb}. Here we restrict to quasi-covers such that $Y$ is a smooth manifold and $f$ is a surjective submersion. Note that every surjective submersion admits local sections. If $Y \to X$, $Y' \to X$ are two surjective submersions then the fiber product $Y \times_X Y'$ is a submanifold of $Y \times Y'$ and the projections $Y \times_X Y \to Y$, $Y \times_X Y' \to Y'$ are surjective submersions as well. This also applies to the iterated fiber products $Y^{[k]}$ of $Y$ and the projections $\partial_i : Y^{[k]} \to Y^{[k-1]}$ defined in Section \ref{ggrb}. The whole of Section \ref{ggrb} carries over to the smooth setting, so for instance we now take our bundle gerbes to be smooth, meaning they consist of a smooth line bundle $L \to Y^{[2]}$ and the gerbe multiplication $\theta_{x,y,z} : L_{x,y} \otimes L_{y,z} \to L_{x,z}$ is smooth in $(x,y,z)$.\\

Let $f: Y \to X$ be a surjective submersion. Define maps $\delta : \Omega^p(Y^{[k]}) \to \Omega^p(Y^{[k+1]})$ by
\begin{equation*}
\delta = \sum_{i=0}^k (-1)^i \partial_i^*.
\end{equation*}
In \cite{mm} it is shown that the sequence
\begin{equation}\label{gerbeseq}
0 \to \Omega^p(X) \buildrel f^* \over \longrightarrow \Omega^p(Y^{[1]}) \buildrel \delta^* \over \longrightarrow \Omega^p(Y^{[2]}) \buildrel \delta^* \over \longrightarrow \cdots
\end{equation}
is exact for any $p$.

\begin{definition}[\cite{mm},\cite{stev}]
Let $\mathcal{G} = (Y,f,L,\theta)$ be a bundle gerbe defined with respect to $Y$. A {\em (gerbe) connection} on $\mathcal{G}$ is a unitary connection $\nabla$ on $L$ such that the gerbe product $\theta : \partial^*_2 L \otimes \partial^*_0 \to \partial^*_1 L$ preserves the induced connections.
\end{definition}
Making use of (\ref{gerbeseq}), we find that every gerbe admits a gerbe connection. If $\nabla$ is a connection for $\mathcal{G} = (Y,f,L,\theta)$ then the curvature of $\nabla$ is a closed complex $2$-form $F \in \Omega^2( Y^{[2]}) \otimes \mathbb{C}$. The fact that $\nabla$ preserves $\theta$ ensures that $\delta(F) = 0$. Note also that since $\nabla$ is unitary, $F/2\pi i$ is real.
\begin{definition}[\cite{mm},\cite{stev}]
If $\nabla$ is a gerbe connection for the gerbe $\mathcal{G} = (Y,f,L,\theta)$ and $F \in \Omega^2(Y^{[2]}) \otimes \mathbb{C}$ the curvature, then a {\em curving} for $(\mathcal{G},\nabla)$ is a real $2$-form $B \in \Omega^2(Y)$ such that $2 \pi i \delta(B) = F$.

Note that $\delta(dB) = d \delta(B) = d (F/2 \pi i) = 0$, so by (\ref{gerbeseq}) there exists a unique $3$-form $H \in \Omega^3(X)$ such that $dB = f^*(H)$. We call $H$ the {\em curvature} of the curving $B$.
\end{definition}
Note that for every gerbe $\mathcal{G}$ with connection we can find a curving. Additionally if $H$ is the curvature of such a curving then $H$ is closed and the de Rham cohomology class $[H] \in H^3(X,\mathbb{R})$ is readily seen to be the image in real cohomology of the Dixmier-Douady class of $\mathcal{G}$.\\

Let $\mathcal{G} = (Y,f,L,\theta)$ be a graded gerbe, $\nabla$ a connection on $\mathcal{G}$ and $B$ a curving with curvature $H$. If $B_0 \in \Omega^2(X)$ is any $2$-form on $X$ then $B + f^*(B_0)$ is also a curving. The curvature of $B + f^*(B_0)$ is $H + dB_0$. Such a change in curving will be called a {\em curving shift} by $B_0$ or just a {\em shift} by $B_0$. Note that any two curvings for $(\mathcal{G},\nabla)$ are related by a {\em unique} shift.\\

If $\mathcal{G}$ is a graded gerbe then the definitions of connection, curving and curvature remain unaltered. Much of Section \ref{ggrb} can be adapted to the setting of graded bundle gerbes with connection and curving. Let us list the main points to be adapted:
\begin{itemize}
\item{Just as we have a pullback operation on gerbes, we can just as easily pullback connections and curvings.}
\item{There is a notion of strict isomorphism of gerbes with connections and curvings, namely a strict isomorphism of the underlying gerbes that preserves the connections and curvings.}
\item{Given a graded line bundle $L \to Y$ with connection $\nabla$, the trivializable gerbe $\delta(L)$ has natural choices for a connection and curving. The connection is just the induced connection on $\delta(L) = \partial_1^*(L) \otimes \partial_0^*(L^*)$, while the curving is $-F/2 \pi i$, where $F \in \Omega^1(Y) \otimes \mathbb{C}$ is the curvature of $L$. Notice that the curvature of such a curving is zero.}
\item{There is a natural way to define the product of two gerbes with connections and curvings.}
\end{itemize}

There is also a notion of stable isomorphism for gerbes with connections and curving. One simply adapts Definition \ref{defstabiso} to the present setting. Thus the graded line bundle $M \to Z$ in Definition \ref{defstabiso} now becomes a graded line bundle with connection and the strict isomorphism $\phi : r^*(L) \otimes \delta(M) \to r'^*(L')$ becomes a strict isomorphism of gerbes with connections and curvings.\\

Suppose $(\mathcal{G},\nabla,B)$, $(\mathcal{G}',\nabla',B')$ are two gerbes with connections and curvings. Let $\alpha : \mathcal{G} \to \mathcal{G}'$ be a stable isomorphism of the underlying gerbes. One can easily show that there exists a $2$-form $B_1$ on $X$ such that after shifting $B'$ by $B_1$, the stable isomorphism $\alpha$ can be promoted to a stable isomorphism of gerbes with connections and curvings. In particular if $H,H'$ are the curvatures of $B,B'$, then it follows that $H' = H + dB_1$.

Next suppose $\beta : \mathcal{G} \to \mathcal{G}'$ is a second stable isomorphisms of the underlying gerbes. Then we can find a second $2$-form $B_2 \in \Omega^2(X)$ such that after shifting $B'$ by $B_2$ we can promote $\beta$ to a stable isomorphism of gerbes with connections and curvings. Obviously we also have $H' = H + dB_2$, so in particular $B_2 - B_1$ is closed. Adapting the arguments of Section \ref{ggrb}, we see that after pulling back to a common refinement, the two isomorphisms differ by a graded line bundle with connection and using the descent isomorphism we get a graded line bundle with connection $(D,\nabla)$ on $X$, which up to line bundle isomorphism is independent of choices involved. Let $F_\nabla \in \Omega^2(X) \otimes \mathbb{C}$ be the curvature of this line bundle. Then by comparing curvings we find that $B_2 - B_1 = -F_\nabla /2 \pi i$. In particular, the cohomology class $[B_2 - B_1]$ equals $-c_1(D)$, where $c_1(D)$ is the first Chern class of the line bundle $D$.


\subsection{T-duality triples}\label{sstdt}

With the aid of gerbe connections and curvings, we can now translate topological T-duality into a statement involving differential forms. Let $(X,\mathcal{G}),(\hat{X},\hat{\mathcal{G}})$ be T-dual pairs where now we are in the setting of smooth manifolds and smooth gerbes. We can choose an isomorphism $\gamma : p^*(\mathcal{G}) \to \hat{p}^*( \hat{\mathcal{G}}) \otimes q^*(L(V))$ satisfying the Poincar\'e property.\\

Let $V$ be the vertical bundle of $X$. Recall that this is a flat vector bundle on $M$ such that $\pi^*(V)$ identifies with the vertical tangent bundle ${\rm Ker}(\pi_*)$, which has a flat structure inherited from the affine structure on $X$. A {\em connection} on $X$ is defined to be an invariant $1$-form $A \in \Omega^1_{{\rm inv}}(X,\pi^*(V))$ with values in $\pi^*(V)$, such that the restriction of $A$ to the vertical tangent bundle ${\rm Ker}(\pi_*)$ induces an isomorphism of flat vector bundles $A : {\rm Ker}(\pi_*) \to \pi^*(V)$. Let $P \to M$ be the principal ${\rm Aff}(T^n)$-bundle such that $X = P \times_{{\rm Aff}(T^n)} T^n = P/{\rm GL}(n,\mathbb{Z})$ and let $a : P \to X$ be the projection. It follows that $a^*(A)$ is a connection on $P$, in fact it is straightforward to see that in this manner connections on $X$ are in bijection with connections on $P$. Given a connection $A$ on $X$, there exists a unique $V$-valued $2$-form $F \in \Omega^2(M,V)$ on $M$ such that $d_\nabla A = \pi^*(F)$, where $\nabla$ is the flat connection on $V$. We call $F$ the {\em curvature} of $A$. Recall that the constant sections of $V$ can be identified with the local system $\Lambda \otimes \mathbb{R}$, so the change of coefficient map $\Lambda \to \Lambda \otimes \mathbb{R}$ induces a map $H^2(M,\Lambda) \to H^2(M,V)$. It is straightforward to see that the de Rham cohomology class of $F$ in $H^2(M,V)$ coincides with the image of the twisted Chern class $c \in H^2(M,\Lambda)$ under $H^2(M,\Lambda) \to H^2(M,V)$.\\

Let $A,\hat{A}$ be connections on $X,\hat{X}$. Since $X,\hat{X}$ are T-dual their vertical bundles $V,\hat{V}$ are dual, so we can think of $A \in \Omega^1(X,\pi^*(V))$ and $\hat{A} \in \Omega^1(\hat{X},\hat{\pi}^*(V^*))$. Let $F,\hat{F}$ be the corresponding curvatures, so $F \in \Omega^2(M,V)$, $\hat{F} \in \Omega^2(M, V^*)$. Let $( \; , \; )$ denote the pairing of $V$ and $V^*$ and extend this to a pairing $( \; \buildrel \wedge \over , \; )$ of differential forms valued in $V,V^*$. 

\begin{theorem}\label{tdtthm}
Let $(X,\mathcal{G}),(\hat{X},\hat{\mathcal{G}})$ be T-duals, $h \in H^3(X,\mathbb{Z}),\hat{h} \in H^3(\hat{X},\mathbb{Z})$ the ungraded Dixmier-Douady classes of $\mathcal{G},\hat{\mathcal{G}}$ and $h_\mathbb{R} \in H^3(X,\mathbb{R}),\hat{h}_\mathbb{R} \in H^3(\hat{X},\mathbb{R})$ the images of $h,\hat{h}$ in real cohomology. Let $A,\hat{A}$ be connections on $X,\hat{X}$, with curvatures $F,\hat{F}$. For any choice of $A,\hat{A}$ there exists a $3$-form $H_3 \in \Omega^3(M)$ such that
\begin{equation}\label{tdt1}
dH_3 + (F \buildrel \wedge \over , \hat{F}) = 0
\end{equation}
and such that $h_\mathbb{R},\hat{h}_\mathbb{R}$ are represented in de Rham cohomology by
\begin{eqnarray}
H &=& \pi^*(H_3) + ( A \buildrel \wedge \over , \pi^*(\hat{F})), \label{tdt2} \\
\hat{H} &=& \hat{\pi}^*(H_3) + ( \hat{A} \buildrel \wedge \over , \hat{\pi}^*(F) ). \label{tdt3}
\end{eqnarray}
In addition we have
\begin{equation}\label{hhhatrel}
\hat{p}^*(\hat{H}) - p^*(H) = d \mathcal{B},
\end{equation}
where
\begin{equation}\label{calb}
\mathcal{B} = ( p^*(A) \buildrel \wedge \over , \hat{p}^*(\hat{A}) ).
\end{equation}
\end{theorem}

If $(X,\mathcal{G}),(\hat{X},\hat{\mathcal{G}})$ are T-duals then a choice of connections $A,\hat{A}$ and $3$-form $H_3$ on $M$ satisfying (\ref{tdt1}),(\ref{tdt2}),(\ref{tdt3}) will be called a {\em T-duality triple} and denoted $(A,\hat{A},H_3)$. The rest of this section is dedicated to the proof of Theorem \ref{tdtthm}.

\begin{lemma}
We may choose de Rham representatives $H \in \Omega^3(X)$, $\hat{H} \in \Omega^3(\hat{X})$ for $h_\mathbb{R},\hat{h}_\mathbb{R}$ such that $H,\hat{H}$ are invariant and
\begin{eqnarray}
H &=& \pi^*(H_3) + (A \buildrel \wedge \over , \pi^*(\hat{F}) ), \label{hstruct1} \\
\hat{H} &=& \hat{\pi}^*(\hat{H}_3) + (\hat{A} \buildrel \wedge \over , \hat{\pi}^*(F) ), \label{hstruct2}
\end{eqnarray}
for some $H_3,\hat{H}_3 \in \Omega^3(M)$.
\end{lemma}
\begin{proof}
It suffices to show this for $h_\mathbb{R}$, since the proof is identical for $\hat{H}$. Recall that by Proposition \ref{iso} the inclusion $\Omega^*_{{\rm inv}}(X) \to \Omega^*(X)$ of invariant forms induces an isomorphism in (untwisted) cohomology $H^*_{{\rm inv}}(X,\mathbb{R}) \to H^*(X,\mathbb{R})$. As in the proof of Proposition \ref{iso} recall that there is a natural filtration on differential forms, $F^p \Omega^n(X)$ and a similar filtration $F^p \Omega^n_{{\rm inv}}(X)$ on invariant forms. The inclusion $i : \Omega^*_{{\rm inv}}(X) \to \Omega^*(X)$ clearly respects the filtrations and as shown in the proof of Proposition \ref{iso}, the inclusion isomorphism $i_* : H^*_{{\rm inv}}(X,\mathbb{R}) \to H^*(X,\mathbb{R})$ preserves the induced filtrations on $H^*_{{\rm inv}}(X,\mathbb{R}),H^*(X,\mathbb{R})$.\\

Recall that $h$, the ungraded Dixmier-Douady class of $\mathcal{G}$ must lie in the term $F^{2,3}(\pi,\mathbb{Z})$ of the filtration on $H^3(X,\mathbb{Z})$. Accordingly $h_\mathbb{R}$ lies in $F^{2,3}(\pi,\mathbb{R})$. From the above argument this means we can find a representative for $h_\mathbb{R}$ which lies in $F^2 \Omega^3_{{\rm inv}}(X)$. Let $H'$ be such a representative. Using the connection $A$ we know that $H'$ has a decomposition of the form
\begin{equation*}
H' = \pi^*(a) + (A \buildrel \wedge \over , \pi^*(b) )
\end{equation*}
for some $a \in \Omega^3(M), b \in \Omega^2(M,V^*)$. Note that since $dH' = 0$, we must have $d_\nabla b = 0$, where $\nabla$ denotes flat connection on $V^*$. Thus $b$ defines a class in $H^2(M,V^*)$. In fact this class is a representative in $E_2^{2,1}(\pi,\mathbb{R})$ for the projection of $h_\mathbb{R}$ to $E_\infty^{2,1}(\pi,\mathbb{R})$. From Proposition \ref{toprelations} we deduce that the class of $b$ agrees the class of $\hat{c}$ as elements of $E_\infty^{2,1}(\pi,\mathbb{R}) = E_3^{2,1}(\pi,\mathbb{R})$. Note however that $\hat{c}$ is represented by $\hat{F}$ in $H^2(M,V^*)$, so we may conclude that in $E_2^{2,1}(\pi,\mathbb{R})$ we have $b = \hat{F} + d_2 \beta$, for some $\beta \in E_2^{0,2}(\pi,\mathbb{R}) = H^0(M, \wedge^2 V^*)$. We may think of $\beta$ as a covariantly constant section of $\wedge^2 V^*$. If we contract twice with the connection $A$ we get a $2$-form $B_0$ such that $B_0$ is invariant, $dB_0 \in F^2 \Omega^3_{{\rm inv}}(X)$ and $dB_0$ has the form $dB_0 = \pi^*(w) + ( A \buildrel \wedge \over , \pi^*(w'))$, for some $w \in \Omega^3(M)$, $w' \in \Omega^2(M,V^*)$ such that $w'$ coincides with $d_2 \beta$ at the $E_2$ stage. Thus $b$ and $\hat{F} + w'$ coincide at the $E_2$ stage and their difference at the $E_1$ stage has the form $d_1 \tau$ for some $\tau \in E_1^{1,1}(\pi,\mathbb{R}) = \Omega^1(M , V^*)$. Contracting $\tau$ with the connection $A$ we get a $2$-form $B_1$ such that $B_1$ is invariant, $dB_1 \in F^2 \Omega^3_{{\rm inv}}(X)$ and $dB_1$ has the form $dB_1 = \pi^*(v) + ( A \buildrel \wedge \over , \pi^*(v'))$, for some $v \in \Omega^3(M)$, $v' \in \Omega^2(M,V^*)$ where $v'$ agrees with $d_1 \tau$ in $E_1^{2,1}(\pi,\mathbb{R})$. Thus $b$ and $\hat{F} + w' + v'$ coincide as elements of $E_1^{2,1}(\pi,\mathbb{R}) = \Omega^2(M,V^*)$, so that in fact $b = \hat{F} + w' + v'$.

We conclude that there exists some invariant $2$-form $B$ (in fact $B = -B_2 - B_1$ suffices) such that $H = H' + dB$ has the form
\begin{equation*}
H = \pi^*(H_3) + ( A \buildrel \wedge \over , \pi^*( \hat{F} )).
\end{equation*}
This proves the result for $h_\mathbb{R}$.
\end{proof}

Let us now choose $H,\hat{H}$ of the form given in Equations (\ref{hstruct1}),(\ref{hstruct2}). It is straightforward to see that $\mathcal{G},\hat{\mathcal{G}}$ can be given connections and curvings in such a way that $H,\hat{H}$ are the corresponding curvatures. In fact one could just make any choice of connections and curvings and then shift the curvings by appropriate $2$-forms on $M$. 

The lifting gerbe $L(V)$ admits a canonical flat connection, where by a flat gerbe connection we simply mean that the connection on the underlying line bundle of the gerbe is flat. If a gerbe admits a flat connection then we can choose the curving and curvature both to be zero. Recall that the line bundle that underlies the lifting gerbe $L(V)$ is the line bundle associated to a principal circle bundle, which is a pullback of the principal circle bundle ${\rm Pin}^c(n) \to {\rm O}(n)$. However ${\rm Pin}^c(n) \to {\rm O}(n)$ is the principal ${\rm U}(1)$-bundle associated to a principal $\mathbb{Z}_2$-bundle, namely ${\rm Pin}_+(n) \to {\rm O}(n)$ (alternatively one could use ${\rm Pin}_-(n)$, they both give rise to the same group ${\rm Pin}^c(n)$). Thus if we think of ${\rm Pin}^c(n) \to {\rm O}(n)$ as a principal circle bundle, we see that it has a canonical flat connection, hence so does the lifting gerbe $L(V)$.\\

From the discussion at the end of Section \ref{grbconn} there exists a $2$-form $D \in \Omega^2(C)$ such that after shifting the curving of $\hat{p}^*(\hat{\mathcal{G}})$ by $D$ we can promote $\gamma$ to a stable isomorphism of gerbes with connections and curvings. In this case we also have $\hat{p}^*(\hat{H}) = p^*(H) + dD$. If we restrict to the fibers of $X,\hat{X}$, then $H,\hat{H}$ are invariant forms which are trivial in cohomology. Since the fibers are tori this means $H,\hat{H}$ actually vanish on the fibers. Then $D$ restricted to the fibers is a closed $2$-form. For any $m \in M$ let $T_m = \pi^{-1}(m)$, $\hat{T}_m = \hat{\pi}^{-1}(m)$ be the fibers of $X,\hat{X}$ over $m$. So $D|_{T_m \times \hat{T}_m}$ defines a de Rham cohomology class $[D|_{T_m \times \hat{T}_m }] \in H^2(T_m \times \hat{T}_m , \mathbb{R})$ and we can further project this to a class $\rho_m \in H^1(T_m , \mathbb{R}) \otimes H^1(\hat{T}_m ,\mathbb{R})$. Going back to the discussion at the end of Section \ref{grbconn} again and considering that $\gamma$ satisfies the Poincar\'e property, one sees that $\rho \in H^1(T_m , \mathbb{R}) \otimes H^1(\hat{T}_m ,\mathbb{R}) \simeq Hom( H_1(T_m , \mathbb{R}) , H^1(\hat{T}_m , \mathbb{R}) )$ is just the image in real coefficients of the isomorphism $\delta_m \in Hom( H_1(T_m , \mathbb{Z}) , H^1(\hat{T}_m , \mathbb{Z}))$ determined by $\gamma$.\\

Let us define $\mathcal{B} \in \Omega^2(C)$ to be
\begin{equation}\label{defofb}
\mathcal{B} = ( p^*(A) \buildrel \wedge \over , \hat{p}^*(\hat{A})).
\end{equation}
Observe that $\mathcal{B}$ has the property that for all $m \in M$, $\mathcal{B}|_{T_m \times \hat{T}_m}$ is a closed $2$-form on the fiber $T_m \times \hat{T}_m$ and that the cohomology class $[ \mathcal{B}|_{T_m \times \hat{T}_m } ] \in H^2(T_m \times \hat{T}_m , \mathbb{R})$ agrees with $\delta_m \in H^1(T_m , \mathbb{R}) \otimes H^1(\hat{T}_m , \mathbb{R})$, where we think of $H^1(T_m , \mathbb{R}) \otimes H^1(\hat{T}_m , \mathbb{R})$ as a subgroup of $H^2(T_m \times \hat{T}_m , \mathbb{R})$.

From Equations (\ref{hstruct1}),(\ref{hstruct2}) and the (\ref{defofb}) we immediately find
\begin{equation*}
\hat{p}^*(\hat{H}) - p^*(H) = d \mathcal{B} + q^*( \hat{H}_3 - H_3).
\end{equation*}
We have already seen however that $\hat{p}^*(\hat{H}) - p^*(H) = dD$. We conclude that $q^*( \hat{H}_3 - H_3 ) = d(D - \mathcal{B})$.

\begin{lemma}\label{absorbdiff}
Let $\omega \in \Omega^3(M)$ be a closed $3$-form on $M$ and $\alpha \in \Omega^2(C)$ a $2$-form on $C$ such that $q^*(\omega) = d\alpha$. For any $m \in M$, let $T_m \times \hat{T}_m$ denote the fiber of $C$ over $m$. We have that $q^*(\omega)|_{T_m \times \hat{T}_m} = 0$, so $\alpha |_{T_m \times \hat{T}_m}$ is closed. Let $a_m = [\alpha|_{T_m \times \hat{T}_m}] \in H^2(T_m \times \hat{T}_m , \mathbb{R})$ be the corresponding cohomology class. Suppose that the projection of $a_m$ to $H^1(T_m , \mathbb{R}) \otimes H^1(\hat{T}_m , \mathbb{R})$ is zero for all $m \in M$. Then there exists invariant $2$-forms $B \in \Omega^2_{{\rm inv}}(X)$, $\hat{B} \in \Omega^2_{{\rm inv}}(\hat{X})$ such that $q^*(\omega) = p^*(dB) - \hat{p}^*(d\hat{B})$.
\end{lemma}
\begin{proof}
First note that we may assume that $\alpha$ is invariant without changing the fact that the class $a_m$ projects to zero in $H^1(T_m , \mathbb{R}) \otimes H^1(\hat{T}_m , \mathbb{R})$. Choose connections $A,\hat{A}$ on $X,\hat{X}$ which then defines a connection $A \oplus \hat{A}$ on $C$, since $C$ is the fiber product of $X$ and $\hat{X}$. Use this connection to decompose $\alpha$ into forms on $M$. On decomposing we get terms $\alpha_2 \in \Omega^2(M)$, $\alpha_1 \in \Omega^1(M,V^*)$, $\hat{\alpha}_1 \in \Omega^1(M,V)$, $\alpha_{0} \in \Omega^0(M,\wedge^2 V^*)$, $\hat{\alpha}_0 \in \Omega^0(M,\wedge^2 V)$ and lastly $\tilde{\alpha}_0 \in \Omega^0(M , V^* \otimes V)$. For instance $\alpha_2$ contributes a term $q^*(\alpha_2)$, $\alpha_1$ contributes a term $( p^*(A) \buildrel \wedge \over , q^*(\alpha_1))$ and so on. Notice that the vanishing condition on the $a_m$ is precisely the condition that $\tilde{\alpha}_0 = 0$. Now all the remaining terms can be thought of as pullbacks from $M,X$ or $\hat{X}$, namely $\alpha_2$ from $M$, $\alpha_1,\alpha_0$ from $X$ and $\hat{\alpha}_1,\hat{\alpha}_0$ from $\hat{X}$. The result follows.
\end{proof}

Since $D$ and $\mathcal{B}$ restrict on the fibers to the same cohomology class in $H^1(T_m , \mathbb{R}) \otimes H^1(\hat{T}_m , \mathbb{R})$ we may apply Lemma \ref{absorbdiff}. We find that there exists invariant $2$-forms $B \in \Omega^2_{{\rm inv}}(X)$, $\hat{B} \in \Omega^2_{{\rm inv}}(\hat{X})$ such that $q^*( \hat{H}_3 - H_3 ) = p^*(dB) - \hat{p}^*(d\hat{B})$. If we now shift the curvings of $\mathcal{G},\hat{\mathcal{G}}$ by $B,\hat{B}$ and replace $H,\hat{H}$ by $H+dB,\hat{H}+d\hat{B}$, we arrive at Theorem \ref{tdtthm}. Note that Equation (\ref{tdt1}) holds since $H$ is closed.\\

In fact, this argument has actually shown something a little stronger than the existence of T-duality triples:
\begin{proposition}\label{linkgerbeh}
Let $(X,\mathcal{G}),(\hat{X},\hat{\mathcal{G}})$ be T-dual pairs and $\gamma : p^*(\mathcal{G}) \to \hat{p}^*(\hat{\mathcal{G}}) \otimes q^*(L(V))$ an isomorphism satisfying the Poincar\'e property. Choose conections $A,\hat{A}$ on $X,\hat{X}$. Then there exists connections and curvings for $\mathcal{G},\hat{\mathcal{G}}$ and a $3$-form $H_3$ on $M$ such that
\begin{itemize}
\item{$(A,\hat{A},H_3)$ a T-duality triple,}
\item{the $3$-forms $H,\hat{H}$ in (\ref{tdt2}),(\ref{tdt3}) are the curvatures of the curvings,}
\item{there is a $2$-form $D \in \Omega^2(C)$ so that after shifting the curving of $\hat{p}^*(\hat{\mathcal{G}})$ by $D$, we can promote $\gamma$ to a stable isomorphism of gerbes with connections and curvings,}
\item{define $\mathcal{B}$ as in (\ref{calb}). Then $D - \mathcal{B}$ is exact.}
\end{itemize}
\end{proposition}
\begin{proof}
Recall the the proof of Lemma \ref{absorbdiff}. If $\alpha = D - \mathcal{B}$ is not invariant we can replace it by an invariant $2$-form $\alpha' = D' - \mathcal{B}$ representing the same cohomology class. We found that by shifting the curvings of $\mathcal{G},\hat{\mathcal{G}}$ appropriately we could absorb away the class $\alpha'$. Thus after making these shifts we can assume $D' = \mathcal{B}$. But $\alpha - \alpha'$ is exact hence so is $D-D' = D - \mathcal{B}$.
\end{proof}


\section{T-duality transforms}\label{stdtrans}

In this section we show that that the T-duality relation on pairs $(X,\mathcal{G}),(\hat{X},\hat{\mathcal{G}})$ implies isomorphisms between certain topological invariants and geometric structures on $X,\hat{X}$. Namely we consider twisted cohomology, twisted $K$-theory and Courant algebroids. The results in this section are well known in the case of principal torus bundles, so what we have accomplished here is to show that these isomorphisms persist in the presence of monodromy.


\subsection{Twisted cohomology}\label{tdttwico}

Let $(X,\mathcal{G}),(\hat{X},\hat{\mathcal{G}})$ be T-duals. Choose a T-duality triple $(A,\hat{A},H_3)$ and define $H,\hat{H},\mathcal{B}$ as in (\ref{tdt2}),(\ref{tdt3}),(\ref{calb}). Let $\xi,\hat{\xi} \in H^2(M,\mathbb{Z}_2)$ be the grading classes associated to $\mathcal{G},\hat{\mathcal{G}}$. We can associate to $\xi,\hat{\xi}$ flat line bundles $\mathbb{R}_\xi,\mathbb{R}_{\hat{\xi}}$. From this data we may define twisted cohomology groups $H^*(X,( \pi^*(\mathbb{R}_\xi) , H))$, $H^*(\hat{X}, (\hat{\pi}^*(\mathbb{R}_{\hat{\xi}}) , \hat{H}))$ which up to isomorphism depend only on $(X,\mathcal{G}),(\hat{X},\hat{\mathcal{G}})$. To simplify the notation let us write $H^*(X,(\xi,H))$, $H^*(\hat{X},(\hat{\xi},\hat{H}))$ for these groups. We will show that the T-duality triple $(A,\hat{A},H_3)$ determines an isomorphism $T : H^*(X,(\xi,H)) \to H^{*-n}(X,(\hat{\xi},\hat{H}))$, where $n$ is the rank of the fibers of $X,\hat{X}$.\\

Let $\Omega^*(X,\xi)$ denote the space of differential forms on $X$ with values in $\pi^*(\mathbb{R}_\xi)$ and similarly define $\Omega^*(\hat{X},\hat{\xi})$. We also let $d_{\xi,H}$, $d_{\hat{\xi},\hat{H}}$ denote the corresponding twisted differentials. We define a T-duality transformation $T : \Omega^*(X,\xi) \to \Omega^{*-n}(\hat{X},\hat{\xi})$ as follows:
\begin{equation}\label{tdtdef}
T \omega = \int_{C/\hat{X}} e^{-\mathcal{B}} p^*(\omega),
\end{equation}
where we use $\int_{C/\hat{X}}$ to denote integration over the fibers of $C \to \hat{X}$. The expression (\ref{tdtdef}) is also known as the {\em Hori formula} \cite{hor},\cite{bem},\cite{bhm1}. It can be understood as the twisted de Rham version of the Fourier-Mukai transform of \cite{muk}. This transformation has also been proposed in the context of mirror symmetry \cite{yau}.

Note that even though the fibers of $C \to \hat{X}$ need not be oriented it is still possible to define the fiber integration as a map $\int_{C/\hat{X}} : \Omega^*(C , \xi ) \to \Omega^{*-n}(\hat{X},\hat{\xi})$. To see this note that by (\ref{tc1}) we have $\hat{\xi} = \xi + w_1(V)$, where $w_1(V)$ is the first Stiefel-Whitney class of the vertical bundle of $X \to M$. But the vertical bundle of $C \to \hat{X}$ is just $\hat{p}^*(V)$, so has first Stiefel-Whitney class $\hat{p}^*(w_1(V))$. The lack of a consistent fiber orientation for $C \to \hat{X}$ means that the fiber integration has a sign ambiguity, but this is exactly accounted for by the orientation local system $\mathbb{R}_{w_1(V)}$. Using (\ref{hhhatrel}) one easily sees that $d_{\hat{\xi},\hat{H}}( T \omega) = T( d_{\xi,H} \omega)$, that is $T$ is a chain map. We denote also by $T$ the induced map $T : H^*(X,(\xi,H)) \to H^{*-n}(\hat{X},(\hat{\xi},\hat{H}))$ in cohomology.
\begin{theorem}\label{ttdttwcoiso}
The map $T : H^*(X,(\xi,H)) \to H^{*-n}(\hat{X},(\hat{\xi},\hat{H}))$ in cohomology induced by the map $T : \Omega^*(X,\xi) \to \Omega^{*-n}(\hat{X},\hat{\xi})$ given in (\ref{tdtdef}) is an isomorphism.
\end{theorem}
\begin{proof}
We see immediately that $T : \Omega^*(X,\xi) \to \Omega^{*-n}(\hat{X},\hat{\xi})$ sends invariant forms to invariant forms and so determines a map $T : H^*_{{\rm inv}}(X,(\xi,H)) \to H^{*-n}_{{\rm inv}}(\hat{X},(\hat{\xi},\hat{H}))$ between invariant twisted cohomology. Clearly we have a commutative diagram
\begin{equation*}\xymatrix{
H^*(X,(\xi,H)) \ar[r]^-T & H^{*-n}(\hat{X},(\hat{\xi},\hat{H})) \\
H^*_{{\rm inv}}(X,(\xi,H)) \ar[r]^-T \ar[u]^{i_*} & H^{*-n}_{{\rm inv}}(\hat{X},(\hat{\xi},\hat{H})) \ar[u]^{i_*}
}
\end{equation*}
where $i_*$ denotes the maps induced by inclusion of invariant forms. By Proposition \ref{iso} we need only show that $T : H^*_{{\rm inv}}(X,(\xi,H)) \to H^{*-n}_{{\rm inv}}(\hat{X},(\hat{\xi},\hat{H}))$ is an isomorphism.\\

The connection $A$ determines an isomorphism $TX = \pi^*(TM \oplus V)$ and similarly the connection $\hat{A}$ gives an isomorphism $T\hat{X} = \hat{\pi}^*(TM \oplus V^*)$. The invariant forms on $X,\hat{X}$ are then naturally identified with sections of certain bundles on $M$. Specifically we have
\begin{eqnarray*}
\Omega^*_{{\rm inv}}(X,\xi) &=& \Gamma(M , \mathbb{R}_\xi \otimes \wedge^*( T^*M \oplus V^*) ), \\
\Omega^*_{{\rm inv}}(\hat{X},\hat{\xi}) &=& \Gamma( M , \mathbb{R}_{\hat{\xi}} \otimes \wedge^*( T^*M \oplus V)).
\end{eqnarray*}
It is easy to see that $T : \Gamma(M , \mathbb{R}_\xi \otimes \wedge^*( T^*M \oplus V^*) ) \to \Gamma( M , \mathbb{R}_{\hat{\xi}} \otimes \wedge^*( T^*M \oplus V))$ is $\mathcal{C}^\infty(M)$-linear, so corresponds to a bundle map $T : \mathbb{R}_\xi \otimes \wedge^*( T^*M \oplus V^*) \to \mathbb{R}_{\hat{\xi}} \otimes \wedge^*( T^*M \oplus V)$. It suffices to show that $T$ is a bundle isomorphism.\\

Let $e_1 , \dots , e_n$ denote a local integral frame for $V$ (that is a local frame for $\Lambda$), $\hat{e}_1 , \dots , \hat{e}_n$ for $\hat{V} = V^*$ and let $e^1 , \dots , e^n$, $\hat{e}^1 , \dots , \hat{e}^n$ be the dual frames. Let $I = \{ i_1 , \dots , i_k \}$ be an ordered collection of indices $i_1 < \dots < i_k$ and use $e^I$ to denote $e^{i_1} \wedge e^{i_2} \wedge \dots \wedge e^{i_k}$. We also let $|I| = k$ denote the cardinality of $I$. Locally $\mathcal{B}$ corresponds to $e^1 \wedge \hat{e}^1 + \dots + e^n \wedge \hat{e}^n$. One finds that
\begin{equation*}
e^{-\mathcal{B}} = \sum_{k=0}^n (-1)^{\frac{1}{2}k(k+1)} \sum_{|I| = k} e^I \wedge \hat{e}^I.
\end{equation*}
From this we find that $e^{-\mathcal{B}} e^I$ consists of $(-1)^{\epsilon} (i_{\hat{e}_I} \hat{e}^1 \wedge \dots \wedge \hat{e}^n) \wedge e^1 \wedge \dots \wedge e^n$ plus terms that are lower degree in $e^1, \dots , e^n$. Here $i_{\hat{e}_I} \hat{e}^1 \wedge \dots \wedge \hat{e}^n$ denotes the multiple contraction $i_{\hat{e}_{i_1}} \dots i_{\hat{e}_{i_k}} (\hat{e}^1 \wedge \dots \wedge \hat{e}^n)$ and $\epsilon = \frac{1}{2}n(n-1)$. Now if $\omega$ is a local section of $\wedge^* T^*M$ one finds
\begin{equation}\label{tdtform}
T( \omega \wedge e^I ) = (-1)^\epsilon \omega \wedge i_{\hat{e}_I} (\hat{e}^1 \wedge \dots \hat{e}^n),
\end{equation}
showing that $T$ is indeed a bundle isomorphism.
\end{proof}

Next we will show that for compact oriented manifolds the twisted cohomology groups are equipped with natural bilinear forms which are preserved (up to sign) by T-duality.\\

Let $\Omega^*(X,\xi)$ denote the differential forms on $X$ twisted by a flat line bundle $(\mathbb{R}_\xi,\nabla)$ corresponding to a class $\xi \in H^1(X,\mathbb{Z}_2)$. Define an involution $\sigma : \Omega^*(X,\xi) \to \Omega^*(X,\xi)$ by $\sigma(\omega) = (-1)^{\frac{1}{2}k(k-1)}\omega$, where $\omega$ has degree $k$. Also let $(-1)^{\omega}$ be $1$ or $-1$ according to whether $\omega$ is even or odd. Now if $H$ is any closed $3$-form we find that for all $\omega$
\begin{equation*}
\sigma d_{\xi,H} \sigma \omega = -(-1)^{\omega} d_{\xi,-H}\omega.
\end{equation*}
It follows from this that $\sigma$ descends to an isomorphism $\sigma  : H^*(X,(\xi,H)) \to H^*(X,(\xi,-H))$ at the level of twisted cohomology.

Observe that if $H,H'$ are any two closed $3$-forms and $\xi,\xi' \in H^1(X,\mathbb{Z}_2)$ then the wedge product of forms descends to a product $\wedge : H^a(X,(\xi,H)) \otimes H^b(X,(\xi',H')) \to H^{a+b}(X,(\xi+\xi',H+H'))$. With these ingredients we can now define a bilinear form $\langle \, , \, \rangle$ on $H^*(X,(\xi,H))$ as follows. For $a,b \in H^*(X,(\xi,H))$ we define
\begin{equation*}
\langle a , b \rangle = \int_X \sigma(a) \wedge b.
\end{equation*}
Note if $\mathbb{R}_\xi$ is the flat line bundle corresponding to $\xi \in H^1(X,\mathbb{Z}_2)$ then $\mathbb{R}_\xi \otimes \mathbb{R}_\xi$ has a canonical trivialization, so we can regard $\sigma(a) \wedge b$ is an element of ordinary de Rham cohomology $H^*(X)$. The integration is then just the ordinary integration of the top degree part of $\sigma(a) \wedge b$. We call $\langle \, , \, \rangle$ the {\em Mukai pairing} on $H^*(X,(\xi,H))$, a twisted generalization of the pairing used by Mukai \cite{muk2},\cite{muk3}. This pairing, or rather the pairing at the level of bundles $(\mathbb{R}_\xi \otimes \wedge^* T^*X) \otimes (\mathbb{R}_\xi \otimes \wedge^* T^*X) \to {\rm det}(T^*M)$ is also used in generalized geometry \cite{gual} under the same name. If $X$ has dimension $d$ then $\langle a , b \rangle = (-1)^{\frac{1}{2}d(d-1)}\langle b , a \rangle$.

\begin{proposition}\label{mptc}
Let $M$ be compact of dimension $m$ and suppose $(X,\mathcal{G}),(\hat{X},\hat{\mathcal{G}})$ are rank $n$, oriented T-dual pairs over $M$. Let $(A,\hat{A},H_3)$ be a T-duality triple so that the transform $T : H^*(X,(\xi,H)) \to H^{*-n}(\hat{X},(\hat{\xi},\hat{H}))$ is defined. Choose orientations on $M$ and the vertical bundle of $X$. We get corresponding orientations on $X,\hat{X}$ and use these orientations to define Mukai pairings on $X,\hat{X}$. Then for all $a,b \in H^*(X,(\xi,H))$ we have
\begin{equation*}
\langle Ta , Tb \rangle = (-1)^{nm} \langle a , b \rangle.
\end{equation*}
\end{proposition}
\begin{proof}
By Proposition \ref{iso} we may represent twisted cohomology classes by invariant forms. Now using Equation (\ref{tdtform}) the result follows by a straightforward calculation.
\end{proof}


\subsection{Twisted $K$-theory}\label{tdttkt}

Twisted $K$-theory has its origins in \cite{dk}, where it is defined in terms of finite dimensional algebra bundles. It was later realized that passing to infinite dimensional algebra bundles allowed a definition of twisted $K$-theory that includes the case of non-torsion degree $3$ classes \cite{ros}. Still more recently a great deal of interest has been generated from string theory, where twisted $K$-theory is understood to classify D-brane charge and Ramond-Ramond flux in the presence of non-trivial $H$-flux \cite{boma}.\\

There are various ways to define twisted $K$-theory and corresponding models for the category of twists. For instance one can describe twists as principal ${\rm PU}(\mathcal{H})$-bundles, where $\mathcal{H}$ is a separable Hilbert space, or as maps $X \to K(\mathbb{Z},3)$ using the fact that $K(\mathbb{Z},3)$ can be identified with ${\bf B}{\rm PU}(\mathcal{H})$. Twisted $K$-theory can then be defined as the $K$-theory of an associated $C^*$-algebra of sections of a bundle of compact operators \cite{ros},\cite{boma}. An alternative approach adopted in \cite{atseg1},\cite{atseg2} is to define twisted $K$-theory in terms of homotopy classes of sections of an associated bundle of Fredholm operators. In \cite{bcmms} the approach to twisted $K$-theory using bundle gerbes was initiated. This point of view is closely related to the previous approaches except that the twists are now described using bundle gerbes. However we would like to use {\em graded} bundle gerbes which requires a slightly more general version of twisted $K$-theory. For this we can turn to Freed, Hopkins and Teleman \cite{fht}. Their definition of twisted $K$-theory involves graded central extensions instead of graded bundle gerbes, but these two notions are very similar.\\

For a space $X$ and graded bundle gerbe $\mathcal{G}$ on $X$ we can then define twisted $K$-theory groups, denoted $K^i(X,\mathcal{G})$ where $i \in \mathbb{Z}_2$. Most of the properties of twisted $K$-theory that we use can be found within \cite{fht},\cite{cw0},\cite{cw},\cite{bcmms},\cite{atseg1}. We briefly summarize the main features of twisted $K$-theory for our purposes, additional details are given in Appendix \ref{grgtwkt}:
\begin{itemize}
\item{$K^*( X , \; \; ) : GrGrb(X) \to Gr_{\mathbb{Z}_2}{\rm Ab}$ is a functor from the category of graded gerbes on $X$ to the category of $\mathbb{Z}_2$-graded abelian groups.}
\item{There is a product structure $K^i(X,\mathcal{G}) \otimes K^j(X,\mathcal{H}) \to K^{i+j}(X,\mathcal{G} \otimes \mathcal{H})$ which is graded commutative and associative in the evident sense.}
\item{When $X$ is compact, a graded line bundle $L \to X$ defines a class $[L] \in K^0(X)$. When two stable isomorphisms $\alpha,\beta : \mathcal{G} \to \mathcal{H}$ differ by $L$, the induced maps $\alpha,\beta : K^*(X,\mathcal{G}) \to K^*(X,\mathcal{H})$ differ by multiplication by $L$, $ [L] \otimes : K^*(X,\mathcal{H}) \to K^*(X,\mathcal{H})$.}
\item{For any proper map $\phi : Y \to X$ there are pullback maps $\phi^* : K^*(X,\mathcal{G}) \to K^*(Y,\phi^*(\mathcal{G}))$ which are contravariant in $\phi$.}
\item{For any smooth map $\phi : Y \to X$ there is a pushforward \cite{fht},\cite{cw0},\cite{cw} $\phi_* : K^i(Y,\phi^*(\mathcal{G}) \otimes L(\phi)) \to K^{i-n}(Y,\mathcal{G})$ where $L(\phi)$ is the lifting gerbe of $TY \oplus \phi^*(TX)$ and $n$ is rank of $TY \oplus \phi^*(TX)$. In the case where $\phi : Y \to X$ is a fiber bundle we can replace $L(\phi)$ by $L(V)$ where $V = {\rm Ker}(\phi_*)$ is the vertical tangent bundle \cite{bar}.}
\item{There is a Mayer-Vietoris sequence in twisted $K$-theory \cite{fht},\cite{cw0} (see also Appendix \ref{grgtwkt}).}
\end{itemize}

Given a T-dual pair $(X,\mathcal{G}),(\hat{X},\hat{\mathcal{G}})$, choose an isomorphism $\gamma : p^*(\mathcal{G}) \to \hat{p}^*(\hat{\mathcal{G}}) \otimes q^*(L(V))$ satisfying the Poincar\'e property, where as usual $V$ denotes the vertical bundle of $X \to M$. From the choice of isomorphism $\gamma$ we may define a T-duality transform $T : K^*(X,\mathcal{G}) \to K^{*-n}(\hat{X},\hat{\mathcal{G}})$ in twisted $K$-theory, where $n$ is the rank of the fibers of $X,\hat{X}$. We can show that $T$ is an isomorphisms, at least when the base $M$ admits a finite good cover. The map $T$ is defined to make the following diagram commute:
\begin{equation*}\xymatrix{
K^*(C,p^*(\mathcal{G})) \ar[r]^-\gamma & K^*(C,\hat{\mathcal{G}} \otimes q^*(L(V))) \ar[d]^{\hat{p}_*} \\
K^*(X,\mathcal{G}) \ar[r]^-T \ar[u]^{p^*} & K^{*-n}(\hat{X},\hat{\mathcal{G}})
}
\end{equation*}
The map $ \hat{p}_* : K^*(C,\hat{\mathcal{G}}) \to K^{*-n}(\hat{X},\hat{\mathcal{G}})$. is the pushforward operation in twisted $K$-theory.\\

We say that an open cover $\mathcal{U}$ of $M$ is a {\em good cover} if all non-empty finite intersections in $\mathcal{U}$ are contractible. If $M$ is compact then $M$ admits a finite good cover \cite[Theorem 5.1]{botttu}.\\

\begin{theorem}\label{ttdttktiso}
Suppose that $M$ admits a finite good cover. Then for any T-dual pairs $(X,\mathcal{G}),(\hat{X},\hat{\mathcal{G}})$ on $M$ and isomorphism $\gamma : p^*(\mathcal{G}) \to \hat{p}^*(\hat{\mathcal{G}}) \otimes q^*(L(V))$ satisfying the Poincar\'e property, the map $T : K^*(X,\mathcal{G}) \to K^{*-n}(\hat{X},\hat{\mathcal{G}})$ is an isomorphism.
\end{theorem}
\begin{proof}
Our proof is simply an adaptation of the approach taken in \cite{bunksch}. We prove this result by using induction on the number of elements in a finite good cover. To start the induction we need to show the result is true if the base is contractible. We will prove this shortly.

Now assume the result holds is whenever the base has a good cover with at most $k$ elements. Suppose $M$ has a good cover $\mathcal{U}$ with $k+1$ elements, say $\mathcal{U} = \{ U_0 , U_1 , \dots , U_{k} \}$. Then $N = U_1 \cup U_2 \cup \dots \cup U_k$ has a finite good cover with $k$-elements and so does $N \cap U_{0}$, namely $\{ U_0 \cap U_1 , U_0 \cap U_2 , \dots U_0 \cap U_k \}$. For an open subset $W \subset M$ we let $X|_W$ denote $\pi^{-1}(W)$ and $\mathcal{G}|_W$ denote the restriction of $\mathcal{G}$ to $X|_W$. Similarly define $\hat{X}|_W , C|_W$ and $\hat{\mathcal{G}}|_W$. By restriction $\gamma$ defines isomorphisms $\gamma|_W : p^*(\mathcal{G}|_W) \to \hat{p}^*(\hat{\mathcal{G}}|_W) \otimes q^*(L(V))$ satisfying the Poincar\'e property. Using $\gamma|_W$ we may construct over $W$ a T-duality transformation which we denote by $T|_W : K^*(X|_W , \mathcal{G}|_W) \to K^{*-n}(\hat{X}|_W , \hat{\mathcal{G}}|_W)$. We obtain a commutative diagram with rows given by Mayer-Vietoris sequences:
\begin{equation*}\xymatrix@C=10pt{
\cdots \ar[r] &  K^*(X|_{N \cap U_0} , \mathcal{G} ) \ar[d]^{T|_{N \cap U_0}} \ar[r] & K^*(X|_N,\mathcal{G})  \oplus  K^*(X|_{U_0},\mathcal{G}) \ar[d]^{T|_N \oplus T_{U_0} } \ar[r] & K^*(X,\mathcal{G}) \ar[d]^{T} \ar[r] & \cdots \\
\cdots \ar[r] & K^{*-n}(\hat{X}|_{N \cap U_0} , \hat{\mathcal{G}} ) \ar[r] & K^{*-n}(\hat{X}|_N,\hat{\mathcal{G}}) \oplus K^{*-n}(\hat{X}|_{U_0},\hat{\mathcal{G}}) \ar[r] & K^{*-n}(\hat{X},\hat{\mathcal{G}}) \ar[r] & \cdots
}
\end{equation*}
where we have dropped the restriction notation for the gerbes $\mathcal{G},\hat{\mathcal{G}}$. Now since $N,U_0$ and $N \cap U_0$ all admit good covers with no more than $k$ elements we have by induction that $T|_N,T|_{U_0}$ and $T|_{N \cap U_0}$ are isomorphisms. Thus by an application of the five lemma we find that $T$ is also an isomorphism.\\

All that remains is to show the result when $M$ is contractible. In this case the torus bundles $X,\hat{X}$ are both trivial, so we can identify $X,\hat{X}$ with the trivial bundles $M \times T^n$, $M \times \hat{T}^n$ (here $T^n = \mathbb{R}^n/\mathbb{Z}^n$ and $\hat{T}^n = (\mathbb{R}^n)^*/(\mathbb{Z}^n)^*$ is the dual torus). The gerbes $\mathcal{G},\hat{\mathcal{G}}$ are by definition of T-duality trivial on the fibers, so in this case we have that $\mathcal{G},\hat{\mathcal{G}}$ are trivial. Choosing trivializations of $\mathcal{G},\hat{\mathcal{G}}$, we can then identify $K^*(X,\mathcal{G})$ with $K^*(M \times T^n)$, which itself is isomorphic to $K^{*-b}(T^n)$, where $b$ is the dimension of the base. To see this one can apply the Atiyah-Hirzebruch spectral sequence \cite{atseg1} to the fiber bundle $M \times T^n \to M$. We can similarly identify $K^*(\hat{X},\hat{\mathcal{G}})$ with $K^{*-b}(\hat{T}^n)$. The T-duality transform then reduces to a map $T' : K^{*-b}(T^n) \to K^{*-b-n}(T^n)$ defined by a commutative diagram
\begin{equation}\xymatrix{
K^*(T^n \times \hat{T}^n) \ar[r]^-{\otimes P} & K^*(T^n \times \hat{T}^n) \ar[d]^{\hat{p}_*} \\
K^*(T^n) \ar[r]^-{T'} \ar[u]^{p^*} & K^{*-n}(\hat{T}^n)
}
\end{equation}
where $P$ is a certain line bundle on $T^n \times \hat{T}^n$. In fact, since $\gamma$ satisfies the Poincar\'e property we can choose the trivializations of $\mathcal{G},\hat{\mathcal{G}}$ so that the map $K^*(T^n \times \hat{T}^n) \to K^*(T^n \times \hat{T}^n)$ is given by tensoring with the Poincar\'e line bundle. Thus $T'$ is a $K$-theoretic version of the Fourier-Mukai transform \cite{muk} and a straightforward computation shows that $T'$ is an isomorphism.
\end{proof}

Assume $X$ is compact and let $\mathcal{G}$ be a graded gerbe on $X$ with grading class $\xi \in H^1(X,\mathbb{Z}_2)$. Choose a connection and curving for $\mathcal{G}$ with curvature $3$-form $H \in \Omega^3(X)$. We then have a twisted Chern character $Ch_{\mathcal{G}} : K^*(X , \mathcal{G}) \to H^*(X , (\xi , H)) $ which depends on $\mathcal{G}$ as well as the connection and curving. From the point of view of T-duality, or rather the Fourier-Mukai transform, it is natural to consider a modified version of the twisted Chern character. Define $v :  K^*(X , \mathcal{G}) \to H^*(X , (\xi , H)) $ by $v(x) = Ch_{\mathcal{G}}(x)\sqrt{\hat{A}(X)}$, where $\hat{A}(X)$ is the $\hat{A}$-genus of $X$. We call $v$ the {\em Mukai map}. For any $x \in K^*(X,\mathcal{G})$, $v(x)$ is also known as the {\em Mukai vector} of $x$.\\

Let $(X,\mathcal{G}),(\hat{X},\hat{\mathcal{G}})$ be T-dual pairs over $M$. Choose connections, curvings and T-duality triple $(A,\hat{A},H_3)$ according to Proposition \ref{linkgerbeh}. We then have a commutative diagram:
\begin{equation*}\xymatrix{
K^*(C,p^*(\mathcal{G})) \ar[r]^-\gamma \ar[d]^v & K^*(C,\hat{p}^*(\hat{\mathcal{G}}) \otimes q^*(L(V)) ) \ar[d]^{v} \\
H^*(C,p^*(\xi,H))  \ar[r]^-{e^{-\mathcal{B}}} & H^*(C,\hat{p}^*(\hat{\xi} , \hat{H})) 
}
\end{equation*}
where the vertical maps in this diagram correspond to the Mukai maps for $p^*(\mathcal{G})$ and $\hat{p}^*(\hat{\mathcal{G}}) \otimes q^*(L(V))$. Note that since $V$ is flat $\sqrt{\hat{A}(V)} = 1$, and $\sqrt{\hat{A}(C)} = q^*\sqrt{\hat{A}(M)}$. Furthermore, it is clear that the Mukai map is preserved by pullback and pushforward in the sense that we have commutative diagrams
\begin{equation*}\xymatrix{
K^*(X,\mathcal{G}) \ar[r]^-{p^*} \ar[d]^v & K^*(C,p^*(\mathcal{G})) \ar[d]^v \\
H^*(X,(\xi,H))  \ar[r]^-{p^*} & H^*(C, p^*(\xi,H)) 
}
\end{equation*}
and
\begin{equation*}\xymatrix{
K^*(C, \hat{p}^*(\hat{\mathcal{G}}) \otimes q^*(L(V)) ) \ar[d]^{v} \ar[r]^-{\hat{p}_*} & K^{*-n}(\hat{X},\hat{\mathcal{G}}) \ar[d]^v \\
H^*(C,\hat{p}^*(\hat{\xi},\hat{H}))  \ar[r]^-{\hat{p}_*} & H^{*-n}(\hat{X},(\hat{\xi},\hat{H})) 
}
\end{equation*}
where the second diagram commutes by an application of Riemann-Roch, Theorem \ref{rrgbg}. Putting these together we immediately have shown the following:
\begin{proposition}\label{tkchern}
Let $M$ be compact and $(X,\mathcal{G}),(\hat{X},\hat{\mathcal{G}})$ be T-dual pairs over $M$.  Choose connections, curvings and T-duality triple $(A,\hat{A},H_3)$ according to Proposition \ref{linkgerbeh}. From this data we have induced T-duality isomorphisms $T : K^*(X,\mathcal{G}) \to K^{*-n}(\hat{X},\hat{\mathcal{G}})$ and $T : H^*(X,(\xi,H)) \to H^{*-n}(\hat{X},(\hat{\xi},\hat{H}))$. We then have a commutative diagram:
\begin{equation*}\xymatrix{
K^*(X,\mathcal{G}) \ar[r]^-T \ar[d]^v & K^{*-n}(\hat{X},\hat{\mathcal{G}}) \ar[d]^v \\
H^*(X,(\xi,H))  \ar[r]^-T & H^{*-n}(\hat{X},(\hat{\xi},\hat{H})) .
}
\end{equation*}
\end{proposition}

Note that it is somewhat superfluous to use the Mukai maps rather than the twisted Chern characters, because we have $\sqrt{\hat{A}(X)} = \pi^*(\sqrt{\hat{A}(M)})$ and $\sqrt{\hat{A}(\hat{X})} = \hat{\pi}^*( \sqrt{\hat{A}(M)})$. We prefer to use the Mukai maps since as we will see they can be used to relate natural pairings in twisted cohomology and twisted $K$-theory.\\

Just as we defined a Mukai pairing on twisted cohomology, there is a similar pairing at the level of twisted $K$-theory. Let $\mathcal{G}$ be a graded gerbe on $X$. If we define twisted $K$-theory in terms of homotopy classes of sections of bundles of Fredholm operators, then by taking adjoints we get a canonical isomorphism $\sigma : K^*(X,\mathcal{G}) \to K^*(X,\mathcal{G}^*)$, where $\mathcal{G}^*$ is the dual graded gerbe, which is defined by taking the dual of the underlying line bundle of $\mathcal{G}$ which inherits a dual gerbe product. If we assume $X$ is $d$ dimensional, compact and spin we have a natural map $I : K^*(X) \to K^{*-d}(pt)$, which is just the pushforward from $X$ to a point. Note that $I$ only depends on a choice of orientation of $X$ rather than a choice of spin structure. Using this we get a natural pairing on $K^*(X,\mathcal{G})$ defined as follows
\begin{equation*}
\langle a , b \rangle = I( \sigma(a) b ).
\end{equation*}
Naturally we call $\langle \, , \, \rangle$ the {\em Mukai pairing} in twisted $K$-theory. We could generalize this to ${\rm Spin}^c$-manifolds but then pushforward to $K^*(pt)$ depends on the explicit choice of ${\rm Spin}^c$-structure, not just the orientation.

Let $\mathcal{G}$ be a graded gerbe with grading class $\xi \in H^1(X,\mathbb{Z}_2)$. Choose a connection and curving for $\mathcal{G}$ with curvature $3$-form $H \in \Omega^3(X)$. We have the Mukai map $v : K^*(X,\mathcal{G}) \to H^*(X,(\xi,H))$. The Riemann-Roch theorem for graded bundle gerbes immediately implies that
\begin{equation}\label{resppair}
\langle v(a) , v(b) \rangle = \langle a , b \rangle,
\end{equation}
where the left hand side is the Mukai pairing in twisted cohomology and the right hand side the Mukai pairing in twisted $K$-theory.

\begin{proposition}
Let $M$ be compact $m$-dimensional, oriented and suppose $(X,\mathcal{G}), \linebreak (\hat{X},\hat{\mathcal{G}})$ are rank $n$ T-dual pairs over $M$. Suppose the vertical bundle of $X$ is oriented and that $X$ is spin. Then $\hat{X}$ is also spin and we have Mukai pairings on $X,\hat{X}$ in twisted $K$-theory. Let $T : K^*(X,\mathcal{G}) \to K^{*-n}(\hat{X},\hat{\mathcal{G}})$ be the T-duality transformation associated to an isomorphism $\gamma : p^*(\mathcal{G}) \to \hat{p}^*(\hat{\mathcal{G}}) \otimes q^*(L(V))$ satisfying the Poincar\'e property. We then have
\begin{equation*}
\langle Ta , Tb \rangle = (-1)^{nm}\langle a , b \rangle,
\end{equation*}
for all $a,b \in K^*(X,\mathcal{G})$.
\end{proposition}
\begin{proof}
Choose connections, curvings and T-duality triple $(A,\hat{A},H_3)$ according to Proposition \ref{linkgerbeh}. Then we use Propositions (\ref{tkchern}), (\ref{mptc}) and Equation (\ref{resppair}) as follows:
\begin{eqnarray*}
\langle Ta , Tb \rangle &=& \langle v(Ta) , v(Tb) \rangle \\
&=& \langle T(va) , T(vb) \rangle \\
&=& (-1)^{nm}\langle va , vb \rangle \\
&=& (-1)^{nm}\langle a,b \rangle.
\end{eqnarray*}
\end{proof}


\subsection{Courant algebroids}\label{tdtca}

We show that a T-duality triple determines an isomorphism of Courant algebroids. This section is largely a translation of \cite[Chapter 8]{gual} and \cite[Chapter 7]{cav} into the non-principal case, for which we claim no originality. See also \cite{hu} for connections between T-duality and Courant algebroids.\\

Courant algebroids were introduced by Liu, Weinstein and Xu \cite{lwx} as a generalization of the bracket used by Courant in the study of Dirac structures \cite{cour}, which arise in the study of Hamiltonian systems. We recall the definition of a Courant algebroid from the point of view of the Dorfman bracket \cite{roy}. A {\em Courant algebroid} on a smooth manifold $X$ is vector bundle $E \to X$, $\mathbb{R}$-bilinear bracket $[ \, , \, ] : \Gamma(E) \otimes \Gamma(E) \to \Gamma(E)$ called the {\em Dorfman bracket}, non-degenerate bilinear form $( \, , \, )$ and a bundle map $\rho : E \to TX$ called the {\em anchor} such that
\begin{itemize}
\item{$[a,[b,c]] = [[a,b],c] + [b,[a,c]]$,}
\item{$\rho[a,b] = [\rho(a),\rho(b)]$,}
\item{$[a,fb] = \rho(a)(f)b + f[a,b]$,}
\item{$[a,b] + [b,a] = d (a,b)$,}
\item{$\rho(a)(b,c) = ([a,b],c) + (b,[a,c])$,}
\end{itemize}
where $a,b,c \in \Gamma(E)$, $f$ is a function on $X$ and $d$ is the operator $d : \mathcal{C}^\infty(X) \to \Gamma(E)$ defined by $(df , a ) = \rho(a)(f)$.\\ 

A Courant algebroid $E \to X$ is {\em exact} if the sequence
\begin{equation*}\xymatrix{
0 \ar[r] & T^*X \ar[r]^{\rho^*} & E \ar[r]^\rho & TX \ar[r] & 0
}
\end{equation*}
is exact. Here $\rho^*$ is the transpose $T^*X \to E^*$ of $\rho$ followed by the identification of $E$ and $E^*$ using the pairing. Exact Courant algebroids over a manifold $X$ are classified by third cohomology with real coefficients, $H^3(X,\mathbb{R})$. In fact given a closed $3$-form $H$ representing a class in $H^3(X,\mathbb{R})$ we give $E = TX \oplus T^*X$ the structure of a Courant algebroid with $H$-twisted Dorfman bracket \cite{sw} given by
\begin{equation}\label{courb}
[(A , \alpha) , (B , \beta) ]_H = ([A,B] , \mathcal{L}_A \beta - i_B d \alpha + i_B i_A H).
\end{equation}
The pairing $( \, , \, )$ on $E$ is the natural paring of $TX$ with $T^*X$ and anchor the projection $\rho : E \to TX$. Then $(E,[ \, , \, ]_H , ( \, , \, ) , \rho)$ is an exact Courant algebroid and one can show that every exact Courant algebroid on $X$ is isomorphic to one of this form. Given two closed $3$-forms $H,H'$ the associated exact Courant algebroids are isomorphic if and only if $H$ and $H'$ represent the same class in $H^3(X,\mathbb{R})$. The classification of exact Courant algebroids is due to \v{S}evera \cite{sev}.\\

Let $(X,\mathcal{G})$ be a pair consisting of an affine torus bundle and graded gerbe $\mathcal{G}$. Let $H \in \Omega^3(X)$ be an invariant $3$-form representing the ungraded Dixmier-Douady class of $\mathcal{G}$. We give $E = TX \oplus T^*X$ the structure of an exact Courant algebroid using the $H$-twisted Dorfman bracket (\ref{courb}). Just as we defined invariant forms on $X$ we can also speak of invariant vector fields, thus we may also speak of invariant sections of $E$. Using the fact that $H$ is invariant one easily sees that the invariant sections of $E$ are closed under the Dorfman bracket. In fact, the invariant sections of $E$ can be identified with the sections of a vector bundle $E^{{\rm red}}$ on $M$ and this bundle then inherits the structure of a (non-exact) Courant algebroid on $M$. For instance the anchor is given by composing the anchor $E \to TX$ with the projection $\pi_* : TX \to TM$. Since this composition is invariant in the obvious sense, it descends to a bundle map $\rho^{{\rm red}} : E^{{\rm red}} \to TM$. The resulting Courant algebroid $(E^{{\rm red}} , [ \, , \, ]_H^{{\rm red}} , ( \, , \, ) , \rho^{{\rm red}})$ on $M$ will be simply denoted $E^{{\rm red}}$. Up to isomorphism it does not depend on the choice of invariant representative $H$.\\

Let $(X,\mathcal{G}),(\hat{X},\hat{\mathcal{G}})$ be T-duals. Choose a T-duality triple $(A,\hat{A},H_3)$ and use it to define $H,\hat{H},\mathcal{B}$ as given in Equations (\ref{tdt2}),(\ref{tdt3}),(\ref{calb}). The twisted connections $A$,$\hat{A}$ yield associated splittings
\begin{eqnarray*}
TX &=& \pi^*(TM \oplus V), \\
T\hat{X} &=& \hat{\pi}^*(TM \oplus V^*).
\end{eqnarray*}
Let $E^{{\rm red}}$ denote the Courant algebroid on $M$ obtained by taking invariant sections of $TX \oplus T^*X$ with $H$-twisted Dorfman bracket and similarly let $\hat{E}^{{\rm red}}$ be the Courant algebroid on $M$ obtained by taking invariant sections over $\hat{X}$. The choice of $A,\hat{A}$ yield isomorphisms
\begin{eqnarray*}
E^{{\rm red}} &=& TM \oplus V \oplus V^* \oplus T^*M, \\
\hat{E}^{{\rm red}} &=& TM \oplus V^* \oplus V \oplus T^*M.
\end{eqnarray*}
Let $\phi : E^{{\rm red}} \to \hat{E}^{{\rm red}}$ be the bundle isomorphism which exchanges positions of the $V,V^*$ factors, that is
\begin{equation*}
\phi( Y , a , \alpha , \eta) = (Y , \alpha , a , \eta).
\end{equation*}

\begin{theorem}\label{cai} The map $\phi : E^{{\rm red}} \to \hat{E}^{{\rm red}}$ is an isomorphism of Courant algebroids.
\end{theorem}
\begin{proof}
Our proof is a straightforward adaptation of \cite[Theorem 7.2]{cav}. Recall that from the triple $(A,\hat{A},H_3)$ we may define a T-duality map $T : \Omega^*(X,\xi) \to \Omega^{*-n}(\hat{X},\hat{\xi})$ as given in (\ref{tdtdef}). Moreover $T$ is an isomorphism between invariant forms on $X$ and $\hat{X}$ and intertwines the twisted differentials $d_{\xi,H},d_{\hat{\xi},\hat{H}}$. We also established in Section \ref{tdttwico} that the invariant forms on $X,\hat{X}$ correspond to sections of the bundles $S = \mathbb{R}_\xi \otimes \wedge^* T^*M \otimes \wedge^* V^*$ and $\hat{S} = \mathbb{R}_{\hat{\xi}} \otimes \wedge^* T^*M \otimes \wedge^* V$ respectively. Under this identification the T-duality transform $T$ was simply a bundle isomorphism $T : S \to \hat{S}$.\\

There is a natural action of sections of $E = TX \oplus T^*X$ on $\Omega^*(X)$ given by a bundle map $\gamma : E \otimes \wedge^* T^*M \to \wedge^* T^*M$. Let $a = (Y,\eta)$ be a section of $TX \oplus T^*X$ and $\omega \in \Omega^*(X)$. Then
\begin{equation*}
\gamma_a \omega = i_Y \omega + \eta \wedge \omega.
\end{equation*}
If $L$ is any flat line bundle the above action can similarly be defined on $\Omega^*(X,L)$. If $a$ and $\omega$ are invariant then so is $\gamma_a(\omega)$, so we obtain a corresponding bundle map $\gamma^{{\rm red}} : E^{{\rm red}} \otimes S \to S$ and similarly we obtain $\hat{\gamma}^{{\rm red}} : \hat{E}^{{\rm red}} \otimes \hat{S} \to \hat{S}$.\\

A straightforward computation shows that
\begin{equation}\label{clifft}
T( \gamma^{{\rm red}}_a \omega ) = \hat{\gamma}^{{\rm red}}_{\phi(a)} (T \omega).
\end{equation}

The Dorfman bracket is a derived bracket \cite{kos} in the sense that for $a,b$ sections of $E$ and $\omega \in \Omega^*(X)$ we have
\begin{equation*}
\gamma_{[ a , b ]_H} \omega = [ [d_{\xi,H} , \gamma_a ] , \gamma_b ] \omega.
\end{equation*}
In the above equation we think of $d_{\xi,H}$ as an operator of even degree, $\gamma_a,\gamma_b$ as operators of odd degree and take graded commutators. Writing this out in full we have
\begin{equation*}
\gamma_{[a,b]_H} \omega = d_{\xi,H}( \gamma_a \gamma_b \omega) - \gamma_a d_{\xi,H}(\gamma_b \omega) - \gamma_b d_{\xi,H} (\gamma_a \omega) + \gamma_b \gamma_a d_{\xi,H} \omega.
\end{equation*}
A similar identity holds for the Courant algebroid $\hat{E}$ on $\hat{X}$. Let us now restrict to invariant sections. Using (\ref{clifft}) and the fact that $T$ intertwines differentials we find
\begin{eqnarray*}
\hat{\gamma}^{{\rm red}}_{\phi([a,b]^{{\rm red}}_H)} T \omega &=& T ( \gamma^{{\rm red}}_{[a,b]^{{\rm red}}_H} \omega) \\
&=& T ( [ [d_{\xi,H} , \gamma^{{\rm red}}_a ] , \gamma^{{\rm red}}_b ] \omega )\\
&=& [ [d_{\hat{\xi},\hat{H}} , \hat{\gamma}^{{\rm red}}_{\phi(a)} ] , \hat{\gamma}^{{\rm red}}_{\phi{b}} ] (T\omega) \\
&=& \hat{\gamma}^{{\rm red}}_{ [\phi{a},\phi{b}]^{{\rm red}}_{\hat{H}} } ( T \omega ),
\end{eqnarray*}
which reduces to simply
\begin{equation*}
\phi( [a,b]^{{\rm red}}_H ) = [ \phi(a) , \phi(b) ]^{{\rm red}}_{\hat{H}}.
\end{equation*}

We have shown that $\phi$ exchanges the Courant brackets on $E^{{\rm red}}, \hat{E}^{{\rm red}}$. To complete the proof one also needs to check that $\phi$ exchanges the anchors and bilinear forms. The exchange of the anchors is immediate. For the bilinear forms one need only note the following identity:
\begin{equation*}
(\gamma_a \gamma_b + \gamma_b \gamma_a)\omega = (a,b)\omega,
\end{equation*}
where $a,b$ are sections of $E$ and $\omega$ a form on $X$. A similar identity holds on $\hat{X}$. Applying $T$ and again using (\ref{clifft}), we see that $\phi$ indeed preserves the pairings.
\end{proof}


\section{Examples}\label{examps}

In the following examples we are concerned with computing some twisted $K$-theory groups for torus bundles with non-trivial monodromy. This can be used as a test of T-duality when the appropriate pair of groups can be computed by other means. On the other hand we can use T-duality to compute twisted $K$-theory groups not easily obtained by more elementary means. In the following examples all gerbes will have trivial gradings. The main tool for computing twisted $K$--theory groups $K^*(X,\mathcal{G})$ will be the Atiyah-Hirzebruch spectral sequence. According to \cite{atseg2}, we have that the $E_3$ stage has the form $H^p( X , K^q(pt))$ and $d_3 = {\rm Sq}^3_\mathbb{Z} - h\smallsmile$, where $h$ is the (undgraded) Dixmier-Douady class and ${\rm Sq}^3_\mathbb{Z}$ is the third integral Steenrod operation. The examples are of low enough dimension that the Steenrod operation will not appear.


\subsection{T-duality and the Euclidean algorithm}\label{tdea}

In this example we will show how one can compute a twisted $K$-theory group by repeated applications of T-duality in the fashion of the Euclidean algorithm. We take our base space space to be $M = T^2$ the $2$-torus. The fundamental group $\pi_1(M) = \mathbb{Z}^2$ with generators $x,y$ say. We define a representation $\rho : \pi_1(M) \to {\rm SL}(2,\mathbb{Z})$ as follows:
\begin{equation*}
\begin{aligned}
\rho(x) = \left[ \begin{matrix} 1 & m  \\ 0 & 1 \end{matrix} \right], & & \rho(y) = \left[ \begin{matrix} 1 & n \\ 0 & 1 \end{matrix} \right]
\end{aligned}
\end{equation*}
where $m,n \in \mathbb{Z}$ are positive integers with no common factor. Let $\Lambda_\rho$ be the $\mathbb{Z}^2$-valued local system determined by $\rho$. Then $\wedge^2 \Lambda_\rho \simeq \mathbb{Z}$ and $\Lambda^*_\rho \simeq \Lambda_\rho$. A straightforward computation reveals the cohomology of $M$ with local coefficients in $\Lambda_\rho$ to be as follows:
\begin{equation*}
\renewcommand{\arraystretch}{1.4}
\begin{tabular}{|l|l|}
\hline
$i$ & $H^i(M,\Lambda_\rho)$ \\
\hline
$0$ & $\mathbb{Z}$ \\
$1$ & $\mathbb{Z}_m \oplus \mathbb{Z}_n$ \\
$2$ & $\mathbb{Z}$ \\
\hline
\end{tabular}
\end{equation*}

Rank $2$ affine torus bundles over $M$ with monodromy $\rho$ are classified by their Chern class $j \in H^2(M,\Lambda_\rho) = \mathbb{Z}$. Let $X_j \to M$ be the corresponding torus bundle. We compute the integral cohomology of $X_j$ by means of the Leray-Serre spectral sequence and Poincar\'e duality, noting that $X_j$ is oriented. We find
\begin{equation*}
\renewcommand{\arraystretch}{1.4}
\begin{tabular}{|l|l|l|}
\hline
$i$ & $H^i(X_j,\mathbb{Z})$, $j=0$ & $H^i(X_j,\mathbb{Z})$, $j \neq 0$ \\
\hline
$0$ & $\mathbb{Z}$ & $\mathbb{Z}$ \\
$1$ & $\mathbb{Z}^3$ & $\mathbb{Z}^2$ \\
$2$ & $\mathbb{Z}^2$ & $\mathbb{Z}_j$ \\
$3$ & $\mathbb{Z}^3$ & $\mathbb{Z}^2 \oplus \mathbb{Z}_j$ \\
$4$ & $\mathbb{Z}$ & $\mathbb{Z}$ \\
\hline
\end{tabular}
\end{equation*}
We can also work out the subgroup $F^{2,3}(\pi,\mathbb{Z})$ of $H^3(X_j,\mathbb{Z})$ of T-dualizable flux. In the case $j=0$ it is given by an inclusion $(1,0) : \mathbb{Z} \to \mathbb{Z} \oplus \mathbb{Z}^2$ while for $j \neq 0$ it is the subgroup $\mathbb{Z}_j \subset \mathbb{Z}^2 \oplus \mathbb{Z}_j$. The projection $E_2^{2,1}(\pi,\mathbb{Z}) \to E_\infty^{2,1}(\pi,\mathbb{Z}) = F^{2,3}(\pi,\mathbb{Z})/F^{3,3}(\pi,\mathbb{Z})$ is given by $\mathbb{Z} \to \mathbb{Z}$ for $j=0$ and $\mathbb{Z} \to \mathbb{Z}_j$ for $j \neq 0$. Let $k \in \mathbb{Z} \simeq E^{2,1}_2(\pi,\mathbb{Z})$ represent the flux on $X_j$. Since $H^3(M,\mathbb{Z}) = 0$ we find that $F^{2,3}(\pi,\mathbb{Z}) = E_\infty^{2,1}(\pi,\mathbb{Z})$ so that the integer $k$ completely determines a flux $h_k \in H^3(X_j,\mathbb{Z})$. If $j \neq 0$ then $h_k = h_{k+j}$. The class $k$ determines an affine torus bundle $\hat{X}_k \to M$ with monodromy $\Lambda^*_\rho$ and Chern class $k \in H^2(M,\Lambda^*_\rho)$. Since $\Lambda_\rho$ is self-dual the space $\hat{X}_k$ is identical to $X_k$. The class $j$ likewise determines a flux $\hat{h}_j \in H^3(\hat{X}_k , \mathbb{Z})$ which under $\hat{X}_k \simeq X_k$ just becomes $h_j \in H^3(X_k,\mathbb{Z})$. If we write the Chern class $j$ and flux $k$ as a pair $(j,k) \in \mathbb{Z}^2$ then T-duality in this case is simply the interchange $(j,k) \to (k,j)$.\\

Now we compute the twisted $K$-theory groups $K^i(X_j , h_k)$ for all pairs $(j,k)$. We can use the Atiyah-Hirzebruch spectral sequence to compute the groups, but this does not yield a complete solution since the spectral sequence only computes the $K$-theory up to an extension problem. To the extent that we can directly solve the extension problem we find
\begin{equation*}
\renewcommand{\arraystretch}{1.4}
\begin{tabular}{|l|l|l|l|l|}
\hline
$i$ & $K^i(X_j,h_k)$ & $K^i(X_j,h_k)$ & $K^i(X_j,h_k)$ & $K^i(X_j,h_k)$ \\
& $j=0, k =0$ & $j=0,k \neq 0$ & $j \neq 0, k = 0\, ({\rm mod} \, j)$ & $j \neq 0, k \neq 0 \, ({\rm mod} \, j)$ \\
\hline
$0$ & $\mathbb{Z}^4$ & $\mathbb{Z}^2 \oplus \mathbb{Z}_k$ & $*$ & $*$ \\
$1$ & $\mathbb{Z}^6$ & $\mathbb{Z}^4 \oplus \mathbb{Z}_k$ & $\mathbb{Z}^4 \oplus \mathbb{Z}_j$ & $\mathbb{Z}^4 \oplus \mathbb{Z}_d$ \\
\hline
\end{tabular}
\end{equation*}
where $*$ denotes a term that is not completely determined due to an extension problem and $d$ is the greatest common divisor of $j$ and $k$. We use T-duality to determine the remaining terms. First the case of $K^0(X_j,h_k)$ with $j \neq 0, k = 0\, ({\rm mod} \, j)$ is easily computed by the T-duality interchange $(j,0) \to (0,j)$ and thus $K^0(E_j) = K^0(X_0 , h_j) = \mathbb{Z}^2 \oplus \mathbb{Z}_j$.

The computation of $K^0(X_j,h_k)$ with $j \neq 0, k \neq 0 \, ({\rm mod} \, j)$ is more subtle. We make use of the fact that twisted $K$-theory is unchanged under the T-duality exchange $(j,k) \to (k,j)$ and is unchanged under shifts $(j,k) \to (j,k+j)$, which follows simply from the identity $h_{k+j} = h_k$. By repeated application of T-duality swaps $(p,q) \to (q,p)$ and shifts $(p,q) \to (p,q+p)$ we can run the Euclidean algorithm to reduce the pair $(j,k)$ down to $(d,0)$ where $d$ as above is the greatest common divisor of $j$ and $k$. So we obtain $K^0(X_j,h_k) = \mathbb{Z}^2 \oplus \mathbb{Z}_d$ in this case. The complete table is thus
\begin{equation*}
\renewcommand{\arraystretch}{1.4}
\begin{tabular}{|l|l|l|l|l|}
\hline
$i$ & $K^i(X_j,h_k)$ & $K^i(X_j,h_k)$ & $K^i(X_j,h_k)$ & $K^i(X_j,h_k)$ \\
& $j=0, k =0$ & $j=0,k \neq 0$ & $j \neq 0, k = 0\, ({\rm mod} \, j)$ & $j \neq 0, k \neq 0 \, ({\rm mod} \, j)$ \\
\hline
$0$ & $\mathbb{Z}^4$ & $\mathbb{Z}^2 \oplus \mathbb{Z}_k$ & $\mathbb{Z}^2 \oplus \mathbb{Z}_j$ & $\mathbb{Z}^2 \oplus \mathbb{Z}_d$ \\
$1$ & $\mathbb{Z}^6$ & $\mathbb{Z}^4 \oplus \mathbb{Z}_k$ & $\mathbb{Z}^4 \oplus \mathbb{Z}_j$ & $\mathbb{Z}^4 \oplus \mathbb{Z}_d$ \\
\hline
\end{tabular}
\end{equation*}

This example demonstrates clearly that repeated applications of T-duality can yield stronger results than is possible through a single T-duality. This phenomenon is due to the non-uniqueness of T-duals, since if T-duals were unique two applications of T-duality would take us back to the starting point.


\subsection{Example with three dimensional base}

Here is an example we found of T-duality on a $3$-dimensional base where the monodromy is non-trivial and in which the twisted $K$-theory could easily be computed.\\

Let $M$ be the non-trivial $S^2$-bundle over $S^1$ obtained by attaching the ends of $S^2 \times [0,1]$ by the antipodal map. The fundamental group of $M$ is $\pi_1(M) = \mathbb{Z}$. Let $x \in \pi_1(M)$ be a generator, we define a monodromy representation $\rho$ by setting
\begin{equation*}
\rho(x) = \left[ \begin{matrix} -1 & -1 \\ 0 & -1 \end{matrix} \right]
\end{equation*}
let $\Lambda = \mathbb{Z}^2$ and $\Lambda_\rho$ the corresponding local system. Note then that $\wedge^2 \Lambda_\rho \simeq \mathbb{Z}$ and $\Lambda_\rho^* \simeq \Lambda_\rho$. The orientation local system $\mathbb{Z}_{\rm orn}$ for $M$ is given by letting $x$ act as $-1$. We compute the cohomology of $M$ with coefficients in $\mathbb{Z}$ and $\Lambda_\rho$:
\begin{equation*}
\renewcommand{\arraystretch}{1.4}
\begin{tabular}{|l|l|l|}
\hline
$i$ & $H^i(M,\mathbb{Z})$ & $H^1(M,\Lambda_\rho)$ \\
\hline
$0$ & $\mathbb{Z}$ & $0$ \\
$1$ & $\mathbb{Z}$ & $\mathbb{Z}_4$ \\
$2$ & $0$ & $\mathbb{Z}$ \\
$3$ & $\mathbb{Z}_2$ & $\mathbb{Z}$ \\
\hline
\end{tabular}
\end{equation*}
For $j \in H^2(M,\Lambda_\rho) \simeq \mathbb{Z}$ let $\pi : X_j \to M$ be the corresponding affine torus bundle. As the first step towards computing the twisted $K$-theory of $X_j$ we first compute the integral cohomology with the Leray-Serre spectral sequence. There are some difficulties in carrying this out, specifically we do not have a prescription for the third differential which might be non-trivial here. To proceed with the computation we restrict to the case that $j$ is odd since in this case the only possibly non-trivial third differential $d_3 : E_3^{0,2}(\pi,\mathbb{Z}) \to E_3^{3,0}(\pi,\mathbb{Z})$ vanishes by the simple reason that $E_3^{3,0} = 0$ when $j$ is odd. The cohomology of $X_j$ is then found to be
\begin{equation*}
\renewcommand{\arraystretch}{1.4}
\begin{tabular}{|l|l|}
\hline
$i$ & $H^i(X_j,\mathbb{Z})$ \\
\hline
$0$ & $\mathbb{Z}$ \\
$1$ & $\mathbb{Z}$ \\
$2$ & $0$ \\
$3$ & $\mathbb{Z}_j$ \\
$4$ & $\mathbb{Z}_j$ \\
$5$ & $\mathbb{Z}_2$ \\
\hline
\end{tabular}
\end{equation*}
It turns out that $F^{2,3}(\pi,\mathbb{Z}) = H^3(X_j,\mathbb{Z})$, so that all the flux is T-dualizable in this case. We also find that $F^{3,3}(\pi,\mathbb{Z}) = 0$ so that $F^{2,3}(\pi,\mathbb{Z}) = E_\infty^{2,1}(\pi,\mathbb{Z}) = \mathbb{Z}_j$, while $E_2^{2,1}(\pi,\mathbb{Z}) = \mathbb{Z}$. Let $k \in \mathbb{Z} \simeq E_2^{2,1}(\pi,\mathbb{Z})$. So $k$ projects to a class $h_k \in H^3(X_j,\mathbb{Z})$ and we observe that $h_{k+j} = h_k$. If we restrict to the case that $k$ is odd then we know that T-duality is realized by the interchange $(j,k) \to (k,j)$. Let $d$ be the greatest common divisor of $j$ and $k$. The twisted $K$-theory groups can be computed from the Atiyah-Hirzebruch spectral sequence. We find
\begin{equation*}
\renewcommand{\arraystretch}{1.4}
\begin{tabular}{|l|l|}
\hline
$i$ & $K^i(X_j,h_k)$ \\
\hline
$0$ & $\mathbb{Z} \oplus \mathbb{Z}_2 \oplus \mathbb{Z}_d$ \\
$1$ & $\mathbb{Z} \oplus \mathbb{Z}_d$ \\
\hline
\end{tabular}
\end{equation*}
which is easily seen to be symmetrical under the T-duality exchange $(j,k) \to (k,j)$.


\appendix

\section{Twisted $K$-theory}\label{grgtwkt}

Here we summarize the properties of twisted $K$-theory with graded gerbes used in the paper. Most of the results are adapted from references such as \cite{fht},\cite{cw0},\cite{cw},\cite{bcmms},\cite{atseg1}. For a space $X$ and graded bundle gerbe $\mathcal{G}$ on $X$ we denote by $K^i(X,\mathcal{G})$ the corresponding twisted $K$-theory groups, where $i \in \mathbb{Z}_2$.\\

The definition of twisted $K$-theory we use is essentially that of \cite{fht}, with the exception that we use a compactly supported version as done in \cite{cw0},\cite{cw}. To be exact, the definition we are using is that of our previous paper \cite{bar}.\\

If $\mathcal{G},\mathcal{H}$ are stably isomorphic graded gerbes then $K^*(X,\mathcal{G})$ and $K^*(X,\mathcal{H})$ are isomorphic. Therefore the groups $K^*(X,\mathcal{G})$ depend only on the Dixmier-Douady class $[\mathcal{G}] = (\xi,h) \in H^1(X,\mathbb{Z}_2) \times H^3(X,\mathbb{Z})$ of $\mathcal{G}$. However one should take care to note that the isomorphisms $K^*(X,\mathcal{G}) \simeq K^*(X,\mathcal{H})$ are {\em non-canonical}. Different stable isomorphisms $\mathcal{G} \to \mathcal{H}$ generally induce different isomorphisms $K^*(X,\mathcal{G}) \to K^*(X,\mathcal{H})$. One way to take this into account is to think of twisted $K$-theory as a functor
\begin{equation*}
K^*( X , \; \; ) : GrGrb(X) \to Gr_{\mathbb{Z}_2}{\rm Ab},
\end{equation*}
where $Gr_{\mathbb{Z}_2}{\rm Ab}$ is the category of $\mathbb{Z}_2$-graded abelian groups and $GrGrb(X)$ is the category of gerbes and equivalence classes of stable isomorphisms described in Section \ref{ggrb}. We also note that if $\mathcal{G} = 1$ is the trivial gerbe, then $K^*(X,\mathcal{G}) = K^*(X)$ is ordinary topological $K$-theory.\\

We now review the main properties of twisted $K$-theory, some of which were described in Section \ref{tdttkt}. There is a product structure $K^i(X,\mathcal{G}) \otimes K^j(X,\mathcal{H}) \to K^{i+j}(X,\mathcal{G} \otimes \mathcal{H})$ which is graded commutative and associative. When $\mathcal{G},\mathcal{H}$ are the trivial gerbe this just becomes the usual product in $K$-theory.

When $X$ is compact, a graded line bundle $L \to X$ defines a class $[L] \in K^0(X)$. Let $\alpha,\beta : \mathcal{G} \to \mathcal{H}$ be two stable isomorphisms which differ by $L$. Then the induced maps $\alpha,\beta : K^*(X,\mathcal{G}) \to K^*(X,\mathcal{H})$ differ by multiplication by $L$, $ [L] \otimes : K^*(X,\mathcal{H}) \to K^*(X,\mathcal{H})$, meaning we have a commutative diagram of the form
\begin{equation*}\xymatrix{
K^*(X,\mathcal{G}) \ar[r]^-{\alpha} \ar[dr]^-{\beta} & K^*(X,\mathcal{H}) \ar[d]^{ [L] \otimes } \\
& K^*(X,\mathcal{H})
}
\end{equation*}
When $X$ is not compact two stable isomorphisms $\alpha,\beta$ which differ by a graded line bundle $L$ still give rise to a commutative diagram as above but now we have to think of $L$ as an element of non-compactly supported $K$-theory. At any rate what matters here is that the group ${\rm Pic}(X)$ of graded line bundles on $X$ acts on the twisted $K$-theory groups $K^*(X,\mathcal{H})$ by automorphisms.\\

For any proper map $\phi : Y \to X$ there are pullback maps $\phi^* : K^*(X,\mathcal{G}) \to K^*(Y,\phi^*(\mathcal{G}))$. Recall that there is also a pullback functor $\phi^* : GrGrb(X) \to GrGrb(Y)$. The pullback in twisted $K$-theory is then a natural transformation $\phi^* : K^*(X , \; \; ) \Rightarrow K^*(Y, \; \;) \circ \phi^*$. The pullback in twisted $K$-theory is contravariant in the evident sense.\\

For any smooth map $\phi : Y \to X$ there is a pushforward \cite{fht},\cite{cw0},\cite{cw} $\phi_* : K^i(Y,\phi^*(\mathcal{G}) \otimes L(\phi)) \to K^{i-n}(Y,\mathcal{G})$ where $L(\phi)$ is the lifting gerbe of $TY \oplus \phi^*(TX)$ and $n$ is the rank of $TY \oplus \phi^*(TX)$. In the case where $\phi : Y \to X$ is a fiber bundle we can replace $L(\phi)$ by $L(V)$ where $V = {\rm Ker}(\phi_*)$ is the vertical tangent bundle \cite{bar}. Another special case of the pushforward is when $Y$ is an open subset of $X$ and $\phi$ is inclusion. In this case the lifting gerbe for $TY \oplus \phi^*(TY)$ is canonically trivial and $\phi_* : K^*(Y,\phi^*\mathcal{G}) \to K^*(X,\mathcal{G})$ is given as follows: if we describe twisted $K$-theory in terms of homotopy classes of compactly supported sections of certain bundles of Fredholm operators, then $\phi_*$ here represents trivially extending sections on $Y$ to sections on $X$.\\

There is a Mayer-Vietoris sequence in twisted $K$-theory \cite{fht},\cite{cw0}. If $U_1,U_2$ are open sets covering $X$ and $\mathcal{G}$ is graded gerbe on $X$, there is a natural long exact sequence
\begin{equation*}\xymatrix{
K^0(X,\mathcal{G}) \ar[d] & K^0(U_1,\mathcal{G}) \oplus K^0(U_2,\mathcal{G}) \ar[l] & K^0(U_1 \cap U_2 , \mathcal{G}) \ar[l] \\
K^1(U_1 \cap U_2,\mathcal{G}) \ar[r] & K^1(U_1,\mathcal{G}) \oplus K^1(U_2,\mathcal{G}) \ar[r] & K^1(X,\mathcal{G}) \ar[u]
}
\end{equation*}
where $\mathcal{G}$ defines twists on $U_1,U_2, U_1 \cap U_2$ by restriction. The Mayer-vietoris sequence is natural is the sense that pullback, pushforward and isomorphism of twists yield commuting Mayer-Vietoris sequences.\\

The last property of twisted $K$-theory we need concerns the twisted Chern character. Thus we suppose that $\mathcal{G}$ is a graded gerbe on $X$ and we choose a connection and curving for $\mathcal{G}$. Letting $\xi \in H^1(X,\mathbb{Z}_2)$ denote the grading class and $H \in \Omega^3(X)$ the curvature, the twisted Chern character is a map $Ch_{\mathcal{G}} : K^*(X,\mathcal{G}) \to H^*(X,(\xi,H))$. When $X$ is non-compact one needs to use compactly supported twisted twisted cohomology. For simplicity, we restrict attention to the case that $X$ is compact. The twisted Chern character in the case of $K$-theory twisted by ungraded gerbes is defined in \cite{bcmms}, \cite{ms}. We showed in \cite{bar} that the definition can easily be extended to graded bundle gerbes. When $X$ is compact, the twisted Chern character yields an isomorphism $Ch_{\mathcal{G}} : K^*(X,\mathcal{G}) \otimes \mathbb{R} \to H^*(X,(\xi,h))$ \cite{fht2}.

In Section \ref{sstdt} we explained that every lifting gerbe $L(V)$ admits a canonical flat connection. Using this one can give a version of the Riemann-Roch formula in twisted $K$-theory:
\begin{theorem}[Riemann-Roch formula for graded bundle gerbes]\label{rrgbg} For any $X,Y$, map $f: X \to Y$ and graded gerbe $\mathcal{G}$ with connection and curving, we have for $a \in K^*(X,L(f)\otimes f^*\mathcal{G})$:
\begin{equation*}
Ch_{\mathcal{G}}( f_* a ) \hat{A}(Y) = f_* ( Ch_{L(f)\otimes f^*\mathcal{G}} (a) \hat{A}(X) ).
\end{equation*}
\end{theorem}

Note that for the above equation to make sense we first have to define a pushforward operation in compactly supported, twisted cohomology. However we do not need such generality. We need only the case where $f : X \to Y$ is a fiber bundle with compact fibers. In this case the pushforward can simply be defined as a fiber integration $ \int_{X/Y} : \Omega^*(X,f^*(H)) \to \Omega^{*-n}(Y,H)$. A proof of this formula for ungraded bundle gerbes and oriented spaces is given in \cite{cw}, \cite{cmw}. In \cite{bar} we showed how these assumptions could be dropped, giving the general form above.


\section{A \v{C}ech approach to the Leray spectral sequence}\label{cechapp}

Given paracompact spaces $Z,W$, a map $f : Z \to W$ and a sheaf $\mathcal{F}$ on $Z$ we have the Leray spectral sequence $( E_r^{p,q}(f,\mathcal{F}), d_r)$ with $E_2^{p,q}(f,\mathcal{F}) = H^p(W , R^q f_* \mathcal{F})$ which abuts to the sheaf cohomology $H^*(Z,\mathcal{F})$. We give a general approach to determining the differential $d_2$ in terms of \v{C}ech representatives and use this to prove Proposition \ref{obstruction}.\\

Let $\mathcal{U} = \{ U_i \}_{i \in A}$ be an open cover on $W$ and $\mathcal{V} = \{ V_j \}_{j \in B}$ an open cover on $Z$. Let us take the indexing sets $A,B$ to be strictly totally ordered. As usual for $i_0, i_1 , \dots , i_p \in A$ we let $U_{i_0 i_1 \dots i_p}$ denote the multiple intersection and similarly for $\mathcal{V}$. We define a double complex $(C^{p,q}(\mathcal{U},\mathcal{V},\mathcal{F}), \delta_{\mathcal{U}},\delta_{\mathcal{V}})$ as follows. We set
\begin{equation*}
C^{p,q}(\mathcal{U},\mathcal{V},\mathcal{F}) = \prod_{\substack{ i_0 < i_1 < \cdots < i_p \\ j_0 < j_1 < \cdots < j_q \\ f^{-1}(U_{i_0 \dots i_p) \cap V_{j_0 \dots j_q} \neq \emptyset}}} \mathcal{F}( f^{-1}(U_{i_0 i_1 \dots i_p}) \cap V_{j_0 j_1 \dots j_q}).
\end{equation*}
So an element of $C^{p,q}(\mathcal{U},\mathcal{V},\mathcal{F})$ is a collection $\{ s_{i_0 i_1 \dots i_p , j_0 j_1 \dots j_q} \}$ of sections of $\mathcal{F}$ defined on the non-empty intersections $f^{-1}(U_{i_0 i_1 \dots i_p}) \cap V_{j_0 j_1 \dots j_q}$. Note that we extend the notation $s_{i_0 \dots i_p, j_0 \dots j_q}$ to unordered indices by skew-symmetry in $i_0, \dots , i_p$ and $j_0, \dots , j_q$. Next we define the differentials $\delta_{\mathcal{U}} : C^{p,q}(\mathcal{U},\mathcal{V},\mathcal{F}) \to C^{p+1,q}(\mathcal{U},\mathcal{V},\mathcal{F})$ and $\delta_{\mathcal{V}} : C^{p,q}(\mathcal{U},\mathcal{V},\mathcal{F}) \to C^{p,q+1}(\mathcal{U},\mathcal{V},\mathcal{F})$ to be the \v{C}ech differentials in the horizontal and vertical directions:
\begin{eqnarray*}
(\delta_{\mathcal{U}} s)_{i_0 \dots i_{p+1} , j_0 \dots j_q} &=& \sum_{k=0}^{p+1} (-1)^k s_{i_0 \dots \hat{i}_k \dots i_{p+1}, j_0 \dots j_q}, \\
(\delta_{\mathcal{V}} s)_{i_0 \dots i_{p} , j_0 \dots j_{q+1}} &=& \sum_{k=0}^{q+1} (-1)^{p+k+1} s_{i_0 \dots i_p, j_0 \dots \hat{j}_k \dots j_{q+1}},
\end{eqnarray*}
where we have suppressed the restriction notation for sections of $\mathcal{F}$ and as usual $\hat{i}_k,\hat{j}_k$ denote omission of those indices. Clearly we have $\delta^2_{\mathcal{U}} = 0$, $\delta^2_{\mathcal{V}} = 0$, $\delta_{\mathcal{U}} \delta_{\mathcal{V}} + \delta_{\mathcal{V}} \delta_{\mathcal{U}} = 0$. Let $(C^*(\mathcal{U},\mathcal{V},\mathcal{F}),\delta_{\mathcal{U}}+\delta_{\mathcal{V}})$ be the associated single term complex and let $H^n(C^*(\mathcal{U},\mathcal{V},\mathcal{F}))$ denote the cohomology of this complex.

Using the filtration of the single term complex $C^*(\mathcal{U},\mathcal{V},\mathcal{F})$ by $p$-degree we get a spectral sequence $E_r^{p,q}(\mathcal{U},\mathcal{V},\mathcal{F})$ and a filtration $F^{p,n}(\mathcal{U},\mathcal{V},\mathcal{F})$ of the cohomology groups $H^n(C^*(\mathcal{U},\mathcal{V},\mathcal{F}))$. We would like to compare this spectral sequence to the Leray spectral sequence. First let us find a candidate for what a morphism between $E_2$-stages should look like. The $E_1$-stage for $C^*(\mathcal{U},\mathcal{V},\mathcal{F})$ is obtained by taking $\delta_{\mathcal{V}}$-cohomology. Thus
\begin{equation*}
E_1^{p,q}(\mathcal{U},\mathcal{V},\mathcal{F}) = \prod_{I = \{ i_0 \dots i_p \} } H^q( \mathcal{V}|_{f^{-1}(U_I)} , \mathcal{F}|_{f^{-1}(U_I)}).
\end{equation*}
In the above equation $I$ denotes a subset $I = \{i_0 \dots i_p\}$ of size $p+1$, $U_I = U_{i_0 \dots i_p}$ and $\mathcal{V}$ is the restriction of the open cover $\mathcal{V}$ to the subset $f^{-1}(U_I)$. Note that the assignment to an open subset $U \subseteq W$ of the \v{C}ech cohomology group $H^q( \mathcal{V}_{f^{-1}(U)} , \mathcal{F}|_{f^{-1}(U)} )$ defines a presheaf on $W$ which we denote by $\tilde{R}^q f_* \mathcal{F}$. Clearly the natural map $H^q( \mathcal{V}_{f^{-1}(U)} , \mathcal{F}|_{f^{-1}(U)} ) \to H^q( f^{-1}(U) , \mathcal{F}|_{f^{-1}(U)} )$ together with sheafification determines a map $\tilde{R}^q f_* \mathcal{F} \to R^q f_* \mathcal{F}$, where as usual $R^q f_* \mathcal{F}$ is the sheaf associated to the presheaf $U \mapsto H^q( f^{-1}(U) , \mathcal{F}|_{f^{-1}(U)} )$. Now to pass to the $E_2$-stage we take $\delta_{\mathcal{U}}$-cohomology:
\begin{equation*}
E_2^{p,q}(\mathcal{U},\mathcal{V},\mathcal{F}) = H^p( \mathcal{U} , \tilde{R}^q f_* \mathcal{F}),
\end{equation*}
meaning the \v{C}ech cohomology of the presheaf $\tilde{R}^q f_* \mathcal{F}$ with respect to the cover $\mathcal{U}$. Clearly there are natural maps
\begin{equation}\label{mapa}
a : H^p( \mathcal{U} , \tilde{R}^q f_* \mathcal{F}) \to H^p( W , R^q f_* \mathcal{F}).
\end{equation}
Observe that $H^p( W , R^q f_* \mathcal{F}) = E_2^{p,q}(f,\mathcal{F})$, the $E_2$-stage of the Leray spectral sequence for $\mathcal{F}$.
\begin{theorem}\label{mss}
There is a morphism of spectral sequences $E_r^{p,q}(\mathcal{U},\mathcal{V},\mathcal{F}) \to E_r^{p,q}(f,\mathcal{F})$ which at the $E_2$-stage is given by the maps $a$ in (\ref{mapa}).
\end{theorem}
\begin{proof}
Let $\mathcal{F} \to C^0 \to C^1 \to \dots $ denote the \v{C}ech resolution \cite{voi1} of $\mathcal{F}$ with respect to the cover $\mathcal{V}$ and let $\mathcal{F} \to I^0 \to I^1 \to \dots$ be an injective resolution. Then according to \cite[Proposition 4.27]{voi1}, we can find a morphism $c^q : C^q \to I^q$ of complexes such that the diagram
\begin{equation*}\xymatrix{
\mathcal{F} \ar[d]^{id} \ar[r] & C^0 \ar[d]^{c^0} \\
\mathcal{F} \ar[r] & I^0
}
\end{equation*}
commutes. Applying $f_*$ we get a morphism of complexes $f_*(C^q) \to f_*(I^q)$. For each $q$ let $f_*(C^q) \to D^{0,q} \to D^{1,q} \to \dots $ be the \v{C}ech resolution of $f_*(C^q)$ with respect to the cover $\mathcal{U}$. Then $D^{p,q}$ is a double complex by the functorial nature of \v{C}ech resolutions. Let $f_*(I^q) \to J^{0,q} \to J^{1,q} \to \dots $ be a Cartan-Eilenberg resolution \cite{gm} of the complex $f_*(I^q)$. Then by \cite[Proposition 4.3]{voi2} we can find a morphism $d^{p,q} : D^{p,q} \to J^{p,q}$ of complexes such that the diagram
\begin{equation*}\xymatrix{
f_*(C^q) \ar[d]^{c^q} \ar[r] & D^{0,q} \ar[d]^{d^{0,q}} \\
f_*(I^q) \ar[r] & J^{0,q}
}
\end{equation*}
commutes. Let $\Gamma$ denote the global sections functor for sheaves. We have a morphism of complexes $\Gamma( D^{p,q}) \to \Gamma( J^{p,q})$. Let $D^*,J^*$ be the associated single complexes. We get filtrations on $\Gamma(D^*),\Gamma(J^*)$ by $p$-degree and the morphism $\Gamma(D^*) \to \Gamma(J^*)$ preserve the filtration so defines a morphism $E_r^{p,q}(D^*) \to E_r^{p,q}(J^*)$ between the associated spectral sequences. We observe that the spectral sequence associated to the filtration on $\Gamma(J^*)$ is precisely the Leray spectral sequence. Actually the spectral sequence $E_r^{p,q}(J^*)$ is only independent of the choice of resolutions $I^*,J^{**}$ from the $E_2$-stage onwards, so the Leray spectral sequence really begins with $E_2^{p,q}(J^*)$. Thus $E_r^{p,q}(J^*) = E_r^{p,q}(f,\mathcal{F})$ for $r \ge 2$.\\

A closer look at the construction of the double complex $\Gamma(D^{p,q})$ reveals that in fact $\Gamma(D^{p,q})$ is the double complex $(C^{p,q}(\mathcal{U},\mathcal{V},\mathcal{F}),\delta_{\mathcal{U}},\delta_{\mathcal{V}})$. In particular $E_r^{p,q}(D^*) = E_r^{p,q}(\mathcal{U},\mathcal{V},\mathcal{F})$. It remains only to check that the morphisms $E_2^{p,q}(D^*) \to E_2^{p,q}(J^*)$ are given by the maps $a$ in (\ref{mapa}). Starting from $\Gamma(D^{p,q}) \to \Gamma(J^{p,q})$, the morphisms at the $E_1$-stage are obtained by taking cohomology in the $q$-direction. Since $J^{p,q}$ is a Cartan-Eilenberg resolution the cocycles, coboundaries and cohomology in the $q$-direction are all injective objects and from this one sees that $E_1^{p,q}(J^*)$ is equivalently obtained by taking $q$-cohomology of $J^{p,q}$ and then applying $\Gamma$. Starting from the morphisms $D^{p,q} \to J^{p,q}$, take $q$-cohomology. Let $D_1^{p,q},J_1^{p,q}$ be the resulting cohomology objects which for fixed $q$ are complexes in $p$. So there are morphisms $D_1^{p,q} \to J_1^{p,q}$ and applying $\Gamma$ we get maps $\Gamma(D_1^{p,q}) \to \Gamma(J_1^{p,q})$. We get a commutative diagram of the form
\begin{equation*}\xymatrix{
\Gamma( D_1^{p,q} ) \ar[r] & \Gamma( J_1^{p,q} ) \ar[d]^{\simeq} \\
E_1^{p,q}(D^*) \ar[r] \ar[u] & E_1^{p,q}(J^*) 
}
\end{equation*}
where the maps $E_1^{p,q}(D^*) \to \Gamma( D_1^{p,q})$ are as follows: we have
\begin{equation*}
E_1^{p,q}(D^*) = \prod_{|I| = p+1} (\tilde{R}^q f_* \mathcal{F})( U_I ),
\end{equation*}
where as before $\tilde{R}^q f_* \mathcal{F}$ is the presheaf on $W$ such that $U \mapsto H^q( \mathcal{V}|_{f^{-1}(U)} , \mathcal{F}|_{f^{-1}(U)})$. On the other hand one sees that for fixed $q$ the complex $D_1^{p,q}$ is the \v{C}ech resolution of the sheaf $S\tilde{R}^q f_* \mathcal{F}$ associated to $\tilde{R}^q f_* \mathcal{F}$ using the open cover $\mathcal{U}$ on $W$. The natural map $\tilde{R}^q f_* \mathcal{F} \to S\tilde{R}^q f_* \mathcal{F}$ then induces the maps $E_1^{p,q}(D^*) \to \Gamma( D_1^{p,q})$. Next we observe that the maps $D_1^{p,q} \to J_1^{p,q}$ for fixed $q$ are morphisms of complexes in $p$. We have already observed that for fixed $q$, $D_1^{p,q}$ is the \v{C}ech resolution of $S\tilde{R}^q f_* \mathcal{F}$ with respect to the cover $\mathcal{U}$. On the other hand using the definition of Cartan-Eilenberg resolutions we have for fixed $q$ that $J_1^{p,q}$ is an injective resolution of $R^q f_* \mathcal{F}$. Thus for each fixed $q$, $\Gamma(D^{p,q})$ is the \v{C}ech complex for the sheaf $S\tilde{R}^q f_* \mathcal{F}$, while $\Gamma(J_1^{p,q})$ is a complex such that the cohomology is $H^p(W , R^q f_* \mathcal{F})$. From this obtain easily that the maps $E_2^{p,q}(D^*) \to E_2^{p,q}(J^*)$ are the maps in (\ref{mapa}).
\end{proof}

Our goal is to use the complex $C^{p,q}(\mathcal{U},\mathcal{V},\mathcal{F})$ for suitable choice of covers $\mathcal{U},\mathcal{V}$ to compute the $d_2$ differential in the Leray spectral sequence in terms of explicit cocycle data. Thus suppose $x \in E_2^{p,q}(f,\mathcal{F}) = H^p( W , R^q f_* \mathcal{F})$. We can find a cover $\mathcal{U} = \{ U_i \}$ such that $x$ has a \v{C}ech representative $x = [\{ x_{i_0 i_1 \dots i_p} \}]$, where $x_{i_0 i_1 \dots i_p} \in (R^q f_*\mathcal{F})(U_I)$. Since $W$ is paracompact, there is an isomorphism between \v{C}ech cohomology of $R^q f_* \mathcal{F}$ and the \v{C}ech cohomology of the presheaf $U \mapsto H^q( f^{-1}(U) , \mathcal{F})$ \cite{spa}. As a result, for every $I = \{ i_0 , i_1 , \dots , i_p \}$ we can choose an open cover $\mathcal{Z}_I = \{ Z_r \}_{r \in R_I}$ of $f^{-1}(U_I)$ such that $x_{i_0 i_1 , \dots i_p}$ has a \v{C}ech representative $[ \{ x_{i_0 i_1 , \dots i_p , r_0 , r_1 , \dots , r_q} \}]$, where $x_{i_0 i_1 , \dots i_p , r_0 , r_1 , \dots , r_q} \in \mathcal{F}( Z_{r_0 r_1 \dots r_q})$. For any $z \in Z$ we have that $z \in \pi^{-1}(U_I)$ for only finitely many $I$. Therefore we may choose an open neighborhood $z \in L_z$ such that for each $I$ with $z \in \pi^{-1}(U_I)$ there exists an $r \in R_I$ such that $L_z \subseteq Z_r$. Choosing such an $L_z$ for each $z \in Z$ we find an open cover $\mathcal{V}$ of $Z$ with the property that for all $I$, $\mathcal{V}|_{f^{-1}(U_I)}$ is a refinement of $\mathcal{Z}_I$. The point is that it now follows that for all $I$, $x_{i_0 i_1 \dots i_p}$ has a representative in $H^p( \mathcal{V}|_{f^{-1}(U_I)} , \mathcal{F}|_{f^{-1}(U_I)})$. We have thus shown that:
\begin{proposition}
Given any $x \in E_2^{p,q}(f,\mathcal{F}) = H^p(W, R^q f_* \mathcal{F})$ we can find a cover $\mathcal{U}$ of $W$ and cover $\mathcal{V}$ of $Z$ such that $x$ lies in the image of the morphism $a : E_2^{p,q}(\mathcal{U},\mathcal{V},\mathcal{F}) \to E_2^{p,q}(f,\mathcal{F})$ of Theorem \ref{mss}.
\end{proposition}

Now we specialize to the case relevant to Proposition \ref{obstruction}. Thus let $f : Z \to W$ a locally trivial torus bundle and $\mathcal{F} = \mathcal{C}_U$. Suppose $x \in E_2^{0,1}(f,\mathcal{F})$. Then we have just shown that we can find covers $\mathcal{U},\mathcal{V}$ of $W,Z$ such that $x$ has a representative $x' \in E_2^{0,1}(\mathcal{U},\mathcal{V},\mathcal{F})$. We use this to this to determine $d_2 x \in E_2^{2,0}(f,\mathcal{F})$ by first determining $d_2 x' \in E_2^{2,0}(\mathcal{U},\mathcal{V},\mathcal{F})$. To do this let $\tilde{x} \in C^{0,1}(\mathcal{U},\mathcal{V},\mathcal{F})$ be a representative for $x'$. So in particular $\delta_{\mathcal{V}} \tilde{x} = 0$ and $\delta_{\mathcal{U}} \tilde{x} + \delta_{\mathcal{V}} \tilde{y} = 0$ for some $\tilde{y} \in C^{1,0}(\mathcal{U},\mathcal{V},\mathcal{F})$. We then have that $\delta_{\mathcal{U}} \tilde{y}$ is a representative for $d_2 x'$. Explicitly we have ${\rm U}(1)$-valued functions $\tilde{x}_{i_0 , j_0 j_1}$, $\tilde{y}_{i_0 i_1,j_0}$ such that
\begin{eqnarray}
\tilde{x}_{i_0 , j_0 j_1} \tilde{x}_{i_0 , j_1 j_2} \tilde{x}_{i_0 , j_2 j_1} &=& 1, \label{lbc} \\
\tilde{x}_{i_1 , j_0 j_1}\tilde{x}^{-1}_{i_0 , j_0 j_1} &=& \tilde{y}_{i_0 i_1 , j_1} \tilde{y}^{-1}_{i_0 i_1 , j_0}. \label{isocond}
\end{eqnarray}
Equation (\ref{lbc}) says that for each $i_0$, $\{ \tilde{x}_{i_0 , j_0 j_1} \}$ are transition functions for a line bundle $L_i \to f^{-1}(W)$. Then Equation (\ref{isocond}) says that the $\{ \tilde{y}_{i_0 i_i , j_0} \}$ define line bundle isomorphisms $\phi_{i_0 i_1} : L_{i_1} \to L_{i_0}$ on the double intersections $f^{-1}(U_{i_0 i_1})$. We see immediately that $\phi_{i_0 i_1} \phi_{i_1 i_2} \phi_{i_2 i_0} = g_{i_0 i_1 i_2}$ is given by $\delta_{\mathcal{U}} \tilde{y}$, which represents $d_2 x'$. From this we have proven Proposition \ref{obstruction}.


\bibliographystyle{amsplain}

\end{document}